\documentclass[a4paper,10pt]{article}

\usepackage[latin1]{inputenc}   
\usepackage{amsmath,amsthm}
\usepackage[english]{babel}
\usepackage[dvips]{graphicx}
\usepackage{amsfonts,amssymb}
\usepackage{geometry}
\usepackage{hyperref}
\usepackage{stmaryrd}
\usepackage{fancyhdr}
\usepackage{url}
\usepackage{dsfont}

\newtheorem*{conj*}{Conjecture}

\newtheorem*{thm*}{Theorem}

\newtheorem{prop}{Proposition}[section]

\newtheorem{LM}{Lemma}[section]

\newtheorem{thm}{Theorem}[section]

\newtheorem{df}{Definition}[section]
\newtheorem{cor}{Corollary}[section]

\newtheoremstyle{pourlesremarques}{\topsep}{\topsep}{\normalfont}{}{\bfseries}{.}{ }{}
\theoremstyle{pourlesremarques}
\newtheorem{rem}{Remark}[section]
\newtheorem*{rem*}{Remark}
\newtheoremstyle{pourlesexemples}{\topsep}{\topsep}{\normalfont}{}{\bfseries}{.}{ }{}
\theoremstyle{pourlesexemples}

\renewcommand{\o}{\mathfrak{O}}
\newcommand{\p}{\mathfrak{P}}
\newcommand{\w}{\varpi}
\renewcommand{\d}{\delta}
\newcommand{\e}{\epsilon}

\renewcommand{\l}{\lambda}
\newcommand{\C}{\mathbb{C}}

\newcommand{\N}{\mathbb{N}}
\newcommand{\Z}{\mathbb{Z}}

\newcommand{\1}{\mathbf{1}}
\renewcommand{\O}{\mathfrak{O}}
\newcommand{\D}{\Delta}
\newcommand{\sm}{\mathcal{C}^\infty}

\title {\textbf{On the local Bump-Friedberg L-function}}
\author{Nadir MATRINGE\footnote{Nadir Matringe, Universit\'e de Poitiers, Laboratoire de Math\'ematiques et Applications,
T\'el\'eport 2 - BP 30179, Boulevard Marie et Pierre Curie, 86962, Futuroscope Chasseneuil Cedex. Email: Nadir.Matringe@math.univ-poitiers.fr}}

\begin{document}
\maketitle

 \begin{abstract}
 Let $F$ be a $p$-adic field. If $\pi$ be an irreducible representation of $GL(n,F)$, Bump and Friedberg associated to $\pi$ 
 an Euler fator $L(\pi,BF,s_1,s_2)$ in \cite{BF}, that should be equal to $L(\phi(\pi),s_1)L(\phi(\pi),\Lambda^2,s_2)$, where $\phi(\pi)$ 
 is the Langlands' parameter of $\pi$. The main result of this paper is to show that this equality is true when $(s_1,s_2)=(s+1/2,2s)$, for 
$s$ in $\C$. 
 To prove this, we classify in terms of distinguished discrete series, generic representations of 
$GL(n,F)$ which are $\chi_\alpha$-distinguished by the Levi subgroup 
 $GL([(n+1)/2],\!F)\! \times\! GL([n/2],\!F)$, for $\chi_\alpha(g_1,g_2)=\alpha(det(g_1)/det(g_2))$, where $\alpha$ is a character 
of $F^*$ of real part between $-1/2$ and $1/2$.
 We then adapt 
 the technique of \cite{CP} to reduce the proof of the equality to the case of discrete series. The equality for discrete 
series is a consequence of the relation between linear periods and Shalika periods for discrete series, and the main result of \cite{KR}.
\end{abstract}

% \tableofcontents %% Just for papers exceeding 50 pages.

\section{Introduction}\label{sec:intro}

This paper owes much to \cite{CP}, and follows the path of the adaptation of \cite{CP} used in \cite{M2009} and \cite{M2011} to study the local 
Asai $L$-factor. We sometimes clarify some arguments of [loc. cit.]. In our case, we have to deal with several technical complications, 
especially in the study 
of distinguished generic representations.\\

Section \ref{prelim} is devoted to the introduction of the basic objects. We recall a few facts about representations of 
Whittaker type and Langlands' type. We 
continue with a reminder about Bernstein-Zelevinsky derivatives, and the classification of generic representations of $GL(n,F)$ in 
terms of discrete series. We then recall results of \cite{CP} concerning the relation between derivatives and restriction of Whittaker 
functions to smaller linear groups. We end the section with a result of Bernstein, which is used to extend meromorphically invariant linear 
forms on the space of induced representations of $GL(n,F)$.\\

Let $M$ be the maximal Levi subgroup $GL([(n+1)/2],F)\times GL([n/2],F)$ of $GL(n,F)$. We recall that if $\chi$ is a character of 
$GL([(n+1)/2],F)\times GL([n/2],F)$, a representation $\pi$ of $GL(n,F)$ is said to be $\chi$-distinguished if $Hom_M(\pi,\chi)$ is nonzero. 
In Section \ref{classification}, we classify $\chi$-distinguished generic representations of $GL(n,F)$ in terms of $\chi$-distinguished 
discrete series (Theorem \ref{distgen}), for $\chi$ a character of the form $\chi_\alpha(g_1,g_2)=\alpha(det(g_1)/det(g_2))$, where $\alpha$ 
is a character of $F^*$ of real part in $[-1/2,1/2]$. Results of \cite{M2012.2} are used. When $n$ is even, and $\chi$ is trivial, 
our classification implies that $\pi$ is 
distinguished if and only if its Langlands' parameter fixes a non degenerate symplectic form (i.e. factors 
through the $L$-group of $SO(2n+1,F)$).\\

In section \ref{Lfactors}, we recall some classical facts from \cite{JPS} about $L$-functions of pairs of irreducible
 representations of $GL(n,F)$, and some other facts from \cite{CP} about their exceptional poles. We then introduce the Rankin-Selberg 
integrals which will be of interest to us, and which define a 
local factor $L^{lin}(\pi,\chi_{\alpha},s)$ for any representation $\pi$ of Whittaker type of $GL(n,F)$. This factor satisfies a functional 
equation, and in general, it has an exceptional pole at zero if and only if $\pi$ is $\chi_\alpha^{-1}$-distinguished. We then prove rationality 
of the Rankin-Selberg integrals associated to a representation $\pi$ of Whittaker type, when $\pi$ is twisted by an unramified character 
of the Levi subgroup from which it is induced 
(Theorem \ref{rationality}).\\

In Section \ref{egalitegalois}, using our classification of $\chi_\alpha$-distinguished generic representations of the group $GL(n,F)$, we prove
 the inductivity relation of the factor $L^{lin}(\pi,\chi_{\alpha},s)$ for representations of Langlands' type of $GL(n,F)$, when 
$\alpha$ is of real part between $-1/2$ and $0$ (Theorem \ref{indLanglands }). Then, using the relation between Shalika periods and linear periods for 
unitary discrete series $\D$ of $GL(n,F)$, we prove the equality 
of $L^{lin}(\D,s)$ ($\alpha$ is now trivial) and $L(\phi(\D),s+1/2)L(\Lambda^2(\phi(\D)),2s)$ for $\phi(\D)$ the Langlands' 
parameter of $\D$. The equality for irreducible representations follows (Corollary \ref{egalfinal}).

\section{Preliminaries}\label{prelim}

\subsection{Notations and basic facts}
We fix until the end of this paper a nonarchimedean local field $F$ of characteristic zero. We denote by $\O$ the ring of integers of 
$F$, by $\w$ a uniforiser of $F$, and by $q$ the cardinality of $\O/\w\O$.
We denote by $|.|$ the absolute value on $F$ normalised by the condition $|\w|=q^{-1}$, and by $v$ the valuation of $F$ such that $v(\w)=1$. If $\alpha$ is a character of $F$, and $|.|_{\C}$ is 
the positive character $z\mapsto (z\bar{z})^{1/2}$ of $\C^*$, the character $|\alpha|_{\C}$ is of the from $|.|^r$ for a real number $r$ 
that we call the real part of $\alpha$, with notation $r=Re(\alpha)$.\\
Let $n$ be a positive integer, we will denote by $G_n$ the group $GL(n,F)$ of units of the algebra $\mathcal{M}(n,F)$, and by $Z_n$ its center. We 
denote by $\eta_n$ the row vector $(0,\dots,0,1)$ of $F^n$, so if $g$ belongs to $G_n$, $\eta_n g$ is the bottom row of $g$.
 We denote by $K_n$ maximal compact subgroup 
$GL(n,\o)$ of $G_n$, it admits a filtration by the subgroups $K_{n,r}=I_n+\w^r\mathcal{M}(n,\o)$ for $r\geq 1$. For $g\in G_n$, we 
denote by $det(g)$ the determinant of $g$, and sometimes, we denote by abuse of notation 
$|g|$ for $|det(g)|$. If $g_1$ and $g_2$ are two elements of $G_n$, we denote by $g_1(g_2)$ the matrix $g_1^{-1} g_2 g_1$.\\
We call partition of a positive integer $n$, a family $\bar{n}=(n_1,\dots,n_t)$ of positive integers 
(for a certain $t$ in $\mathbb{N}-\left\lbrace 0\right\rbrace $), 
such that the sum $n_1+\dots+n_t$ is equal to $n$. To such a partition, we associate a subgroup of $G_n$ denoted by $P_{\bar{n}}$, 
given by matrices of the form $$\left (\begin{array}{cccccccc}
g_1 & \star & \star & \star & \star \\
    & g_2 & \star &  \star & \star \\
    &  & \ddots & \star & \star \\  
    &  &     & g_{t-1} &  \star \\
    &  &     &   & g_t
 \end{array}\right),$$ with $g_i$ in $G_{n_i}$ for $i$ between $1$ and $t$. We call it the standard parabolic subgroup 
associated with the partition $\bar{n}$. We denote by $N_{\bar{n}}$ its unipotent radical subgroup, 
given by the matrices $$\left (\begin{array}{cccccccc}
I_{n_1} & \star & \star\\
     &  \ddots & \star\\  
    &  & I_{n_t}
 \end{array}\right)$$ and by $M_{\bar{n}}$ its standard Levi subgroup given by the matrices $diag(g_1,\dots,g_t)$ with the $g_i$'s in 
$G_{n_i}$.
 The group $P_{\bar{n}}$ identifies with the semidirect product $N_{\bar{n}}\rtimes M_{\bar{n}}$, we will denote $N_{(n-1,1)}$ by $U_n$.
 We will denote by $P_{\bar{n}}^-$ the opposite group of $P_{\bar{n}}$, which is its image under the transpose. 
If we consider $G_{n-1}$ as a subgroup of $G_n$ via the embedding $g\mapsto diag(g,1)$ we denote by $P_n$ the mirabolic subgroup 
$G_{n-1}U_n$ of $G_n$. We denote by $A_n$ the diagonal torus $M_{(1,\dots,1)}$ of $G_n$, by $N_n$ the unipotent radical $N_{(1,\dots,1)}$, and by 
$B_n$ the standard Borel $A_n N_n$. It will sometimes be useful to paremetrise $A_n$ with simple roots, i.e. 
to write an element $t=diag(t_1,\dots,t_n)$ of $A_n$, as $t=z_1\dots z_n$, where $z_n=t_nI_n$, and $z_i=diag((t_i/t_{i+1})I_i,I_{n-i})$ 
belongs to the center of $G_i$ embedded in $G_n$.\\
We denote by $w_n$ the element of the symetric group $\mathfrak{S}_n$ naturally embedded in $G_n$, defined by 
$$\begin{pmatrix} 1 & 2 & \dots & m-1 & m & m+1 & m+2 & \dots & 2m-1 & 2m \\
 1 & 3 & \dots & 2m-3 & 2m-1 & 2 & 4 & \dots & 2m-2 & 2m  
 \end{pmatrix}$$ when $n=2m$ is even, and by 
$$\begin{pmatrix} 1 & 2 & \dots & m-1 & m & m+1 & m+2 &\dots & 2m & 2m+1 \\
 1 & 3 & \dots & 2m-3 & 2m-1 & 2m+1 & 2 & \dots & 2m-2 & 2m  
 \end{pmatrix}$$ when $n=2m+1$ is odd. We denote by $W'_n$ the semi-direct product $A_n\mathfrak{S}_n$, which is the normaliser of $A_n$ in $G_n$. When $n=2m+1$, we denote by $H_n$ the group $w_n(M_{(m+1,m)})$, by $h(g_1,g_2)$ the matrix 
$w_n(diag(g_1,g_2))$ of $H_n$ (with $g_1$ in $G_{m+1}$ and $g_2$ in $G_m$). 
When $n=2m$, we denote by $H_n$ the group $w_n(M_{(m,m)})$, by $h(g_1,g_2)$ the matrix 
$w_n(diag(g_1,g_2))$ of $H_n$ (with $g_1$ and $g_2$ in $G_m$). 
In both cases, if we denote by $\e_n$ the matrix $diag(1,-1,1,-1,\dots)$ of $G_n$, the group $H_n$ is the subgroup of $G_n$ fixed by the involution 
$g\mapsto \e_n g\e_n$. If $\alpha$ is a character 
of $F^*$, we denote by $\chi_\alpha$ the character $\chi_\alpha:h(g_1,g_2)\mapsto \alpha(det(g_1)/det(g_2))$ of $H_n$, we denote by $\d_n$ the character 
$\d_n=\chi_{|.|}: h(g_1,g_2)\mapsto |g_1|/|g_2|$ of $H_n$, 
and we denote by $\chi_n$ (resp. $\mu_n$) the character of $H_n$ equal to $\d_n$ when $n$ is odd (resp. even), 
and trivial when $n$ is even (resp. odd). Hence ${\chi_n}_{|H_{n-1}}=\mu_{n-1}$, and ${\mu_n}_{|H_{n-1}}=\chi_{n-1}$. 
 If $C$ is a subset of $G_n$, we sometimes denote by $C^\sigma$ the set $C\cap H_n$.\\ 

Let $G$ be an $l$-group (locally compact totally disconnected), we denote by $d_{G} g$ or simply $dg$ if the context is clear, a 
right Haar measure on $G$. 
For $x$ in $G$, we denote by $\delta_G (x)$ the positive number defined by the relation $d_g (xg)= \delta_G^{-1} (x)d_g (g)$. The
modulus character $\delta_G$ defines a morphism from $G$ into $\mathbb{R}_{>0}$. We denote by $\Delta_G$ (which we also call modulus character) the morphism from $G$ into $\mathbb{R}_{>0}$ defined by $x \mapsto \delta_G(x^{-1})$. If $G$ is an $l$-group, and $H$ a subgroup of $G$, a representation $(\pi,V)$ of $G$ is said to be smooth 
if for any vector $v$ of the complex vector space $V$, there is an open subgroup $U_v$ of $G$ stabilizing $v$ through $\pi$. We say 
that a smooth representation 
$(\pi,V)$ is admissible if, moreover, for each open subgroup $U$ of $G$, the space $V^{U}$ of $U$-fixed vectors in $V$ is of finite dimension. 
In this case, $V$ 
is of countable dimension. It is known that smooth irreducible representations of $G_n$ are always admissible.
The category of smooth representations of $G$ is denoted by $Alg(G)$, we will often identify two ismorphic objects of this category. If $(\pi,V)$ 
is a smooth representation of $G$,
 we denote by $\pi^{\vee}$ its dual representation in the smooth dual space $V^\vee$ of $V$. We will only 
consider smooth complex representations of $l$-groups. If $\pi$ has a central character, we denote it by $c_{\pi}$. If $G=G_n$, we write 
by abuse of notation $c_\pi$ for the character of $F^*$ defined by $c_\pi(t)=c_\pi(tI_n)$.

\begin{df}
Let $G$ be an $l$-group, $H$ a closed subgroup of $G$, and $(\pi,V)$ a representation of $G$. 
If $\chi$ is a character of $H$, we say that the representation $\pi$ is $(H,\chi)$-distinguished, 
if it admits on its space a nonzero linear form $L$, verifying $L(\pi(h)v)=\chi(h)L(v)$ for all $v$ in $V$ and $h$ in $H$. 
If $\chi=1$, we say $H$-distinguished. We omit ``$H$-'' if the context is clear.
\end{df}
We will be interested in the case $G=G_n$, and $H=H_n$. In this case, it is shown in \cite{JR} that if an irreducible representation $\pi$ 
of $G_n$ is $H_n$-distinguished, the dimension of $Hom_{H_n}(\pi,\1_H)$ is one, and $\pi$ is self-dual, i.e. isomorphic to $\pi^\vee$.\\

If $\rho$ is a complex representation of $H$ in $V_{\rho}$, we denote by $\sm(H \backslash G, \rho, V_{\rho})$ the space of smooth functions $f$ from $G$ to $V_{\rho}$, fixed under the action by right translation of some compact open subgroup $U_f$ of $G$, and which verify $f(hg)=\rho(h)
 f(g)$ for $h \in H$, and $g \in G$. We denote by $\sm_c(H \backslash G, \rho, V_{\rho})$ the subspace of functions with support compact modulo 
$H$ of $\sm(H \backslash G, \rho, V_{\rho})$.\\ 
We denote by $Ind _H^G (\rho)$ the representation by right translation of $G$ in $\sm(H \backslash G, \rho,
 V_{\rho})$ and by $ind _H^G (\rho)$ the representation by right translation of $G$ in $\sm_c(H \backslash G,  \rho,
 V_{\rho})$. If $H=\{1_G\}$, and $\rho$ is the trivial representation acting on $\C$, we write $\sm_c(H \backslash G,  \rho,
 V_{\rho})=\sm_c(G)$. We write $\sm_{c,0}(F^n)$ for the subspace of $\sm_{c}(F^n)$ consisting of functions vanishing at zero. We denote by ${Ind'} _H^G (\rho)$ the normalized induced representation $Ind _H^G ((\Delta _G /\Delta _H)^{1/2} \rho)$ and by ${ind'} _H^G (\rho)$ 
the normalized induced representation $ind _H^G ((\Delta _G /\Delta _H)^{1/2} \rho)$.\\  
Let $n$ be a positive integer, and $\bar{n}=(n_1,\dots,n_t)$ be a partition of $n$, and suppose that we have a representation $(\rho_i, V_i)$ of $G_{n_i}$ for each $i$ between $1$ and $t$. 
Let $\rho$ be the extension to $P_{\bar{n}}$ of the natural representation 
$\rho_1 \otimes \dots \otimes \rho_t$ of $M_{\bar{n}}\simeq G_{n_1} \times \dots \times G_{n_t}$, trivial on $N_{\bar{n}}$. 
We denote by $\rho_1 \times \dots \times \rho_t$ the representation ${ind'} _{P_{\bar{n}}}^{G_n} (\rho)= {Ind'} _{P_{\bar{n}}}^{G_n} (\rho)$. Parabolic induction preserves admissibility.\\
According to Theorem 9.3 of \cite{Z}, if $\rho$ is a cuspidal representation of $G_r$, and $k\geq 1$, the representation 
$\rho\times \dots \times \rho|.|^{k-1}$ of $G_n$ with $n=kr$, has a unique irreducible quotient, that we denote by 
$\D=[\rho,\dots,|.|^{k-1}\rho]$. It is a discrete series, i.e. some twist of $\D$ by a character embeds in $L^2(G_n/Z_n)$. 
Every discrete series representation of $G_n$ 
arises this way, for a unique triple $(k,r,\rho)$, with $n=kr$. We can also call 
such a representation of $G_n$ a segment.\\

We recall that the map $s\mapsto q^{-s}$ induces a group isomorphism between $\C/\frac{2i\pi}{Ln(q)}\Z$ and $\C^*$, we denote 
$\C/\frac{2i\pi}{Ln(q)}\Z$ by $\mathcal{D}$. Thus $\mathcal{D}$ is a regular algebraic variety whose ring of regular functions identifies with
 $\C[q^s,q^{-s}]$. More generally, for $t\in \N^*$, we denote by $u$ the variable in $\mathcal{D}^t$, we write 
 $q^{u}=(q^{u_1},\dots,q^{u_t})$ and $q^{-u}=(q^{-u_1},\dots,q^{-u_t})$, hence the ring of regular functions on $\mathcal{D}^t$ identifies with 
 $\C[q^{\pm u}]=\C[q^{\pm u_1},\dots,q^{\pm u_t}]$.\\
 
Finally, we denote by $W_F$ the Weil group of $F$, by $I_F$ its inertia subgroup, and by $Fr$ the Frobenius element of $W_F$ 
(see \cite{BH}, Chapter 29). We set $W'_F=W_F\times SL(2,\C)$, the Weil-Deligne group of $F$. We recall that 
according to a theorem by Harris-Taylor and Henniart, there is a natural bijection from the set of isomorphism classes of 
irreducible representations of $G_n$ to the the set of isomorphism classes of $n$-dimansional semi-simple representations
 of $W'_F$, called the Langlands' correspondance, which we denote by $\phi$. 
We denote by $Sp(k)$ the (up to isomorphism) unique algebraic $k$-dimensional irreducible representation of $SL(2,\C)$. 
By a semi-simple representations of $W'_F$, we mean a direct sum of representations of the form $\tau \times Sp(k)$, where $(\tau,V)$ 
is a smooth irreducible representation of $W_F$. We define the Artin $L$-factor of $\tau$ to be 
$$L(\tau,s)=1/det[(I_d-q^{-s}\tau(Fr))_{|V^{I_F}}],$$ where $V^{I_F}$ is the subspace of $V$ fixed by $I_F$. 
We the define $L(\tau \times Sp(k),s)$ to be $L(\tau,s+k/2)$. We extend the definition to all semi-simple representations 
of $W'_F$ by setting $$L(\tau\oplus \tau',s)=L(\tau,s)L(\tau',s).$$
 Then, by definition, the exterior-square $L$-function
 $L(\tau,\wedge^2,s)$ of a finite dimensional 
semi-simple representation $\tau$ of $W'_F$ is $L(\wedge^2(\tau),s)$, where $\wedge^2(\tau)$ is the 
exterior-square of $\tau$. Similarly, the factor $L(\tau,Sym^2,s)$ is by definition $L(Sym^2(\tau),s)$, where $Sym^2(\tau)$ is the 
symmetric-square of $\tau$.\\
 
 In this paper, we will often say "representation" instead of "smooth admissible representation" and we will often say that two representations 
 are equal when they are isomorphic.

\subsection{Representations of Whittaker type and their derivatives}

We recall properties of representations of Whittaker type, for which we can define Rankin-Selberg integrals.

\begin{df}\label{whittype}
Let $\pi$ be a representation of $G_n$, such that $\pi$ is a product of discrete series $\D_1\times\dots\times \D_t$ of smaller linear groups, 
we say that $\pi$ is of Whittaker type.\\
If the discrete series $\D_i$ are ordered such that $Re(c_{\D_i})\geq Re(c_{\D_{i+1}})$, we say that $\pi$ is (induced) of Langlands' type.
If $\pi$ is moreover irreducible, we say that $\pi$ is generic. 
\end{df}

Let $\D=[\rho,\dots,\rho|.|^k]$ and $\D'=[\rho',\dots,\rho'|.|^{k'}]$ 
be two segments of $G_n$ and $G_{n'}$ respectively, we say that $\D$ and $\D'$ are linked if either 
$\rho'=\rho|.|^i$, for $i \in \{1,\dots,k+1\}$, and 
$i+k'\geq k+1$, or $\rho=\rho'|.|^{i'}$, for $i' \in \{1,\dots,k'+1\}$, and 
$i'+k\geq k'+1$. We now fix untill the end of this work, a nontrivial character $\theta$ of $(F,+)$, which defines by the formula 
$\theta(n)=\sum_{i=1}^n n_{i,i+1}$ a character still denoted $\theta$ 
of $N_n$.  We have the following proposition, which is a 
consequence of Theorem 7 of \cite{R}, and Theorem 9.7 of \cite{Z}. 

\begin{prop}\label{fourretout}
Let $\pi$ be a representation of Whittaker type, then $Hom_{N_n}(\pi,\theta)$ is of dimension $1$. In particular,
 according to Frobenius reciprocity law, 
the space of intertwining operators $Hom_{G_n}(\pi,Ind_{N_n}^{G_n}(\theta))$ is of dimension $1$, and the image of the 
(unique up to scaling) intertwining operator from $\pi$ to $Ind_{N_n}^{G_n}(\theta))$ is called the Whittaker model of $\pi$. We
 denote it by $W(\pi,\theta)$. The representation 
$\pi$ is generic if and only if the discrete series $\D_i$ commute, or equivalently if the corresponding segments are unlinked. 
\end{prop}

If $\pi=\D_1\times \dots \times\D_t$ is a representation of Whittaker type of $G_n$, and $w$ is the antidiagonal matrix of $G_n$ with only ones an the second diagonal, then $\tilde{\pi}:g\mapsto \pi(^t\!g^{-1})$ is isomorphic to $\pi^\vee$ when $\pi$ is generic, and of Whittaker type in general, isomorphic to 
$\D_t^\vee\times \dots \times \D_1^\vee$. Moreover for $W$ in $W(\pi,\theta)$, then $\tilde{W}:g\mapsto W(w^t\!g^{-1})$ belongs to $W(\tilde{\pi},\theta^{-1})$.\\ 
We recall the following definition.

\begin{df}\label{defder}
We denote by $\Phi^+\!:\!Alg(P_n)\!\rightarrow \!Alg(P_{n+1})$, $\Phi^-\!:\!Alg(P_n)\!\rightarrow \!Alg(P_{n-1})$, $\Psi^+:Alg(G_n)\rightarrow Alg(P_{n+1})$, 
$\Psi^-:Alg(P_n)\rightarrow Alg(G_{n-1})$ the functors defined in 3.2. of \cite{BZ}. 
If $\pi$ is a smooth $P_n$-module, we denote by $\pi^{(k)}$ the $G_{n-k}$-module $(\Phi^-)^{k-1}\Psi^-(\pi)$.
\end{df}

If $\pi$ is a representation of $G_n$, we will sometimes consider $\pi_{|P_n}$ without refering to $P_n$, for example we will write 
$\pi^{(k)}$ for $(\pi_{|P_n})^{(k)}$. It is a consequence of b) and c) of Proposition 3.2. of \cite{BZ}, that when $\tau$ is a $P_n$-module, 
there is an vecotr space isomorphism $Hom_{N_n}(\tau,\theta)\simeq \tau^{(n)}$.
From 3.2. and 3.5. of \cite{BZ} we have:

\begin{prop}\label{BZfiltration}
The functors $\Psi^-\Psi^+$ and $\Phi^-\Phi^+$ are the identity of $Alg(G_n)$ and $Alg(P_n)$ respectively. For any $P_n$-module $\tau$,
 and $k$ in $\{1,\dots,n\}$, the $P_n$-module $$\tau_k=(\Phi^+)^{k-1}(\Phi^-)^{k-1}(\tau)$$ is submodule of 
$\tau$, and one has $\tau_{k+1}\subset \tau_k$, with $\tau_k/\tau_{k+1}\simeq (\Phi^+)^{k-1}\Psi^+(\tau^{(k)})$. 
\end{prop}

For representations of Whittaker type, the previous filtration is determined by Lemma 4.5. of \cite{BZ}, and Proposition 
9.6. of \cite{Z}.

\begin{prop}\label{derwhittaker}
Let $\D=[\rho,\dots,|.|^{k-1}\rho]$ be a discrete series of $G_n$, with $\rho$ a cuspidal representation of $G_r$, and $n=kr$. Then for 
$l\in \{1,\dots,n\}$, the $G_{k-l}$-module $\D^{(l)}$ is null, unless 
$l=ir$ a multiple of $r$, in which case $\D^{(l)}=[|.|^i\rho,\dots,|.|^{k-1}\rho]$ when $i\leq k-1$, and $\D^{(n)}=\1$. 
If $\pi=\D_1\times \dots \times \D_t$ is a representation of $G_n$ of Whittaker type, and $l\in \{1,\dots,n\}$, then 
$\pi^{(k)}$ has a filtration with factors the modules $\D_1^{(a_1)}\times \dots \times \D_t^{(a_t)}$, for non negative integers $a_i$ such that $\sum_{i=1}^t a_i=k$.
\end{prop}

\begin{rem}\label{Kir} 
 Let $\pi$ be a representation of Whittaker type of $G_n$, and $W(\pi,\theta)$ its Whittaker model. Of course 
$Hom_{N_n}(W(\pi,\theta),\theta)$ is of dimension $1$, and spanned by $W\mapsto W(I_n)$. But considering the retriction 
to $P_n$ of the representation of $G_n$ obtained on the Whittaker model, this means that its $n$-th derivative is a vector space of dimension $1$. Finally, the filtration of Proposition \ref{BZfiltration} implies that $W(\pi,\theta)_{|P_n}$ contains 
$ind_{N_n}^{P_n}(\theta)$ (which appears with multiplicity one, as the bottom piece of the filtration). It is shown in \cite{JS} 
that $W(\pi,\theta)$ is isomorphic to $\pi$ as a $G_n$-module when $\pi$ is of Langlands' type.  
\end{rem}

The asymptotics of Whittaker functions are controlled by the exponents of the corresponding representation.

\begin{df}\label{exponent}
Let $\pi=\D_1\times \dots \times \D_t$ be a representation of $G_n$ of Whittaker type. Let $k$ be an element of $\{0,\dots,n\}$, such that 
$\pi^{(k)}$ is not zero, and let $a_1,\dots,a_n$ be a sequence of nonnegative integers, such that $\sum_{i=1}^t a_i=k$, and 
$\D_1^{(a_1)}\times \dots \times \D_t^{(a_t)}$ is nonzero. We will say that the central character of the $G_{n-k}$-module 
$\D_1^{(a_1)}\times \dots \times \D_t^{(a_t)}$ is a $k$-exponent of $\pi$. If $\pi^{(k)}=0$, the family of $k$-exponents of $\pi$ is empty. 
\end{df}

For convenience, we recall the asymptotic expansion of Whittaker functions given in the proof of Theorem 
2.1 of \cite{M2011.2} (see the ``stronger statement'' in [loc. cit.]).

\begin{prop}\label{DL}
 Let $\pi$ be a representation of $G_n$ of Whittaker type. For $k \in \{1,\dots,n\}$, let $(c_{k,i_k})_{i_k=1,\dots,r_k}$ be the family of 
$(n-k)$-exponents of $\pi$, then for every $W$ in $W(\pi,\theta)$, the map $W(z_1\dots z_n)$
 is a linear combination of functions of the form 
$$c_{\pi}(t(z_n))\prod_{k=1}^{n-1} c_{k,i_k}(t(z_k))|z_k|^{(n-k)/2}v(t(z_k))^{m_k}\phi_k(t(z_k)),$$ 
where $z_k=diag(t(z_k)I_{k},I_{n-k})$, for $i_k$ between $1$ and $r_k$, non negative integers $m_k$, and functions $\phi_k$ in $C_c^\infty(F)$.
\end{prop}

We end this section with results of Section 1 of \cite{CP}, about the link between restriction of Whittaker functions and derivatives, which will be fundamental to obtain a factorisation of the $L$ factor introduced in section \ref{The Rankin-Selberg integrals}. We collect them in the following proposition. 

\begin{prop}\label{mirabolicrestriction}
Let $\pi$ be a representation of $G_n$ of Whittaker type, $k$ be an element of $\{1,\dots,n-1\}$ and denote by $\rho$ an irreducible submodule of the $G_{k}$-module $\pi^{(n-k)}$. Then $\rho$ is generic, and we have the following properties.\\
If $W\in W(\pi,\theta)$ actually belongs to the $P_n$-submodule $W(\pi,\theta)_k$ (see Proposition \ref{BZfiltration}) of $W(\pi,\theta)$, then the map 
$W(diag(g,a,1))$ for $g$ in $G_k$ and $a$ in $A_{n-k-1}$ has compact support with respect to $a$. Moreover if $W$ projects on $W'$ in $W(\rho,\theta)$ via the surjection $\pi \rightarrow \pi^{(k)}$, then for any function $\phi'$ in $\sm_c(F^k)$, which is the characteristic function of a sufficiently small neighbourhood of $0$, we have $W(diag(g,I_{n-k}))\phi'(\eta_kg)=W'(g)\phi'(\eta_k g)|g|^{(n-k)/2}$.\\
On the other hand, any $W'$ in $W(\rho,\theta)$ is the image of some $W$ in $W(\pi,\theta)_k$, and there exists $\phi'$ in $\sm_c(F^k)$ with $\phi'(0)\neq 0$, such that $W(diag(g,I_{n-k}))\phi'(\eta_k g )=W'(g)\phi'(\eta_k g )|g|^{(n-k)/2}$. 
\end{prop}

\subsection{Meromorphic continuation of invariant linear forms}\label{Bernstein}

We will use twice a result of Bernstein insuring rationality of the solutions of linear systems with polynomial data.
 The setting is the following. Let $V$ be a complex vector space of countable dimension. Let $R$ be an index set, and let $\Xi_0$ be a collection 
$\left\lbrace (x_r, c_r)| r\in R \right\rbrace$ with $x_r \in V$ and $c_r\in \mathbb{C}$. A linear form $\lambda$ in $V^*= Hom_{\mathbb{C}}(V,\mathbb{C})$ is said to be 
a solution of the system $\Xi_0$ if $\lambda(x_r)=c_r$ for all $r$ in $R$. Let $\mathcal{X}$ be an irreducible algebraic variety over $\mathbb{C}$, and suppose that for each $d$, we have a system 
$\Xi_d =\left\lbrace (x_r(d), c_r(d))| r\in R \right\rbrace$ with the index set $R$ independent of $d$ in $\mathcal{X}$. We say 
that the family of systems 
$\left\lbrace \Xi_{d}, d \in \mathcal{X}\right\rbrace $ is polynomial if $x_r(d)$ and $c_r(d)$ belong respectively to 
$ \mathbb{C}[\mathcal{X}]\otimes_{\mathbb{C}} V$ and $\mathbb{C}[\mathcal{X}]$. Let $\mathcal{M}=\mathbb{C}(\mathcal{X})$ be the field of fractions of 
$\mathbb{C}[\mathcal{X}]$, we denote by $V_{\mathcal{M}}$ the space $\mathcal{M} \otimes_{\mathbb{C}} V$ and by $V_{\mathcal{M}}^*$ the space 
$Hom_{\mathcal{M}}(V_{\mathcal{M}},\mathcal{M})$. Writing $V^*=Hom_{\C}(V,\C)$, we notice that $V_{\mathcal{M}}^*=(V^*)_\mathcal{M}$. 
The following statement is a consequence of Bernstein's theorem and its corollary in Section 
1 of \cite{Ba}.

\begin{thm}{(Bernstein)}\label{Ber}
Suppose that in the above situation, the variety $\mathcal{X}$ is nonsingular and that there exists a non-empty subset $\Omega \subset \mathcal{X}$ open for 
the complex topology of $\mathcal{X}$, such that for each $d$ in $\Omega$, the system $\Xi_d$ has a unique solution $\lambda_{d}$. 
Then the system $\Xi=\left\lbrace (x_r(d),c_r(d))|r\in R \right\rbrace$ over the field $\mathcal{M}=\mathbb{C}(\mathcal{X})$ has a unique solution
 $\lambda(d)$ in $V_{\mathcal{M}}^*$, and $\lambda(d)=\lambda_d$ is the unique solution of $\Xi_d$ on $\Omega$. 
\end{thm}

We give a corollary of this theorem which allows to extend meromorphically invariant linear forms on parabolically induced representations. 
To this end, we recall a few facts about flat sections of induced representations, which are for example discussed in section 3 of \cite{CP}. 
Let $\pi=\pi_1\times \dots \times \pi_t$ be a representation of $G_n$ parabolically induced from $P_{\bar{n}}$, where each $\pi_i$ is a 
representation of $G_{n_i}$. If $u=(u_1,\dots,u_t)$ is an element of $\mathcal{D}^t$, we denote by $\pi_u$ the representation 
$|.|^{u_1}\pi_1\times\dots\times|.|^{u_1}\pi_t$. We denote by $V_{\pi_u}$ the space of $\pi_u$. Call 
$\eta_u$ the map from $G_n$ to $\C$, defined by $\eta_u(nmk)=\prod_{i=1}^t|m_i|^{u_i}$, for $n$ in $N_{\bar{n}}$, $m=diag(m_1,\dots,m_t)$ in
 $M_{\bar{m}}$and $k$ in $K_n$, then $f\mapsto \eta_u f$ is an isomorphism between the $K_n$-modules $V_\pi$ and $V_{\pi_u}$. 
We denote by $f_u$ the map $\eta_u f$ when $f$ 
belongs to $V_\pi$ (hence $f_0=f$); it is clear that ${f_u}_{|K_n}$ is independant of $u$. If we denote by $\mathcal{F}_\pi$ the space of 
restrictions $f_c={f}_{|K_n}$ for $f$ in $V_{\pi}$, then $f_u\mapsto f_c$ is a vector space isomorphism between $V_{\pi_u}$ and $\mathcal{F}_\pi$. 
Thus $\mathcal{F}_\pi$ is a model of the representation $\pi_u$, where the action of $G_n$ is given by 
$\pi_u(g)f_c= (\pi_u (g)f_u)_c$. For $g$ in $G_n$, the map $u\mapsto \pi_u(g)f_c$ belongs to $\C[\mathcal{D}^t]\otimes_\C \mathcal{F}_\pi$.

\begin{cor}\label{rationalinvariant}
With notations as above, let $\mathcal{D}'$ be an affine line of $\mathcal{D}^t$ (i.e. a closed subvariety of $\mathcal{D}^t$, such that there is 
an isomorphism $i$ from $\mathcal{D}'$ to $\mathcal{D}$). Let $H$ be a 
closed subgroup of $G_n$, and $\chi$ a character of $H$, such that for all $s$ in $\mathcal{D}'$, except for $s$ in a finite subset $S$ of
 $\mathcal{D}'$, the space 
$Hom_H(\pi_s,\chi)$ is of dimension $\leq 1$. Suppose that for every $s$ in some subset $\Omega$ of $\mathcal{D}'$, open 
for the topology induced by the complex structure, there is $\l_s$ in $Hom_H(\pi_s,\chi)$, and that there is $g$ in $V_\pi$ such that 
$\l_s(g_s)$ is a non zero element of $\C[\Omega]$ which extends to an element $Q$ of $\C[\mathcal{D}']$, then $\pi_s$ is $(H,\chi)$-distinguished 
for all $s$ in $\mathcal{D}'$.
\end{cor}
\begin{proof}
We consider the space $V=\mathcal{F}_\pi$, which is of countable dimension. Let $(f_\alpha)_{\alpha\in A}$ be a basis of $V$, we consider 
for each $s$ in $\mathcal{D}'$, the system $$\Xi_s=\{(\pi_s(h)f_\alpha-\chi(h)f_\alpha,0),h\in H,\alpha \in A\}\cup \{(g_{s,c},Q(s))\}.$$ 
The family 
$\Xi$ is polynomial, and 
$\Xi_s$ has a unique solution $\widetilde{\l_s}:f_{s,c}\mapsto \l_s(f_s)$ for $s$ in $\Omega-S\cap \Omega$ by hypothesis. Hence, according to 
Theorem \ref{Ber}, there is $\widetilde{\l}$ in $V_{\C(\mathcal{D}')}^*$, such that $\widetilde{\l}(s)=\widetilde{\l_s}$ on $\Omega$. Let 
$P$ be a nonzero 
element of $\C[\mathcal{D}']$, such that $\widetilde{\mu}=P\widetilde{\l}$ belongs to $(V^*)_{\C[\mathcal{D}']}$.
Then, by the principle of analytic continuation, the linear map $\widetilde{\mu}(s)$ belongs to $Hom_H(\pi_s,\chi)$ for all $s$ in $\mathcal{D}'$. 
In particular, for $s_0$ in $\mathcal{D}'$, if $m\in \Z$ is the order of $s_0$ as a zero of $\widetilde{\mu}$, the map 
$(q^{-i(s)}-q^{-i(s_0)})^{-m} \widetilde{\mu}(s)$ evaluated at $s_0$ is a nonzero element of 
$Hom_H(\pi_{s_0},\chi)$. 
\end{proof}

\section{Classification of distinguished generic representations in terms of discrete series}\label{classification}

In the section, we give a classification, of $(H_n,\chi_\alpha)$-distinguished generic representations of $G_n$ for $Re(\alpha)\in [-1/2,1/2]$. 
We recall from \cite{M2012.2}, that 
for any maximal Levi subgroup $M$ which is not conjugate to $H_n$, and $\chi$ a character of $M$, no generic representation of 
$G_n$ can be $(M,\chi)$-distinguished. 
Moreover, discrete series are $(M,\chi)$-distinguished if and only if $n$ is even (and $M\simeq H_n$).

\subsection{Representatives of $H\backslash G/P$}

For this section, we will often write $G=G_n$, $H=H_n$, $W'=W'_n$, $\e$ for the matrix $\e_n$, and $P=P_{\bar{n}}$ for a 
partition $\bar{n}$ of $n$. We first give an invariant which characterises the orbits for the action of $H$ on the flag variety
 $G/P$. We therefore identify the elements of $G/P$ with the set of flags $0\subset V_1\subset \dots \subset V_t=V$, where $V=F^n$, and 
$dim(V_i)=n_1+\dots + n_i$. We do this via the map $g\mapsto 0\subset gV_1^0 \subset \dots \subset gV_t^0$, where $V_i^0$ is the subspace 
of $V$ spanned by the 
$n_1+\dots + n_i$ first vectors of the canonical basis. If 
$0\subset V_1\subset \dots \subset V_t=V$ is an element of $G/P$, as in \cite{M2011}, when $i\leq j$, we denote by $S_{i,j}$ a
 supplementary space of $V_i \cap \e(V_{j-1})+ V_{i-1} \cap \e(V_j)$ in $V_i \cap \e(V_j)$. If $i=j$, we add the condition that 
the supplementary space $S_{i,i}$ we choose is $\e$-stable, which is possible as $\e$ is semi-simple. Eventually, for 
$1\leq i \leq j \leq t$, we denote by $S_{j,i}$, the space $\e(S_{i,j})$, and as in \cite{M2011}, one has 
$V=\oplus_{(k,l)\in \{1,\dots,t\}^2} S_{k,l}$. If $W$ is a subspace of $V$, which is $\e$-stable, we denote by $W^-$ the
eigen-subspace of $\e$ associated to $1$, and by $W^+$ the eigen-subspace of $\e$ associated to $-1$. We denote by $d_{i,j}$
 the dimension of $S_{i,j}$, which only depends on the $V_k$'s, and by $d_{i,i}^+$ and $d_{i,i}^-$ the dimensions 
$dim(S_{i,i}^+)$ and $dim(S_{i,i}^-)$ respectively.

\begin{prop}
Let $0\subset V_1\subset \dots \subset V_t=V$ and $0\subset V'_1\subset \dots \subset V'_t=V$ be two elements of $G/P$, then 
they are in the same $M$-orbit if and only if 
$d_{i,j}=d'_{i,j}$ if $i\leq j$, and moreover $d_{i,i}^+={d'_{i,i}}^+$ for all $i$.
\end{prop}
\begin{proof}
Write $V=\oplus_{(k,l)\in \{1,\dots,t\}^2} S_{k,l}$. For $1\leq i < j \leq t$, we choose an isomorphism 
$h_{i,j}$ between $S_{i,j}$ and $S'_{i,j}$. This defines an isomorphism $h_{j,i}$ between $S_{j,i}$ and $S'_{j,i}$, satisfying 
$h_{j,i}(v)= \e h_{i,j} \e (v)$ for all 
$v$ in $S_{j,i}$. Eventually, for each $l$ between $1$ and $t$, we choose an isomorphism $h_{l,l}$ between $S_{l,l}$ and $S'_{l,l}$ commuting with $\e$, 
which is possible as $d_{l,l}^+={d'_{l,l}}^+$.
\end{proof}
 
Hence the elements of $H\backslash G/P$ correspond to sets of the form 
$$s=\{n_{i,j} ,1\leq i <j \leq t, (n_{k,k}^+,n_{k,k}^-),1\leq k  \leq t\}$$
 such that if we set $n_{j,i}=n_{i,j}$, and $n_{k,k}=n_{k,k}^+ + n_{k,k}^{-}$, then $n_i=\sum_{j=1}^t n_{i,j}$, and moreover 
$\sum_{k=1}^t n_{k,k}^+=\sum_{k=1}^t n_{k,k}^-$ when $n$ is even, and $\sum_{k=1}^t n_{k,k}^+=\sum_{k=1}^t n_{k,k}^-+1$ when $n$ is odd. 
We denote by $I(\bar{n})$ this family of sets, and call its elements the relevant subpartitions of $\bar{n}$. We will often write 
$s= \{n_{i,j}\}$ when the context is clear. As $\bar{n}$ is a partition of $n$, we can write a matrix in $G$ by blocks 
corresponding to the sub-blocks in $s$, we will in particular use the $((i<j),(j>i))$-th diagonal blocks of size $n_{i,j}+n_{j,i}$, and the 
$(i,i)$-th diagonal blocks of size $n_{i,i}$. We hope that the notation is clear.

\begin{prop} \label{HP}{(Representatives for $H\backslash G/P$)}\\ 
Let $s=\{n_{i,j} ,1\leq i <j \leq t, (n_{k,k}^+,n_{k,k}^-),1\leq k  \leq t\}$ be an element of $I(\bar{n})$.\\ 
Let $u_s$ be any element of $G$ such that $u_s\e u_s^{-1}$ is the block diagonal matrix $\e_s$
with $(i,i)$-block $\begin{pmatrix} I_{n_{i,i}^+} & \\ &  -I_{n_{i,i}^-}  \end{pmatrix}$, and $((i<j),(j>i))$-th diagonal block   
$\left( \begin{array}{cc}   & I_{n_{i,j}}\\
  I_{n_{i,j}} &  \end{array} \right),$ (which is possible since $\e$ and $\e_s$ are conjugate), then $Hu_s^{-1}P$ corresponds to 
 to the element $s$ of $I(\bar{n})$. \\
For an (almost) explicit choice, we denote by $a_s$ the block diagonal element with $(i,i)$-th block $I_{n_{i,i}}$, and 
$((i<j),(j>i))$-th diagonal block   
$$ \frac{1}{\sqrt{2}}\left( \begin{array}{cc} I_{n_{i,j}} & -I_{n_{i,j}}\\
  I_{n_{i,j}} &  I_{n_{i,j}} \end{array} \right),$$\\
and choose for $w_s$ an element of $W'$ such that 
$w_s\e w_s^{-1}$ has its $(i<j)$-th diagonal block equal to $I_{n_{i,j}}$, its $(j>i)$-th diagonal block equal to 
$-I_{n_{j,i}}$, and its $(i,i)$-th diagonal block 
$\begin{pmatrix} I_{n_{i,i}^+} & \\ &  -I_{n_{i,i}^-}  \end{pmatrix}$, then $u_s=a_sw_s$ is a possible choice. 
\end{prop}
\begin{proof}
Using the notations of Proposition 2.1 of \cite{M2011}, we denote $S_{i,j}=u_s(V_{i,j}^0)$. From our choice of $u_s$, we have 
$dim(S_{i,i}^+)=n_{i,i}^+$, $dim(S_{i,i}^-)=n_{i,i}^-$, and $\e(S_{i,j})=S_{j,i}$ if $i<j$. It is now clear that the flag of the $V_i$'s, defined by 
$V_i=u_s(V_i^0)$, corresponds to $s$. 
\end{proof}

Hence a set of representatives for $P\backslash G/H$ is obtained as $\{u_s,s\in I(\bar{n})\}$. 

\subsection{Structure of the group $P\cap u_sHu_s^{-1}$ and modulus characters}\label{structure}

Let $u_s$ be an element of $R(P\backslash G/H)$ as in Proposition \ref{HP}, corresponding to a sequence 
$$s=\{n_{i,j} ,1\leq i <j \leq t, (n_{k,k}^+,n_{k,k}^-),1\leq k  \leq t\}$$ 
in $I(\bar{n})$, and $\e_s$ the matrix $u_s\e u_s^{-1}$. We give the structure of the group 
$P\cap u_sHu_s^{-1}$ following \cite{M2011}. We already notice that the subgroup $u_sHu_s^{-1}$ is the subgroup of $G$ fixed by the involution $\theta_s:x\mapsto w_s\e(x)w_s^{-1}$, where $w_s$ is the element $u_s\e(u_s)^{-1}$ of $W'$. We notice that $\theta_s$ is nothing else than $x\mapsto \e_sx\e_s^{-1}$. 
If $C$ is a subset of $G$, we denote by $C^{<\theta_s>}$ the subset $C\cap u_sHu_s^{-1}$ of $C$. We sometimes, by abuse of notation, denote 
by $s$ the subpartition of $\bar{n}$ by the $n_{i,j}$'s, hence the notation $P_s$ makes sense, and designates a parabolic subgroup of $G$ 
included in $P$. The Proposition 3.1. of \cite{M2011} is still valid, with the same proof.

\begin{prop}
The group $P\cap u_sHu_s^{-1}$ is equal to $P_s\cap u_sHu_s^{-1}$.
\end{prop}

\noindent Now, as $u_s$ is $P_s$-admissible (see \cite{JLR}), and Lemma 21 of \cite{JLR} applies. 

\begin{prop}\label{decomposition}
The group $P^{<\theta_s>}= P\cap u_sHu_s^{-1}$ is the semi-direct
 product of $M_s^{<\theta_s>}$ and $N_s^{<\theta_s>}$. The group $M_s^{<\theta_s>}$ is given by the matrices 
$$\begin{bmatrix} 
A_{1,1} &        &        &       &       &        &        &       &        \\
        & \ddots &        &       &       &        &        &       &        \\
        &        & A_{1,t}&       &       &        &        &       &        \\
        &        &        &A_{2,1}&       &        &        &       &        \\
        &        &        &       &\ddots &        &        &       &        \\
        &        &        &       &       & A_{2,t}&        &       &        \\
        &        &        &       &       &        & A_{t,1}&       &        \\
        &        &        &       &       &        &        &\ddots &        \\
        &        &        &       &       &        &        &       &  A_{t,t} 
\end{bmatrix}$$ with $A_{j,i}=A_{i,j}$ in $G_{n_{i,j}}$, and $A_{i,i}$ in the standard Levi subgroup of $G_{n_{i,i}}$ of type $(n_{i,i}^+,n_{i,i}^-)$. 
  
\end{prop}

\noindent We will thus write elements of $M_s^{<\theta_s>}$ as 
$$m=diag(m_{1,1}^+,m_{1,1}^-,\dots,m_{i,j},\dots,m_{j,i},\dots,m_{t,t}^+,m_{t,t}^-),$$ 
with $m_{k,k}^+$ in $G_{n_{k,k}^+}$, $m_{k,k}^-$ in $G_{n_{k,k}^-}$, and $m_{i,j}=m_{j,i}$ in $G_{n_{i,j}}$.
Proposition 3.3 of \cite{M2011} is also still valid, with the same proof.

\begin{prop}\label{util}
If we denote by $P'_s$ the standard parabolic subgroup $M\cap P_s$ of $M$, and by $N'_s$ the unipotent radical of $P'_s$, 
then the following inclusion is true: $$N'_s \subset N_s^{<\theta_s>}N.$$
\end{prop}

We are now going to give some relations between modulus characters, which will be used in the next section, for the classification of distinguished 
representations of $G_n$.

\begin{prop}\label{modulus}
Let $n$ be a positive integer, and $s$ be an element of $I(\bar{n})$. If $\mu$ is a character of $H$, we denote by $\mu^s$ the character 
$h_s\mapsto \mu(u_s^{-1} h_s u_s)$ of $u_s H u_s^{-1}$. Let $\alpha$ be a character of $F^*$, and 
$\chi_{\alpha}:h(g_1,g_2)\mapsto \alpha(det(g_1)/det(g_2))$ the associate character 
of $H_n$. We have the following relations. 
\begin{enumerate}
  \item $(\delta_P\delta_{P'_s})_{|M_s^{<\theta_s>}}=(\delta_{P_s})_{|M_s^{<\theta_s>}}$.
  
  \item We have the equalities, for $m\in M_s^{<\theta_s>}$: 
\begin{equation}\label{modulusquotient}\frac{\d_{P_s^{<\theta_s>}}}{\delta_{P_s}^{1/2}}(m)=
\prod_{1\leq i<j\leq t} |m_{i,i}^+|^{\frac{n_{j,j}^+-n_{j,j}^-}{2}}|m_{i,i}^-|^{\frac{n_{j,j}^- -n_{j,j}^+}{2}}
|m_{j,j}^+|^{\frac{n_{i,i}^- -n_{i,i}^+}{2}} |m_{j,j}^-|^{\frac{n_{i,i}^+ -n_{i,i}^-}{2}}\end{equation}

\begin{equation}\label{chialpha} \chi_\alpha(m)= \prod_{i=1}^t \frac{\alpha(m_{i,i}^+)}{\alpha(m_{i,i}^-)} \end{equation}

\end{enumerate}
\end{prop}
\begin{proof}
As the characters are positive, and $M_s^{<\theta_s>}$ is reductive, according to Lemma 1.10 of \cite{KT}, we need to check 
this equality only on the (connected) center $Z_s^{<\theta_s>}$ of 
$M_s^{<\theta_s>}$. An element  $\l$ of $Z_s^{<\theta_s>}$ is of the form $$diag(\l_{1,1}^+I_{n_{1,1}^+},\l_{1,1}^- I_{n_{1,1}^-},\dots,
\l_{i,j}I_{n_{i,j}},\dots, \l_{j,i}I_{n_{j,i}}, \dots,\l_{t,t}^+I_{n_{t,t}^+},\l_{t,t}^- I_{n_{t,t}^-}),$$ 
with $\l_{i,j}=\l_{j,i}$ when $i<j$. The assertion on $\chi_\alpha^s$ is a consequence of the formula:
$$\chi_\alpha^s(\l)=\prod_{i=1}^t\frac{\alpha(\l_{i,i}^+)^{n_{i,i}^+}}{\alpha(\l_{i,i}^-)^{n_{i,i}^-}}\Rightarrow 
\chi_\alpha^s(m)=\prod_{i=1}^t\frac{\alpha(m_{i,i}^+)^{n_{i,i}^+}}{\alpha(m_{i,i}^-)^{n_{i,i}^-}}.$$
As a particular case, by definition of $\d_n$, one obtains $\d_n^{s}(m)=\prod_{i=1}^t\frac{|m_{i,i}^+|}{|m_{i,i}^-|}$ for 
$m$ in $M_s^{<\theta_s>}$.\\
 Let $\mathfrak{N}$ be $Lie(N_P)$, $\mathfrak{N}_s$ be $Lie(N_{P_s})$ and $\mathfrak{N'}_s$ be $Lie(N_{P'_s})$.
  The characters $\delta_P$ and $\delta_{P'_s}$ satisfy $\delta_P(m)=|det(Ad(m)_{|\mathfrak{N}})|$ and $\delta_{P'_s}(m)=|det(Ad(m)_{|\mathfrak{N'}_s})|$ for $m$ in $M_s$,
 hence the equality $(\delta_P \delta_{P'_s})_{|M_s}= {\delta_{P_s}}_{|M_s}$ holds as we have 
$\mathfrak{N}_s=\mathfrak{N'}_s\oplus \mathfrak{N}_P$, this proves 1.\\
 We now prove $2.$. Let $k<l$ be two integers between $1$ and $t$ such that $n_{k,l}>0$, and let $n'$ be equal to $n-2n_{k,l}$. We denote by 
$s'$ the element of $I(\overline{n'})$, equal to $s-\{n_{k,l},n_{l,k}\}$. We write an element $m$ of $M_s^{<\theta_s>}$ as $diag(m_1,g,m_2,g,m_3)$, 
with 
$m'=diag(m_1,m_2,m_3)\in M_{s'}^{<\theta_{s}'>}$, and denote by $a_i$ the number of terms on $m_i$'s diagonal, then one has
 $$\d_{P_s}^{<\theta_s>}(m)=\d_{P_{s'}}^{\theta_{s'}}(m')|m_1|^{n_{k,l}}|m_3|^{-n_{k,l}}|g|^{a_3-a_1},$$ but 
 $\d_{P_s}(m)=\d_{P_{s'}}(m')|m_1|^{2n_{k,l}}|m_3|^{-2n_{k,l}}|g|^{2a_3-2a_1}$, so we obtain 
$$\frac{\d_{P_s}^{<\theta_s>}}{\d_{P_s}^{1/2}}(m)=\frac{\d_{P_{s'}}^{\theta_{s'}}}{\d_{P_{s'}}^{1/2} }(m')$$ for $m$ in 
$M_s^{\theta_s}$. As a consequence, by induction, if we write $n''=n-\sum_{i<j} 2n_{i,j}$, and $s''=(n_{1,1}^+,n_{1,1}^-,n_{2,2}^+,n_{2,2}^-,\dots,n_{t,t}^+,n_{t,t}^-)$ in 
$I(\overline{n''})$, we have, 
 setting $$m''=diag(m_{1,1}^+,m_{1,1}^-,m_{2,2}^+,m_{2,2}^-,\dots,m_{t,t}^+,m_{t,t}^-),$$ the equality 
$\frac{\d_{P_s^{<\theta_s>}}}{\delta_{P_s}^{1/2}}(m)=\frac{\d_{P_{s''}^{<\theta_{s''}>}}}{\delta_{P_{s''}}^{1/2}}(m'').$ 
 The result then follows as 
$$\d_{P_{s''}}^{<\theta_{s''}>}(m'')=\prod_{i<j}|m_{i,i}^+|^{n_{j,j}^+}|m_{j,j}^+|^{-n_{i,i}^+}|m_{i,i}^-|^{n_{j,j}^-}|m_{j,j}^-|^{-n_{i,i}^-}$$
 and 
 $$\d_{P_{s''}}(m'')=\prod_{i<j}(|m_{i,i}^+||m_{i,i}^-|)^{n_{j,j}}(|m_{j,j}^+||m_{j,j}^-|)^{-n_{i,i}}.$$

\end{proof}

\subsection{Distinguished generic representations}

Let $Pr$ be the continuous map defined by $Pr:g\mapsto g\e g^{-1}$. It induces a continuous injection of $X=G/H$ onto the conjugacy class 
$Y=Ad(G)\e$, which is closed as $\e$ is diagonal. We start with the following lemma. If $X$ is the set of $F$-points of an algebraic variety 
defined over $F$, we will denote by $\mathbb{X}$ this algebraic variety, and write $\mathbb{X}(F)=X$. 

\begin{LM}
The double class $PH$ is closed.
\end{LM}
\begin{proof} 
The double class $PH$ is equal to $Pr^{-1}(Ad(P)\e)$, hence the lemma will be a consequence of the fact that $Ad(P)\e$ is closed. But 
$Ad(\mathbb{P}).\e$ is Zariski closed in $\mathbb{G}$ according to theorem 10.2 of \cite{BT}. As $F$ is of characteristic zero, then 
$Ad(P)\e$ is open and closed in $Ad(\mathbb{P})\e(F)$ according to corollary A.1.6. of \cite{AG}, and we are done. However it is possible
 to give an elementary proof whenever the characteristic is not $2$. Let $p$ 
be an element of $P\cap Ad(G)\e$. Then $p=g\e g^{-1}$ for some $g$ in $G$, hence $pg=g\e$. Now write $g=p'wu$ with $p'\in P$, $w\in W'$, 
and $u$ in a $Ad(A)$-stable subset of 
$U$ of $N$, such that $(p,u)\mapsto pwu$ is a bijection (see for example \cite{C}, before Proposition 1.3.2). Then 
 $pp'wu=p'wu\e= p w^{-1}(\e) w \e(u)$, hence $pp'=p'w^{-1}(\e)$, and $p\in Ad(P)w^{-1}(\e)$. In particular $Ad(\mathbb{P}).\e(F)\subset P\cap Ad(G)\e$ is a disjoint union of 
orbits $Ad(P)w(\e)$ for some $w$'s in $W'$. But any such $w(\e)$ is diagonal, and conjugate to $\e$ by an element of $\mathbb{P}$, this implies at once that 
it is conjugate to  $\e$ by an element of $\mathbb{M}$, hence by an element of $M$, as $w(\e)$ has coefficients in $F$. Thus 
 $Ad(\mathbb{P})\e(F)=Ad(P)\e$, and the result follows.
\end{proof}

We need the following result about preservation of distinction by parabolic induction.

\begin{prop}\label{induced}
Let $n_1=2m_1$ and $n_2=2m_2$ be to even integers, let $\alpha$ be a character of $F^*$. Let $\pi_1$ be 
a $(H_{n_1},\chi_\alpha)$-distinguished representation of $G_{n_1}$, and $\pi_2$ be a $(H_{n_2},\chi_\alpha)$-distinguished representation of $G_{n_2}$, then 
$\pi=\pi_1\times \pi_2$ is $(H_{n_1+n_2},\chi_\alpha)$-distinguished, and $\pi'=\pi_1\times \alpha|.|^{-1/2}$ is
$(H_{n_1+1},\chi_\alpha\d_{n_1+1}^{-1/2})$-distinguished.
\end{prop}
\begin{proof}
Let $P$ be the standard parabolic subgroup of $G$ with Levi $M=G_{n_1}\times G_{n_2}$, and $H=H_{n_1+n_2}$. As $PH$ is closed, we deduce that 
$ind_{P\cap H}^H(\d_{P}^{1/2}\pi_1\otimes \pi_2)$ is a quotient of $\pi_1\times \pi_2$. But 
we have $$Hom_H(ind_{P\cap H}^H(\d_{P}^{1/2}\pi_1\otimes \pi_2),\chi_\alpha)\simeq Hom_{P\cap H}(\d_{P}^{1/2}\pi_1\otimes \pi_2,\d_{P\cap H}\chi_\alpha)$$ 
by Frobenius reciprocity law. We claim that $\d_{P}=\d_{P\cap H}^2$ on $P\cap H$. It suffices to check this on the center $Z_{M\cap H}$ of 
$M\cap H$. An element $Z_{M\cap H}$ is of the form $$z(t_1,u_1,t_2,u_2,v_2)=diag(z(t_1,u_1),z(t_2,u_2)),$$ where 
 
$\!\!\!\!\!\!\!\!\!\!\!\!\!\!\!\!z(t_1,u_1)= \!\! \left(\begin{array}{ccccc}      t_1 &       &         &      &       \\
                                                &  u_1  &         &      &       \\
                                                &       & \ddots  &      &       \\    
                                                &       &         & t_1  &       \\ 
                                                &       &         &      & u_1   \end{array}\right)\in G_{n_1}$, 
$z(t_2,u_2)= \!\! \left(\begin{array}{ccccc} t_2 &       &         &      &       \\
                                                &  u_2  &         &      &       \\
                                                &       & \ddots  &      &       \\    
                                                &       &         & t_2  &       \\ 
                                                &       &         &      & u_2   \end{array}\right)\in G_{n_2}$. 

Then one checks that $$\d_{P\cap H}(z(t_1,u_1,t_2,u_2,v_2))=(|t_1|/|t_2|)^{m_1m_2} (|u_1|/|u_2|)^{m_1m_2},$$
whereas $$\d_{P}(z(t_1,u_1,t_2,u_2,v_2))=|t_1u_1|^{m_1n_2}/|t_2u_2|)^{n_1m_2},$$ and the equality follows. 
Finally, we deduce that $$Hom_H(ind_{P\cap H}^H(\d_{P}^{1/2}\pi_1\otimes \pi_2),\chi_\alpha)\simeq 
Hom_{M\cap H}(\pi_1\otimes \pi_2,\chi_\alpha\otimes \chi_\alpha),$$ 
hence is nonzero by hypothesis, and $\pi=\pi_1\times \pi_2$ is $(H,\chi_\alpha)$-distinguished.\\
For $\pi'$, the proof is the same. This time $H=H_{n_1+1}$ and $P$ is the standard parabolic subgroup of $G=G_{n_1+1}$ 
with Levi $M=G_{n_1}\times G_{1}$, and we write $\d$ for $\d_{n_1+1}$. The $H$-module $ind_{P\cap H}^H(\d_{P}^{1/2}\pi_1\otimes \alpha|.|^{-1/2})$ is still 
a quotient of $\pi_1\times \alpha|.|^{-1/2}$ and we still have 
$$Hom_H(ind_{P\cap H}^H(\d_{P}^{1/2}\pi_1\otimes \alpha|.|^{-1/2}),\chi_\alpha\d^{-1/2})
\simeq Hom_{M\cap H}(\d_{P}^{1/2}\pi_1\otimes \alpha|.|^{-1/2},\d_{P\cap H}\chi_\alpha\d^{-1/2}).$$ 
An element of the center $Z_{M\cap H}$ is of the form $z(t_1,u_1,t_2)=diag(z(t_1,u_1),t_2)$, with $z(t_1,u_1)$ as before, and one checks that 
$\d_{P\cap H}(z(t_1,u_1,t_2))=(|t_1|/|t_2|)^{m_1}$, whereas $$\d_{P}^{1/2}(z(t_1,u_1,t_2))=|t_1 u_1|^{m_1/2}/|t_2|^{m_1},$$ hence 
$\frac{\d_{P\cap H}}{\d_{P}^{1/2}}= |t_1|^{m_1/2}|u_1|^{-m_1/2}$. Finally, one has 
$$\d^{-1/2}z(t_1,u_1,t_2)= |t_1|^{-m_1/2}|t_2|^{-1/2} |u_1|^{m_1/2},$$ hence the character 
$\chi=\frac{\d_{P\cap H}\d^{-1/2}}{\d_{P}^{1/2}}$ of $M\cap H$ is equal to $\chi:diag (h_1,t_2)\mapsto |t_2|^{-1/2}$ for $h_1$ in $H_{n_1}$ and 
$t_2$ in $F^*$. We deduce from the isomorphisms 
$$Hom_H(ind_{P\cap H}^H(\d_{P}^{1/2}\pi_1\otimes \alpha|.|^{-1/2}),
\chi_\alpha\d^{-1/2})\simeq Hom_{M\cap H}(\pi_1\otimes \alpha|.|^{-1/2},\chi_\alpha\chi)$$ 
$$\simeq Hom_{M\cap H}(\pi_1,\chi_\alpha),$$ that the first space is nonzero, hence that $\pi'$ is $(H,\chi_\alpha\d^{-1/2})$-distinguished.  
\end{proof}

When $n$ is even, $(H_n,\chi_\alpha)$-distinction is equivalent to $(H_n,\chi_\alpha^{-1})$-distinction.

\begin{LM}\label{dualdist}
Let $\alpha$ be a character of $F^*$, and $\pi$ be a representation of $G_n$ for $n$ even, it is $(H_n,\chi_\alpha)$-distinguished if and only if it is 
$(H_n,\chi_\alpha^{-1})$-distinguished.
\end{LM}
\begin{proof}
It is enough to show one implication. Suppose that $\pi$ is $(H_n,\chi_\alpha)$-distinguished, and let $w$ be a Weyl element of $G_n$ which satsifies for all 
$h(g_1,g_2)$ in $H_n$, the relation $$w(h(g_1,g_2))=h(g_2,g_1).$$ If 
$L$ is a linear form which is $(H_n,\chi_\alpha)$-invariant on the space of $\pi$, Then $L\circ \pi(w)$ is $(H_n,\chi_\alpha^{-1})$-invariant.
\end{proof}

We will also need to know that, when $n=2m$ is even, a representation of the form $\pi\times \pi^\vee$ is
 $(H_n,\chi_\alpha)$-distinguished for all characters $\alpha $ of $F^*$. For this purpose, 
we introduce the Shalika subgroup $S_n=\{\begin{pmatrix} g & x \\ 0 & g  \end{pmatrix}, g\in G_m, x \in \mathcal{M}_{m}(F)\}$ of $G_n$. 
We denote by $\theta$ again the character $\begin{pmatrix} g & x \\ 0 & g  \end{pmatrix}\mapsto \theta(Tr(x))$ of $S_n$.

\begin{prop}\label{distinduites}
Let $n=2m$ be an even positive integer. Let $\Delta$ be unitary discrete series of $G_m$, and $\alpha$ a character of $F^*$. The representation 
$\Pi_s=|.|^s\Delta\times |.|^{-s}\Delta^\vee$ is $(S_n,\theta)$-distinguished for any complex $s$. As a consequence, 
$\Pi_s$ is $(H_n,\chi_\alpha)$-distinguished 
whenever it is irreducible.
\end{prop}
\begin{proof}
For the first part, we introduce a complex parameter $s$, and consider the representation $\Pi_s=\D|.|^s\times \D^\vee|.|^{-s}$. If $\eta_s$ is the 
smooth function $$\eta_s:diag(g_1,g_2)uk\in M_{(m,m)}N_{(m,m)}K_n\mapsto (|g_1|/|g_2|)^s$$ on $G_n$ (hence $\eta_s$ corrsponds to the map $\eta_{(s,-s)}$ of 
section \ref{Bernstein}), then $f\mapsto f_s=\eta_s f$ is a $\C$-vector space ismorphisme from 
$\Pi$ to $\Pi_s$. We will denote by $\tilde{\eta}_s$ the map $g\in G\mapsto \eta_s(wg)$, where we denote by $w$ the matrix 
$\begin{pmatrix}  & I_m \\ I_m&  \end{pmatrix}$. One can express $\tilde{\eta}_s$ as an integral, see \cite{M2010}, Lemma 3.4.

\begin{LM}{(Jacquet)} Let $\Phi_0$ be the characteristic function of $\mathcal{M}(m,\O)$, then from the Godement-Jacquet
 theory of Zeta functions of simple algebras, the integral $\int_{G_n}\Phi_0 (h) |h|^s d^*h$ is convergent for $Re(s)\geq m-1$, and is equal to $1/P(q^{-s})$ for a nonzero polynomial $P$. Then, for $Re(s)\geq (m-1)/2$, and $g$ in $G_m$, denoting by $\Phi$ the characteristic function of $\mathcal{M}_{m,2m}(\o)$
(matrices with $m$ rows and $2m$ columns) one has $$\tilde{\eta}_s(g)=P(q^{-2s})|g|^s \int_{G_m} \Phi[(h,0)g]|h|^{2s}d^*h.$$
 \end{LM}

Let $L$ be the linear fomr $v\otimes v^\vee \mapsto v^\vee(v)$ on $\pi\otimes \pi^\vee$. To avoid confusions, we will write, as in the statement above,  
$d^*h$ for the Haar measure on $G_m$ instead of $dh$. We define formally  
the linear form $$\lambda_s: f_s\mapsto \int_{x\in \mathcal{M}_m} L(f_s)(w\begin{pmatrix} I_m & x \\  & I_m \end{pmatrix}) \theta^{-1}(x)dx$$ on the space of $\Pi_s$. 
However, for $Re(s)\geq (m-1)/2$, we have $$ \lambda_s( f_s) =  \int_{x\in \mathcal{M}_m} \tilde{\eta}_s \begin{pmatrix}I_m & x \\  & I_m\end{pmatrix} L(f_0)(w\begin{pmatrix} I_m & x \\  & I_m \end{pmatrix})\theta^{-1}(x) dx $$
  $$ = P(q^{-2s}) \int_{x\in \mathcal{M}_m}\int_{h\in G_m} \Phi(h,hx)L(f_0)(w\begin{pmatrix} I_m & x \\  & I_m \end{pmatrix})|h|^{2s} \theta^{-1}(x)d^*h dx $$
  $$ = P(q^{-2s}) \int_{x\in G_m}\int_{h\in G_m} \Phi(h,x)L(f_0)(w\begin{pmatrix} I_m & h^{-1}x \\  & I_m \end{pmatrix})|h|^{2s-m}|x|^{m} \theta^{-1}(x) d^*h d^*x.$$ 
  As $g\mapsto L(f_0)(g)$ is a coefficient of the representation $\Pi$, which is unitary, it is bounded, hence this last integral is absolutely convergent at least for $s>m/2$.
 Moreover, as $P_{(m,m)}w N_{(m,m)}$ is open in $G_n$, the space $\sm_c(P_{(m,m)} \backslash P_{(m,m)}wN_{(m,m)}, \delta_{P_{(m,m)}}^{-1/2}|.|^s\D\otimes |.|^{-s}\D^\vee)$ is a subspace of $\Pi_s$, but it is also isomorphic to $\sm_c(N_{(m,m)})\otimes V_{\D\otimes \D^\vee}$ (by restriction of the functions to 
 $N_{(m,m)}$). Composing with $L$, implies that the space of functions $$L_{f_0}:x\mapsto L(f_0)(w\begin{pmatrix} I_m & x \\  & I_m \end{pmatrix})$$ contains 
 $\sm_c(\mathcal{M}_m)$. We now fix $h_0$ in $\sm_c(P_{(m,m)} \backslash P_{(m,m)}wN_{(m,m)}, \delta_{P_{(m,m)}}^{-1/2} \D\otimes  \D^\vee)$, such that $L_{h_0}$ extends the characteristic function of $\mathcal{M}_m(\p^k)$, for $k$ such that $\p^k\subset Ker(\theta)$. One has, for good normalisations of measures, $\l_s(h_s)=1$ for all $Re(s)>m/2$. Finally, by \cite{JR}, it is know that $Hom_{S_n}(\Pi_s,\theta)$ is of dimension $1$ for all $s$ except maybe the finite number for which $\Pi_s$ is reducible. Bernstein's principle of analytic continuation of invariant linear forms (Corollary \ref{rationalinvariant})) then implies that $\Pi_s$ is $(S_n,\theta)$-distinguished for all $s$. Finally, 
 according to the proof of Proposition 3.1. of \cite{FJ}, which applies to any irreducible $\pi$ such that $Hom_{S_n}(\pi,\theta)\neq 0$ thanks to Lemma 6.1. of \cite{JR}, we see that $Hom_{S_n}(\Pi_s,\theta)\neq 0$ implies that $Hom_{H_n}(\Pi_s,\chi_\alpha)\neq 0$ when $\Pi_s$ is irreducible.
\end{proof}

\begin{rem}
The step $Hom_{S_n}(\Pi_s,\theta)\neq 0$ implies that $Hom_{H_n}(\Pi_s,\chi_\alpha)\neq 0$ actually holds for any $s$ (see 6.2. of \cite{JR}), but we don't need it.\end{rem}

Now we state the main result of this section. For its proof, and only there, we will write $[|.|^{k-1}\rho,\dots,\rho]$ for the segment 
denoted everywhere else by $[\rho,\dots,|.|^{k-1}\rho]$, because this is more convenient to compute Jacquet modules of such representations. 

\begin{thm}\label{distgen} 
Let $\alpha$ be a character of $F^*$, with $Re(\alpha)\in[0,1/2]$. Let $\pi=\Delta_1 \times \dots \times \Delta_t$ be a generic representation of the group $G_n$, for $n$ even. It is 
$H_n$-distinguished if and only if if there is a reordering of the ${\Delta _i}$'s, and an integer $r$ between $1$ and $[t/2]$, such that
 $\Delta_{i+1}= \Delta_i^{\vee} $ for $i=1,3,..,2r-1$, and $\Delta_{i}$ is $(H_{n_i},\chi_\alpha)$-distinguished for $i > 2r$.\\
Let $\pi$ be a generic representation of the group $G_{n}$, for $n$ odd. It is 
$(H_n,\chi_{\alpha|.|^{-1/2}})$-distinguished if and only if it is of the form $\pi \times \alpha|.|^{-1/2}$, for $\pi$ a $(H_{n-1},\chi_\alpha)$-distinguished generic representation of $G_{n-1}$ such that $\pi \times \alpha|.|^{-1/2}$ is still generic (equivalently irreducible).
\end{thm}

\noindent First, according to \cite{GK}, the representation $\pi^\vee$ is isomorphic to 
$g\mapsto \pi(^t\!g^{-1})$, hence it is $(H,\chi)$-distinguished (for $\chi$ a character of $H$) if and only if $\pi$ is 
$(H,\chi^{-1})$-distinguished. We thus notice that the statement above yields a classification of $(H,\chi_\alpha)$-distinguished generic representations 
with $Re(\alpha)\in [-1/2,1/2]$. For convenience, we will say that an induced representation of the shape described in the statement of the theorem is $(H,\chi_\alpha)$-induced. First let's check that these representations are in fact distinguished. Indeed, if a representation $\pi$ of $G_n$ is $(H,\chi_\alpha)$-induced, it is $(H_n,\chi_\alpha)$-distinguished when $n$ is even, and $(H_n,\chi_\alpha\d_n^{-1/2})$-distinguished when $n$ is odd according to Propositions \ref{induced} and \ref{distinduites}.\\
We now prove Theorem \ref{distgen}. Let $\chi$ be equal to $\1_H$ when $n$ is even, and to $\d_n^{-1/2}$ when $n$ is odd.
Let $\Delta$ be the representation $\Delta_1\otimes\dots\otimes\Delta_t$ of $P$, by Bernstein-Zelevinsky's 
version of Mackey's theory (Theorem 5.2 of \cite{BZ}), the $H$-module $\pi$ has a factor series with factors the representations 
$ind_{u_s^{-1}Pu_s\cap H}^H((\delta_P^{1/2}\Delta)_{s})$ (with $(\delta_P^{1/2}\Delta)_{s} (x)=\delta_P^{1/2}\Delta(u_sxu_s^{-1})$) when $u_s$ describes $R(P\backslash G/H)$. Hence if $\pi$ is $(H,\chi)$-distinguished for a character of $H$, one of these representations admits a nonzero $(H,\chi)$-invariant linear form on its space. 
This implies that there is $u_s$ in $R(P\backslash G/H)$ such that the representation $ind_{P\cap u_sHu_s^{-1}} ^{u_sHu_s^{-1}}(\delta_P^{1/2}\Delta)$ admits a nonzero 
$u_sHu_s^{-1}$-invariant linear form on its space. Frobenius reciprocity law says that the vector space 
$Hom_{u_sHu_s^{-1}}(ind_{P\cap u_sHu_s^{-1}} ^{u_sHu_s^{-1}}(\delta_P^{1/2}\Delta),\chi\chi_\alpha)$ is isomorphic to the vector space $Hom_{P_s}(\delta_P^{1/2}\Delta,\delta_{P_s^{<\theta_s>}}\chi^s\chi_\alpha^s)$, as we already saw the equality $P\cap u_sHu_s^{-1}=P_s^{<\theta_s>}$.\\
 Hence there is on the space $V_{\Delta }$ of $\Delta$ a linear nonzero form $A_s$, such that for every $p$ in $P_s^{<\theta_s>}$ and for every $v$ in $V_{\Delta }$, one has $A_s(\Delta(p)v)=\frac{\delta_{P_s}^{<\theta_s>}\chi^s\chi_\alpha^s}{\delta_P^{1/2}}(p)A_s(v)$. We notice that the characters $\delta_P^{1/2}$, $\delta_{P_s}^{<\theta_s>}$ and $\chi^s\chi_\alpha^s$ are trivial 
on $N_s^{<\theta_s>}$. Let $n'$ belong to $ N'_s$, from Proposition \ref{util}, we can write $n'$ 
as a product $n_s n_0$, with $n_s$ in $N_s^{<\theta_s>}$, and $n_0$ in $N$. As $N$ is included in $Ker(\Delta)$, one has 
$A_s(\Delta(n')(v))=A_s(\Delta(n_s n_0)(v))=A_s(\Delta(n_s )(v))= A_s(v)$. Hence $A_s$ induces a nonzero linear form $L_s$ on the Jacquet module of $V_{\Delta}$ 
associated with $N'_s$, which is by definition the quotient of $V_{\Delta}$ by the subspace generated by the vectors $\D(n)v-v$, for $n\in N'_s$ and $v\in V_{\Delta}$. But we also know that $L_s(\Delta(m_s)v)=\frac{\delta_{P_s}^{<\theta_s>}\chi^s\chi_\alpha^s}{\delta_P^{1/2}}(m)L_s(v)$ for $m$ in $M_s^{<\theta_s>}$, which reads, according to 1. of Proposition \ref{modulus}: 
$L_s(\delta_{P'_s}^{-1/2}(m_s)\Delta(m_s)v)=\frac{\delta_{P_s}^{<\theta_s>}\chi^s\chi_\alpha^s}{\delta_{P_s}^{1/2}}L_s(v)$.\\
This says that the linear form $L_s$ is $(M_s^{<\theta_s>},\frac{\delta_{P_s}^{<\theta_s>}\chi^s\chi_\alpha^s}{\delta_{P_s}^{1/2}})$-distinguished on the normalised Jacquet module $r_{M_s,M}(\Delta)$ (as $M_s$ is also the standard Levi subgroup associated with $ N'_s$).\\
We are going to prove by induction on $n$, the following theorem, which, according to the preceding discussion, will imply Theorem 
\ref{distgen} (remembering that a generic representation is a commutative product of discrete series). If $\D$ is the segment 
$[|.|^{k-1}\rho,\dots,\rho]$, we will write $r(\D)$ for $Re(c_\rho)$, $l(\D)$ for $Re(c_{|.|^{k-1}\rho})$, and we will call 
$k$ the length of $\D$.

\begin{thm}
Let $n$ be a positive integer, let $\overline{n}=(n_1,\dots,n_t)$ be a partition of $n$, and let $s$ belong to $I(\overline{n})$. For each $n_i$, let $\D_i$ 
be a discrete series of $G_{n_i}$, and denote by $\D$ the tensor product $\D_1\otimes \dots \otimes \D_t$. Suppose morever that one has $r(\D_i)\geq r(\D_{i+1})$ for $i$ between 
$1$ and $t-1$, and that $\D_i$ is longer than or of equal length as $\D_{i+1}$ if $r(\D_i)=r(\D_{i+1})$. Finally, let $\alpha$ be a character with $Re(\alpha)\in[0,1/2]$ when $n$ is even, and $Re(\alpha)\in[-1/2,0]$ when $n$ is odd. In this situation, if the $M_s$-module $r_{M_s,M}(\Delta)$ is $(M_s^{<\theta_s>},\frac{\delta_{P_s}^{<\theta_s>}\chi_\alpha^s}{\delta_{P_s}^{1/2}})$-distinguished, then for each $i$ between $1$ and $t$, there is a unique $j_i$ such that $\D_i=\D_{i,j_i}$. Moreover $\pi=\D_1\times \dots \times \D_t$ is $(H,\chi_\alpha)$-induced when $n$ is even, whereas when $n$ is odd, for some $i_0$, we have $\D_{i_0}=\alpha$, and the product of the remaining $\D_i$'s is $(H,\chi_{\alpha|.|^{1/2}})$-induced.
\end{thm}

\begin{proof} We write each $\D_i$ under the the form $[\Delta_{i,1},\dots,\Delta_{i,t}]$ according to $s$, then the normalised Jacquet module 
$r_{M_s,M}(\D)$ equals $\D_{1,1}\otimes \D_{1,2} \otimes \dots \otimes \D_{t,t-1}\otimes \D_{t,t}$ by Proposition $9.5$ of \cite{Z}. 
We now make a few observations, which are consequences of Equality (\ref{modulusquotient}) and of the fact that $r_{M_s,M}(\Delta)$ is $(M_s^{<\theta_s>},\frac{\delta_{P_s}^{<\theta_s>}\chi_\alpha^s}{\delta_{P_s}^{1/2}})$-distinguished. We denote by $u_1,\dots,u_r$ the indices such that $n_{u_i,u_i}=1$. 
If $n$ is even, then $r$ is even as well, and there are $r/2$ indices $u_i$ such that $n_{u_i}=n_{u_i,u_i}^+$ and 
$r/2$ indices $u_i$ such that $n_{u_i}=n_{u_i,u_i}^-$. If $n$ is odd, then so is $r=2a+1$, and there are $a+1$ indices $u_i$ such that $n_{u_i,u_i}=n_{u_i,u_i}^+$ and $a$ indices $u_i$ such that $n_{u_i,u_i}=n_{u_i,u_i}^-$. Automatically, we have $\D_{u_i}=\D_{u_i,u_i}$, and we will use the notation 
$\chi_{u_i}$ to remind us it is a character.

\begin{description}
\item[1.] $\chi_{u_1}$ is equal to $\alpha|.|^{-1/2}$ if $n$ is even and $n_{u_1}=n_{u_1,u_1}^+$, to $\alpha^{-1}|.|^{-1/2}$ if $n$ is even and 
$n_{u_1}=n_{u_1,u_1}^-$, to $\alpha$ if $n$ is odd and $n_{u_1}=n_{u_1,u_1}^+$, and to $\alpha^{-1}|.|^{-1}$ if $n$ is odd and 
$n_{u_1}=n_{u_1,u_1}^-$. In the first case, we have $Re(\chi_{u_1,u_1})\leq 0$, in the second case we have $Re(\chi_{u_1,u_1})\leq -1/2$, 
in the third case we have $Re(\chi_{u_1,u_1})\leq 0$ and $Re(\chi_{u_1,u_1})\leq -1/2$ in the last case. 

\item[2.] Let $d$ be an integer between $1$ and $r-1$. If $n_{u_d}=n_{u_d,u_d}^+$, then $\chi_{u_d}$ is of the form $\alpha|.|^{m_d}$ for $m_d\in \frac{1}{2}\Z$. In this situation $\chi_{u_{d+1}}=\chi_{u_d}|.|^{-1}$ if 
$n_{u_{d+1}}=n_{u_{d+1},u_{d+1}}^+$, and $\chi_{u_{d+1}}=\alpha^{-1}|.|^{m_d}$ if 
$n_{u_{d+1}}=n_{u_{d+1},u_{d+1}}^-$. If $n_{u_d}=n_{u_d,u_d}^-$, then $\chi_{u_d}$ is of the form $\alpha^{-1}|.|^{m_d}$ for $m_d\in \frac{1}{2}\Z$. In this situation $\chi_{u_{d+1}}=\chi_{u_d}|.|^{-1}$ if 
$n_{u_{d+1}}=n_{u_{d+1},u_{d+1}}^-$, and $\chi_{u_{d+1}}=\alpha|.|^{m_d}$ if 
$n_{u_{d+1}}=n_{u_{d+1},u_{d+1}}^+$.

\item[3.] The observations 1. and 2. above imply that for all $i$ between $1$ and $r$, we have $Re(\chi_{u_i})\leq 0$ when $n$ is even, $\leq 1/2$ when $n$ is odd and $\leq -1/2$ if moreover $n_{u_1,u_1}=n_{u_1,u_1}^-$.

\item[4.] If $n_{i,j}\neq 0$ for some $i<j$. We denote by $n'$ the integer $n-2n_{i,j}$, by $\overline{n'}$ the partition 
of $n-2n_{i,j}$ equal to $(n_1, \dots,n_i-n_{i,j},\dots,n_j-n_{i,j},\dots, n_t)$, and by $s'$ the element of $I(\overline{n'})$ equal to 
$(n_{1,1},\dots, n_{i,j-1}, n_{i,j+1},\dots,n_{j,i-1}, n_{j,i+1},\dots,n_{t,t})$. In this situation, if we write an element $m$ of $M_s^{<\theta_s>}$ 
as $diag(h,m_{i,j},h',m_{j,i},h'')$, and denote by $m'$ the element $diag(h,h',h'')$, we have: 
\begin{equation}\label{modulusrec0}\frac{\d_{P_s^{<\theta_s>}}}{\delta_{P_s}^{1/2}}(m)=\frac{\d_{P_{s'}^{<\theta_{s'}>}}}{\delta_{P_{s'}}^{1/2}}(m')
.\end{equation}

\item[5.] If $n_{1,1}>1$, then $n_{1,1}^+=n_{1,1}^-$. We denote by $n'$ the integer $n-n_{1,1}$, by $\overline{n'}$ the partition 
of $n-n_{1,1}$ equal to $(n_1-n_{1,1}, n_2,\dots, n_t)$, and by $s'$ the element of $I(\overline{n'})$ equal to 
$(n_{1,2},\dots, n_{1,t},\dots,n_{t,1},\dots,n_{t,t})$. We have for $m=diag(m_{1,1},m')\in M_s^{<\theta_s>}$: 
\begin{equation}\label{modulusrec}\frac{\d_{P_s^{<\theta_s>}}}{\delta_{P_s}^{1/2}}(m)=\frac{\d_{P_{s'}^{<\theta_{s'}>}}}{\delta_{P_{s'}}^{1/2}}(m')
.\end{equation}

\item[6.] If $n$ is odd and $n_{1,1}=1=n_{1,1}^+$, we denote by $n'$ the integer $n-n_{1,1}$, by $\overline{n'}$ the partition 
of $n-n_{1,1}$ equal to $(n_1-n_{1,1}, n_2,\dots, n_t)$, and by $s'$ the element of $I(\overline{n'})$ equal to 
$(n_{1,2},\dots, n_{1,t},\dots,n_{t,1},\dots,n_{t,t})$. We have for $m=diag(m_{1,1},m')\in M_s^{<\theta_s>}$: \begin{equation}\label{modulusrec1}\frac{\delta_{P_s}^{<\theta_s>}\chi_\alpha^s}{\delta_{P_s}^{1/2}}(m)=\alpha(m_{1,1})\frac{\delta_{P_{s'}}^{<\theta_{s'}>}\chi_{\alpha|.|^{1/2}}^{s'}}{\delta_{P_{s'}}^{1/2}}(m').\end{equation}\\

\end{description}

We now start the induction, the case $n=1$ is obvious. We now suppose that the statement of the theorem is satisfied for all positive integers $\leq n-1$ with $n\geq 2$.\\

A) If there is $j>1$ such that $\D_{1,j}$ is not empty. We take the smallest $j>1$, which we denote $l$, such that $\D_{1,l}$ is non empty. In particular,  
$\D_1$ is of either of the form $\D_1=[\D_{1,l},\D_l'']$, or of the form $\D_1=[\D_{1,1},\D_{1,l},\D_1'']$ if $\D_{1,1}$ is non empty. We write $\D_l=[\D_{l,1},\D_l'']$.\\

A.1) First we treat the case "$\D_{1,1}$ non empty". We set $\D_1'=\D_{1,1}$, $\D_2'=\D_1''$, 
$\D_j'=\D_{j-1}$ for $j>2$ different from $l+1$, and $\D_{l+1}'=\D_l''$. In this case, we have $r(\D_i')\geq r(\D_{i+1}')$ for $i$ between 
$1$ and $t$, and we can apply our induction hypothesis to $\D'$ thanks to Equality (\ref{modulusrec0}). This first implies 
that $\D_1''$ is empty, because otherwise it would be equal to $\D_{1,p}$ for some $p>l$, and $\D_p$ would be equal to 
$\D_{p,k}=\D_{k,p}^\vee$, which would imply that $r(\D_p)>r(\D_l)$, and this is absurd. Hence $\D_1=[\D_{1,1},\D_{1,l}]$. It also implies that $\D_l''$ is 
either empty, or equal to $\D_{l,q}$, for some $q> 1$. We thus have three subcases.\\

A.1.1) We suppose that $\D_l''$ is empty. If $n_{1,1}>1$, then $\D_{1,1}$'s central character would be trivial, and $\D_l=\D_{1,l}^\vee$ 
would satisfy $r(\D_l)>r(\D_1)$, and this is absurd. Thus, we have $\D_{1,1}=\chi_1$, for $\chi_1$ a character of $F^*$. In this case, we have observed in 3. above that $Re(\chi_1)$ is $\leq 1/2$. This implies that $l(\D_{1,l})$ is negative, and in turn that $r(\D_l)=r(\D_{l,1})$ is positive, hence 
$>l(\D_{1,l})\geq r(\D_1)$, and this is absurd as well. The case A.1.1) doesn't occur.\\

A.1.2) We now suppose $\D_l''=\D_{l,l}$. If $n_{l,l}>1$, then $\D_{l,l}$'s central character would be trivial, thus we would have $r(\D_{l,1})>0$, which would imply $l(\D_{1,l})<0$, and in turn $r(\D_1)<l(\D_{1,l})<0<r(\D_{l,1})\leq r(\D_l)$, which is absurd. Hence $\D_{l,l}=\chi_l$ is a character of $F^*$, so 
$\D_l=[\D_{l,1},\chi_l]$, and $\D_1=[\D_{1,1},\D_{1,l}]$. If $n_{1,1}>1$, then $r(\D_{1,1})\leq -1/2$ (as $\D_1$'s cuspidal support must consist of characters of 
$F^*$, as it is the case of $\D_{1,l}$ because $\chi_l$ is a character), hence $l(\D_{1,l})\leq -3/2$, and 
$r(\D_{l,1})\geq 3/2$. This implies the inequalities $r(\D_l)=Re(\chi_l)\geq 1/2> 0 > r(\D_{1,l})=r(\D_{1})$, which is absurd. Hence $\D_{1,1}=\chi_1$ is a character of 
$F^*$, and $\D_1=[\chi_1,\D_{1,l}]$, whereas $\D_l=[\D_{l,1},\chi_l]$. As $Re(\chi_1)\leq 0$ thanks to Observation 1., we thus have 
$l(\D_{1,l})\leq -1/2$, hence $r(\D_{l,1})\geq 1/2$, and $Re(\chi_l)=r(\D_l)\geq 0$. This is absurd as $r(\D_1)\leq l(\D_{1,l})\leq -1/2$.\\

A.1.3) We finally suppose $\D_l''=\D_{l,q}$, with $1<q$ and $q\neq l$, i.e. $\D_l=[\D_{l,1},\D_{l,q}]$. In this case $\D_q=\D_{q,l}$ by induction, but this is absurd as it would imply the inequality $r(\D_q)=r(\D_{q,l})>r(\D_{1,l})=r(\D_1)$.\\

Hence, the conclusion of the case A.1) is that $\D_{1,1}$ is empty. We are thus automatically in the following case A.2).\\

A.2) $\D_1=[\D_{1,l},\D_1'']$. We recall that $\D_l=[\D_{l,1},\D_l'']$. By induction, as in the first case, we see that 
either $\D_1''=\D_{1,p}$ for some $p>l$, or empty, and $\D_l''=\D_{l,q}$ for some $q>1$ or it is empty. If $\D_1''$ was not empty, 
we would have $\D_p=\D_{p,1}$, which would contradict $r(\D_l)> r(\D_p)$. Thus $\D_1=\D_{1,l}$. Now if $\D_l''$ 
was not empty, then we would have $\D_q=\D_{q,l}$, and this would imply $r(\D_q)=r(\D_{q,l})>r(\D_{1,l})=r(\D_1)$, a contradiction.\\ 

Hence, in case A), we have $\D_1=\D_{1,l}$, and $\D_l=\D_{l,1}$, and $\D$ is of the expected form by induction, thanks to Equality (\ref{modulusrec0}).\\

B) This is the remaining case $\D_1=\D_{1,1}$, we have five sub-cases.\\

B.1) If $n_{1,1}>1$, then $\D_1$ is $(H,\chi_\alpha)$-distinguished. Using notations of Observation 5., and setting $\D'=\D_2 \otimes \dots \otimes \D_t$, then $r_{M_{s'},M_{\overline{n'}}}(\Delta')$ is $({M}_{s'}^{<\theta_{s'}>},\frac{\delta_{P_{s'}}^{<\theta_{s'}>}\chi_\alpha^{s'}}{\delta_{P_{s'}}^{1/2}})$-distinguished according to 
Equality (\ref{modulusrec}). We conclude by induction that $\D$ is of the expected form.\\

B.2) The integer $n$ is even, and $n_{1,1}=n_{1,1}^+$. In this case $\chi_1=\alpha|.|^{-1/2}$, and by Observation 2., we know that all 
$\chi_{u_i}$'s are of real part $\leq 0$. Now their product is trivial as $r_{M_s,M}(\D)$ is $(M_s^{<\theta_s>},\frac{\delta_{P_s}^{<\theta_s>}\chi_\alpha^s}{\delta_{P_s}^{1/2}})$-distinguished, hence trivial on $F^*$ diagonally embedded 
in the center of $M_s^{<\theta_s>}$. We deduce that 
they all have a real part equal to zero. In particular we have $Re(\alpha)=1/2$, but by Observation 2., $\chi_{u_2,u_2}$ is either equal to $\alpha^{-1}|.|^{-1/2}$, 
or equal to $\chi_1|.|^{-1}$, and in both cases, we have $Re(\chi_{u_2,u_2})<0$, which is absurd.\\

B.3) The integer $n$ is even, and $n_{1,1}=n_{1,1}^-$. In this case $\chi_1=\alpha^{-1}|.|^{-1/2}$, but Observation 2. implies again that the 
$\chi_{u_i}$'s are of real part $\leq 0$, and using the same argument as above, they must have real parts equal to $0$. This contradicts 
the fact that we have $Re(\chi_1)\leq -1/2$.\\

B.4) The integer $n$ is odd, and $n_{1,1}=n_{1,1}^+$. In this case, thanks to Equality (\ref{modulusrec1}), we obtain that $\chi_1=\alpha$, and 
that $r_{M_{s'},M_{\overline{n'}}}(\D')$ is $(M_{s'}^{<\theta_{s'}>},\frac{\delta_{P_{s'}}^{<\theta_{s'}>}\chi_{\alpha|.|^{1/2}}^{s'}}{\delta_{P_{s'}}^{1/2}})$-
distinguished for $\D'$ equal to $\D_2\times \dots \times\D_t$, we conclude by induction that $\D$ is of the expected form.\\

B.5) The integer $n$ is odd, and $n_{1,1}=n_{1,1}^-$. In this case $\chi_1=\alpha^{-1}|.|^{-1}$ is of real part $\leq -1/2$, and $r$ is $\geq 3$. Observation 3. 
implies that the real part of all characters $\chi_i$ is $\leq -1/2$, hence the product of these characters is of real part $\leq -3/2$. This contradicts the fact 
that it should be equal to $\alpha$, as as $r_{M_s,M}(\D)$ is 
$(M_s^{<\theta_s>},\frac{\delta_{P_s}^{<\theta_s>}\chi_\alpha^s}{\delta_{P_s}^{1/2}})$-distinguished.
\end{proof}

We have the following corollary to Theorem \ref{distgen}, in terms of Galois parameter, which was brought to our attention by Wee Teck Gan.

\begin{cor}
Let $\pi$ be a generic representation of $G_{2n}$, and $\phi(\pi)$ its Langlands' parameter. Then $\pi$ is $H_{2n}$-distinguished if and only if $\phi(\pi)$ preserves a non-degenerate symplectic form on its space.
\end{cor}
\begin{proof}
If $\pi$ is a unitary discrete series $\D=St_k(\rho)$, according to the Theorem \ref{discr} of this paper (which is a compilation of results from \cite{K}, 
\cite{KR}, and \cite{S}), the representation $\D$ is 
$H_{n}$-distinguished if and only if either $k$ is odd and $L(\rho,\wedge^2,s)$ has a pole at zero, or $k$ is even and $L(s,Sym^2,\rho)$ has a pole at zero. We 
then notice that $Sp(k)$ is symplectic when $k$ is even, and orthoganal when $k$ is odd, whereas 
$\rho$ preserves a symplectic form if and only if $L(\rho,\wedge^2,s)$ has a pole at zero, and a symmetric bilinear form if and only if $L(s,Sym^2,\rho)$ has a pole at zero. Hence $\D$ is $H_n$-distinguished if and only if $\phi(\D)$ preserves a nonzero symplectic form on its space, which is necessarily nondegenerate as $\phi(\D)$ is irreducible. Now we notice that if $(\rho,V_\rho)$ is a representation of $W'_F$, the representation $\rho\oplus \rho^\vee$ of $W'_F$ preserves the nondegenerate symplectic form $(v+v',w+w')\mapsto w'(v)- v'(w)$. This implies, according to Theorem \ref{distgen} and the unitary discrete series case, that $H_{2n}$-distinguished generic representations have a Galois parameter which fixes a nondegenerate symplectic form.\\
In general, we write $\pi=\D_1\times \dots \times \D_t$, and $\phi_i$ for the Langlands' parameter of $\D_i$. We prove by induction on $t$ that if 
$(\phi(\pi),V)=\oplus_{i=1}^t (\phi_i,V_i)$ fixes a non degenerate symplectic form on its space, then $\pi$ is $H$-induced.\\ 
If $t=1$, $\pi$ is just the discrete series $\D_1$, which is unitary because the central character of $\D_1$ is trivial (as it is the determinant of 
$\phi(\D_1)$), and we already saw that $\pi$ is distinguished.\\
$t\rightarrow t+1$: we write $\phi=\oplus_{i=1}^t \phi_i$, hence $\phi\oplus \phi_{t+1}$ fixes a nondegenerate symplectic form 
$B$. If the sum $\phi\oplus \phi_{t+1}$ is orthogonal for $B$, then $B_{|V_\rho\times V_\rho}$ and $B_{|V_{t+1}\times V_{t+1}}$ 
are nondegenerate and symplectic, hence $V_\rho$ and $V_{t+1}$ are even dimensional, and we conclude by induction. Otherwise, there is $\phi_i$, such that $B_{|V_{t+1}\times V_i}\neq 0$. There is no harm to suppose that $i=t$. But then $V_{t+1}\simeq V_t^\vee$ and 
$B_{|(V_t\oplus V_{t+1})\times (V_t\oplus V_{t+1}})$ is symplectic non degenerate, hence $V=(V_t\oplus V_{t+1})\oplus (V_t\oplus V_{t+1})^\perp$, and we conclude by induction.
\end{proof}

The local Langlands functorial transfer from $SO(2n+1,F)$ to $G_{2n}$ has been made explicit in \cite{Ji-So} for generic representations of $SO(2n+1,F)$. It is easy to 
check from Theorem B of [loc. cit.], Theorem \ref{distgen} above, and Theorem 6.1. of \cite{M2012.2}, that 
tempered representations of $G_{2n}$ which are $H_{2n}$-distinguished are the image by the functorial transfer of the 
generic tempered representations of $SO(2n+1,F)$.

\section{L factors and exceptional poles}\label{Lfactors}

\subsection{Reminder on local L functions of pairs}\label{sctLpaires}

In this paragraph, we recall results from \cite{JPS} and \cite{CP} about $L$-functions of pairs. If $P$ and $Q$ are two elements of $\C[q^{\pm u}]$ for $u\in \mathcal{D}^{t}$, we say that $1/P$ divides $1/Q$ if $P$ divides $Q$.

\begin{prop}\label{dfLpaires}
Let $\pi$ and $\pi'$ be two representations of Whittaker type of $G_n$, let $W$ belong to $W(\pi,\theta)$, $W'$ 
belong to $W(\pi',\theta^{-1})$, and $\phi$ belong to $\sm_c(F^n)$. Then the integrals 
$$ \Psi(W,W',\phi,s)=\int_{N_n\backslash G_n} W(g)W'(g)\phi(L_k(g))|g|^s dg $$ and
$$ \Psi_{(0)}(W,W',s)=\int_{N_{n-1}\backslash G_{n-1}} W(diag(g,1))W'(diag(g,1))|g|^{s-1} dg $$ converge absolutely for $Re(s)$ large, and extend to $\C$ as elements of $\C(q^{-s})$.\\
If $\pi''$ is a representation of $G_m$ for $m<n$, and $W''$ belongs to $W(\pi'',\theta^{-1})$, the same properties are true for for the integrals 
$$\Psi(W,W'',s)=\int_{N_m\backslash G_m} W(diag(g,I_{n-m})W''(g)|g|^{s-(n-m)/2} dg.$$
Moreover, when $W$, $W'$, $W''$ and $\phi$ vary in their respective spaces, the integrals $\Psi(W,W',\phi,s)$, $\Psi(W,W'',\phi,s)$ and $\Psi_{(0)}(W,W',s)$ span respectively a fractional ideal $I(\pi,\pi')$, $I(\pi,\pi'')$ and $I_{(0)}(\pi,\pi')$ of $\C[q^{\pm s}]$. The fractional ideal $I(\pi,\pi')$ (resp. $I(\pi,\pi'')$, resp. $I_{(0)}(\pi,\pi')$)
admits a unique generator $L(\pi,\pi',s)$ (resp. $L(\pi,\pi'',s)$, resp. $L_{(0)}(\pi,\pi',s)$) which is an Euler factor. If 
$n<m$, the factor $L(\pi,\pi'',s)$ is by definition $L(\pi'',\pi,s)$.
\end{prop}

We now fix a Haar measure $dx$ on $F$, which is self-dual with respect to the character $\theta$. If $m$ is a positive integer, and
$\phi$ belongs to $\sm_c(F^m)$, we denote by
$$\widehat{\phi}^\theta:(x_1,\dots,x_m)\mapsto \int_{t\in F^m}\phi(t)\theta(-\sum_{i=1}^m t_ix_i) dt_1\dots dt_m$$ the Fourier transform of $\phi$. We recall the functional equation of $L$-functions of pairs.

\begin{prop}\label{eqfctpaires}
Let $\pi$, $\pi'$ and $\pi''$ be respectively representations of Whittaker type of $G_n$, $G_n$ and $G_m$ for $m<n$. Let $W$ belong to $W(\pi,\theta)$, $W'$ 
belong to $W(\pi',\theta^{-1})$, $W''$ belong to $W(\pi'',\theta^{-1})$ and $\phi$ belong to $\sm_c(F^n)$. Then there are unit $\e(\pi,\pi',\theta,s)$ and 
$\e(\pi,\pi'',\theta,s)$ of $\C[q^{\pm s}]$, such that the following functional equation is satisfied:
$$\frac{\Psi(\tilde{W},\tilde{W'},\widehat{\phi}^\theta,1-s)}{L(\tilde{\pi},\tilde{\pi'},1-s)}=\e(\pi,\pi',\theta,s)
\frac{\Psi(W,W',\phi,s)}{L(\pi,\pi',s)}$$ 
and 
$$\frac{\Psi(\widetilde{W},\widetilde{W''},1-s)}{L(\widetilde{\pi},\widetilde{\pi''},1-s)}=\e(\pi,\pi'',\theta,s)
\frac{\Psi(W,W'',s)}{L(\pi,\pi'',s)}$$
\end{prop}

With the notations above, we write 
$$\gamma(\pi,\pi',\theta,s)= \frac{\e(\pi,\pi',\theta,s)L(\widetilde{\pi},\widetilde{\pi'},1-s)}{L(\pi,\pi',s)}$$ and 
$$\gamma(\pi,\pi'',\theta,s)= \frac{\e(\pi,\pi'',\theta,s)L(\widetilde{\pi},\widetilde{\pi''},1-s)}{L(\pi,\pi'',s)}$$ in $\C(\mathcal{D})$. 
We recall the definition of an exceptional pole:

\begin{df}\label{defpolexpaires}
Let $\pi$ and $\pi'$ be representations of $G_n$ of Whittaker type, and let $s_0$ be a pole of order $d$ of $L(\pi,\pi',s)$. For $W$  
in $W(\pi,\theta)$, $W'$ in $W(\pi,\theta^{-1})$, and $\phi$ in $\sm_c(F^{n})$, we write the Laurent expansion $$\Psi(W,W',\phi,s)=\frac{T(W,W',\phi,s)}{(1-q^{s_0-s})^d}+ \ terms \ of \ higher \ degree.$$
We say that $s_0$ is an exceptional pole of $L(\pi,\pi',s)$ if the trilinear form $T$ vanishes on the space $W(\pi,\theta)\times 
W(\pi,\theta) \times \sm_{c,0}(F^{n})$, i.e. if 
it is of the form $T:(W,\phi)\mapsto B(W,W')\phi(0)$, where $B$ is a nonzero bilinear form. By definition of the integrals $\Psi(W,W',\phi,s)$, 
the linear form $B$ is a nonzero element of the space $Hom_{G_n}(\pi\otimes \pi',|.|^{-s_0})$.
\end{df}

The exceptional poles of $L(\pi,\pi',s)$ are actually those of $L(\pi,\pi',s)/L_{(0)}(\pi,\pi',s)$. 

\begin{prop}\label{radexpaires}
Let $\pi$ and $\pi'$ be two representations of Whittaker type of $G_n$. 
The ideal $I_{(0)}(\pi,\pi')$ is spanned by the integrals $\Psi(W,W',\phi,s)$ for $W$ in $W(\pi,\theta)$, $W'$ in $W(\pi',\theta^{-1})$ and 
$\phi$ in $\sm_{c,0}(F^n)$. The factor $L_{(0)}(\pi,\pi',s)$ thus divides $L(\pi,\pi',s)$, and the quotient $$L_{rad(ex)}(\pi,\pi',s)=L(\pi,\pi',s)/L_{(0)}(\pi,\pi',s)$$ has simple poles, which are exactly the exceptional poles of $L(\pi,\pi',s)$.
\end{prop}
\begin{proof}
Writing the formula, for $Re(s)$ large and $\phi$ in $\sm_{c}(F^n)$:
\begin{eqnarray}\label{iwazpaires} \Psi(W,W',\phi,s) \end{eqnarray}
 $$= \int_{K_n}\int_{N_n\backslash P_n} W(pk)W'(pk)|p|^{s-1}dp\int_{F^*}\phi(t\eta_n(k))c_{\pi}(tk)c_{\pi'}(tk)|t|^{ns}d^*t dk,$$
and using smoothnes of $\phi$, $W$, and $W'$, we see that $\Psi(W,W',\phi,s)$ is a finite sum of integrals of the form $P_i(q^{\pm s})\Psi_{(0)}(W_i,W_i',s)$ if $\phi$ vanishes at zero, hence it belongs to $I_{(0)}(\pi,\pi',s)$. Conversely, take an integral $\Psi_{(0)}(W,W',s)$, and $r\geq 1$ large enough for $W$ and $W'$ to be right 
invariant under $K_{n,r}$. Then, if $\phi$ in $\sm_{c,0}(F^n)$ is the characteristic function of $\eta_n(K_{n,r})$, the integral 
$\Psi(W,W',\phi,s)$ reduces to a positive multiple of $\Psi_{(0)}(W,W',s)$ according to Equation (\ref{iwazpaires}). This proves that $I_{(0)}(\pi,\pi')$ is spanned by the integrals $\Psi(W,W',\phi,s)$ for $W$ in $W(\pi,\theta)$, $W'$ in $W(\pi',\theta^{-1})$ and 
$\phi$ in $\sm_{c,0}(F^n)$. It implies that $L_{(0)}(\pi,\pi',s)$ divides $L(\pi,\pi',s)$. Finally, according to 
Equation (\ref{iwazpaires}), any integral $\Psi(W,W',\phi,s)$ is a finite sum of integrals $\Psi_{(0)}(W_i,W_i',s)\int_{F^*}\phi_i(tk)c_{\pi}(t)c_{\pi'}(t)|t|^{ns}d^*t$, hence $L(\pi,\pi',s)$ divides the product $L_{(0)}(\pi,\pi',s)L(c_{\pi}c_{\pi'},ns)$, and the quotient $L_{rad(ex)}(\pi,\pi',s)$ divides the factor $L(c_{\pi}c_{\pi'},ns)$, it thus has simple poles.
\end{proof}

If $\pi$ and $\pi'$ are both representations of $G_n$ of Whittaker type, the factor $L_{ex}(\pi,\pi',s)$ is by definition the product of the factors $1/(1-q^{s_0-s})^d$ for $s_0$ an exceptional pole of $L(\pi,\pi',s)$, and $d$ its order in $L(\pi,\pi',s)$. If $\pi''$ is a representation of Whittaker type of 
$G_m$ for $m<n$, we say that the factor $L(\pi,\pi'',s)$ has no exceptional poles. Hence, in this case, we write $L_{rad(ex)}(\pi,\pi'')=L_{ex}(\pi,\pi',s)=1$, and 
we define $L_{(0)}(\pi,\pi'',s)$ to be $L(\pi,\pi'',s)$. We recall one of the main results of \cite{CP} (namely Theorems 2.1. and 2.2, or more precisely a consequence of their proofs), which will be used later.

\begin{prop}\label{factLpaires}
Let $\pi$ and $\pi'$ be two generic representations of $G_n$ and $G_m$ respectively ($m\leq n$), both with completely reducible derivatives. 
For $l$ in the set $\{0,\dots,n\}$ and $k$ in the set $\{0,\dots,m\}$, write $$\pi^{(n-k)}=\oplus_{i_k=1}^{e_k} \rho_{i_k}$$ and $$\pi'^{(m-l)}=\oplus_{j_l=1}^{f_l} \rho'_{i_l}$$ for their decomposition in simle components.
Then the factor $L_{(0)}(\pi,\pi',s)$, is equal to the lcm of the factors $L_{ex}(\rho_{i_k},\rho'_{j_l},s)$, for $(k,l)\in\{0\dots,n\}\times \{0\dots,m\}$ with $k+l<n+m$, and $(i_k,j_l)\in\{1,\dots,e_k\}\times \{1,\dots,f_l\}$.
\end{prop}

By the explicit expression of $L$-functions of pairs for discrete series, given in Theorem 8.2. of \cite{JPS}, we have the following result.

\begin{prop}\label{divisionpaires}
Let $\D$ and $\D'$ be two discrete series of $G_n$ and $G_m$ respectively, with $m\leq n$. For any $k$ and $k'$, such that $\D^{(k)}$ and $\D^{(k')}$ are nonzero, the factor  $L(\D^{(k)},\D'^{(k')},s)$ divides $L(\D,\D',s)$.
\end{prop}

We finally recall the following result, whose proof is mentioned in \cite{M2009}.

\begin{prop}\label{exceptionalpaires}
Let $\pi$ and $\pi'$ be generic representations of $G_n$, then the $L$-factor of pairs $L(\pi,\pi',s)$ has an exceptional pole at $s_0$ if 
and only if $\pi'\simeq |.|^{-s_0}\pi^\vee$.
\end{prop}
\begin{proof}
We only need to show that if $\pi'\simeq |.|^{-s_0}\pi^\vee$, then the factor $L(\pi,\pi',s)$ has an exceptional pole at $s_0$. Replacing $\pi$ by 
$|.|^{s_0}\pi$, it is sufficient to prove the statement for $s_0=0$. So suppose that $\pi'=\pi^\vee$. From Equation (\ref{iwazpaires}), for $Re(s)<<0$, 
one has:

$$\Psi(\tilde{W},\tilde{W'},\widehat{\phi}^\theta,1-s)$$
 \begin{equation} =\label{sum} \int_{K_n}
 \int_{N_n\backslash P_n} \tilde{W}(pk)\tilde{W'}(pk)|det(p)|_F^{-s}dp \int_{F^*}
 \widehat{\phi}^\theta(t \eta_n(k))|t|_F^{n(1-s)} d^*t dk.\end{equation}

Now we notice that $B_{\pi^{\vee},\pi,1-s}:(\tilde{W},\tilde{W'})\mapsto \Psi_{(0)}(\tilde{W},\tilde{W'},1-s)/L_(\pi^\vee,\pi,1-s)$ is an element of $Hom_{P_n}(W(\pi^\vee,\theta^{-1}) \otimes W(\pi,\theta),|.|^s)$. We can write:

$$\Psi(\tilde{W},\tilde{W'},\widehat{\phi},1-s)/L(\pi^{\vee},\pi,1-s)$$
\begin{equation} \label{summ}=\int_{K_n}{B}_{\pi^{\vee},\pi,1-s}(\pi^{\vee}(k)\tilde{W},
\pi(k)\tilde{W'})\int_{F^*}
 \widehat{\phi}^\theta(t\eta_n(k))|t|^{n(1-s)} d^*t dk.\end{equation}

The second member of the equality is actually a finite sum: $$\sum_i \lambda
 _i{B}_{\pi^{\vee},\pi,1-s}(\pi^{\vee}(k_i)\tilde{W}),\pi(k_i)\tilde{W'})\int_{F^*}
 \widehat{\phi}(t\eta_n(k_i))|t|^{n(1-s)} d^*t,$$ 
where the $\lambda_i$'s are positive constants and the $k_i$'s are elements of $K_n$ independent of $s$.\\
Notice that for
 $Re(s)<1$, the integral $$\int_{F^*} \widehat{\phi}^\theta(t\eta_n(k_i))|t|^{n(1-s)} d^* t$$ is absolutely convergent, and defines a holomorphic function.
So we have an equality (Equality (\ref{summ})) of analytic functions (actually of polynomials in
 $q^{-s}$), hence it is true for all $s$ with $Re(s)<1$. For $s=0$, we get:

 $$\Psi(\tilde{W},\tilde{W'},\widehat{\phi}^\theta,1)/L(\pi^{\vee},\pi,1)=\int_{K_n}{B}_{\pi^{\vee},\pi,1}
 (\pi^{\vee}(k)\tilde{W},\pi(k)\tilde{W'})\int_{F^*} \widehat{\phi}^\theta(t\eta_n(k))|t|^nd^* tdk.$$

But $B_{\pi^{\vee},\pi,1}$ is a $P_n$-invariant linear form on
 $W(\pi^{\vee},\theta^{-1})\times W(\pi,\theta)$, and it follows from Theorem A of \cite{B} that it is actually $G_n$-invariant.\\
Finally $$\Psi(\tilde{W},\tilde{W'},\widehat{\phi}^\theta,1)/L(\pi^{\vee},\pi,1)=
 {B}_{\pi^{\vee},\pi,1}(\tilde{W},\tilde{W'})\int_{K_n}\int_{F^*}
 \widehat{\phi}^\theta(t\eta_n(k))|t|^n d^*t dk$$ which is equal for a good normalisation $dk$, to:\\
$${B}_{\pi^{\vee},\pi,1}(\tilde{W},\tilde{W'})\int_{P_n\backslash
 G_n}\widehat{\phi}(\eta_n(g))d_{\mu}g$$ where $d_{\mu}$ is up to
 scalar the unique $|det(\ )|^{-1}$ invariant measure on
 $P_n\backslash
 G_n$.
 But as we have $$\int_{P_n\backslash
 G_n}\widehat{\phi}^\theta(\eta_n(g))d_{\mu}g=
 \int_{F^n}\widehat{\phi}^\theta(x)dx=\phi(0),$$
we deduce from the functional equation the equality 
 $\Psi(W,W',\phi,0)/L(\pi,\pi^\vee,0)=0$ whenever $\phi(0)=0$. As one can choose $W$, $W'$, and $\phi$ vanishing at zero, such that
  $\Psi(W,W',\phi,s)$ is the constant function equal to $1$ (see the proof of Theorem 2.7. in \cite{JPS}),
 the factor $L(\pi,\pi',s)$ has a pole at zero, which must be
 exceptional.
\end{proof}

\subsection{The Rankin-Selberg integrals}\label{The Rankin-Selberg integrals}

In this section, we define the Bump-Friedberg integrals attached to representations of $G_n$ of Whittaker type, and prove their basic properties. They will be useful in next section when one needs to know 
how the Rankin-Selberg integrals vary when the representation is ``deformed`` by a parameter $u\in \mathcal{D}^t$. The notations 
are different from \cite{M2012.3}, here we denote by $L^{lin}(\pi,\chi,s)$ the factor denoted by $L^{lin}(\pi,\chi\d_n^{1/2},s)$ in \cite{M2012.3}. The reason for this is that with this definition, the occurence of an exceptional pole at zero of $L^{lin}(\pi,s)$ is related to the occurence of linear periods for $\pi$. 
As we recall every definition, this should not create any confusion.\\

Let $\pi$ be a representation of Whittaker type of $G_n$, we consider the following integrals.

\begin{df}\label{integrals}
Let $W$ belong to $W(\pi,\theta)$, and $\phi$ belong to $\sm_c(F^{[(n+1)/2]})$, $s$ be a complex number, and $\alpha$ a character of $F^*$. We define formally the following integrals:
 $$\Psi(W,\phi,\chi_\alpha,s)=\int_{N_n^\sigma\backslash H_n}W(h)\phi(l_n(h))\chi_\alpha(h)\chi_{n}(h)^{1/2}|h|^sdh$$ and 
$$\Psi_{(0)}(W,\chi_\alpha,s)=\int_{N_{n}^\sigma\backslash {P_n}^\sigma}W(h)\chi_\alpha(h)\mu_{n}^{1/2}(h)|h|^{s-1/2}dh$$
$$=\int_{N_{n-1}^\sigma\backslash {G_{n-1}}^\sigma}W(h)\chi_\alpha(h)\chi_{n-1}^{1/2}(h)|h|^{s-1/2}dh,$$
Where $l_n(h(h_1,h_2))$ is the bottom row of $h_2$ when $n$ is even, and the bottom row of $h_1$ when $n$ is odd.
\end{df}

The following proposition gives a meaning to these integrals.

\begin{prop}\label{CV}
 Let $\pi$ be a representation of $G_n$ of Whittaker type, and $\alpha$ a character of $F^*$. For every $k\in \{1,\dots,n\}$, let $(c_{k,i_k})_{i_k}$ be the $(n-k)$-exponents of $\pi$. Let $W$ belong to 
$W(\pi,\theta)$, $\phi$ belong to $\sm_c(F^{[(n+1)/2]})$, and $\e_k=0$ when $k$ is even and $1$ when $k$ is odd. If for all $k$ in $\{1,\dots,n\}$ (resp. in $\{1,\dots,n-1\}$), we have $Re(s)>-[Re(c_{k,i_k})+\e_k(Re(\alpha)+1/2)]/k$, the integral $\Psi(W,\phi,\chi_\alpha,s)$ 
(resp. $\Psi_{(0)}(W,\chi_\alpha,s)$) converge absolutely. They are hence  
holomorphic on the half planes defiend by this inequality, and admit meromorphic extension to $\C$ as elements of $\C(q^{-s})$. 
\end{prop}
\begin{proof}
Let $B_n^\sigma$ be the standard Borel subgroup of $H_n$. 
The integrals $$\Psi(W,\phi,\chi_n^{1/2}\chi_\alpha,s)=\int_{N_n^\sigma\backslash H_n}W(h)\phi(l_n(h))\chi_n^{1/2}(h)\chi_\alpha(h)|h|^sdh$$ will converge absolutely 
as soon as the integrals $$\int_{A_n}W(a)\phi(l_n(a))\chi_n^{1/2}(a)\chi_n^{1/2}(a)\chi_\alpha(a)|a|^s\d_{B_n^\sigma}^{-1}(a)d^*a$$ will do so for any $W\in W(\pi,\theta)$, and any $\phi\in \sm_c(F^n)$. But, according to Proposition \ref{DL}, and writing $z$ as $z_1\dots z_n$, 
this will be the case if the integrals of the form 
 $$\int_{A_n} \!\!\!\!\!\! c_{\pi}(t(z_n))\phi(l_n(z_n))\!\!\prod_{k=0}^{n-1} c_{k,i_k}(z_k)|z_k|^{(n-k)/2}v(t(z_k))^{m_k}\phi_k(t(z_k))
(\chi_n^{1/2}\chi_\alpha)(z)|z|^s\d_{B_n^\sigma}^{-1}(z)d^*z$$  converge absoultely. But 
$\d_{B_n^\sigma}^{-1}\chi_n^{1/2}(z_k)=|t(z_k)|^{-\frac{k(n-k)}{2}+\frac{\e_k}{2}}$, and $\chi_\alpha(z_k)=\alpha(t(z_k))^{\e_k}$ where $\e_k=0$ if $k$ is even, and $1$ if $k$ is odd, 
hence this will be the case if 
$$\int_{F^*} c_{\pi}(t)\phi(0,\dots,0,t)|t|^{ns}(|t|^{1/2}\alpha(t))^{\e_n/2} d^*t$$ and each integral 
$$\int_{F^*} c_{k,i_k}(t)|t|^{\e_k/2}v(t)^{m_k}\phi_k(t)\alpha^{\e_k}(t)|t|^{ks} d^*t
=\int_{F^*} c_{k,i_k}(t)|t|^{ks}(|t|^{1/2}\alpha(t))^{\e_k}v(t)^{m_k}\phi_k(t)d^*t$$ converges absolutely.
 This will be the case if for every $k\in\{1,\dots,n\}$, one has $$Re(s)+[Re(c_{k,i_k})+\e_k(Re(\alpha)+1/2)]/k>0.$$ 
 Moreover, this also shows that $\Psi(W,\phi,\chi_\alpha,s)$ extends to $\C$ as an element of $\C(q^{-s})$, by Tate's theory of local factors of characters of $F^*$. The case of the integrals $\Psi_{(0)}$ is similar.
\end{proof}

We will need the following corollary in the last section.

\begin{cor}\label{CVdiscr}
Let $\D$ be a unitary discrete series (i.e. with unitary central character), and let $\alpha$ be a character of $F^*$ with $Re(\alpha)\geq -1/2$. Let $W$ belong to 
$W(\pi,\theta)$, and $\phi$ belong to $\sm_c(F^{[(n+1)/2]})$, the integral $\Psi(W,\phi,\chi_\alpha,s)$ converges absolutely for $Re(s)>0$. 
\end{cor}
\begin{proof}
It is a consequence of Lemma \ref{CV} and Theorem 3.2 of \cite{M2011.2}.
\end{proof}

This allows us to define the local $L$-factors. We notice that $H_n$ acts on 
$\sm_c(F^{[(n+1)/2]})$ by right translation (with $h(g_1,g_2).\phi(x)=\phi(xg_2)$ when $n$ is even, and 
$h(g_1,g_2).\phi(x)=\phi(xg_1)$ when $n$ is odd).

\begin{prop}\label{Ldef}
Let $\pi$ be a representation of $G_n$ of Whittaker type, and $\alpha$ a character of $F^*$. The integrals $\Psi(W,\phi,\chi_\alpha,s)$ generate a (necessarily principal) 
fractional ideal $I(\pi,\chi_\alpha)$ of $\C[q^s,q^{-s}]$ when $W$ and $\phi$ vary in their respective spaces, and $I(\pi,\chi_\alpha)$ has a unique generator which is an 
Euler factor, which we denote by $L^{lin}(\pi,\chi_\alpha,s)$.\\
The integrals $\Psi_{(0)}(W,\chi_\alpha,s)$ generate a (necessarily principal) 
fractional ideal $I_{(0)}(\pi,\chi_\alpha)$ of $\C[q^s,q^{-s}]$ when $W$ and $\phi$ vary in their respective spaces, 
and $I_{(0)}(\pi,\chi_\alpha)$ has a unique generator which is an Euler factor, which we denote by $L_{(0)}^{lin}(\pi,\chi_\alpha,s)$. 
\end{prop}
\begin{proof}
 In the first case, replacing $W$ and $\phi$ by an appropriate translate, multiplies the integral $\Psi(W,\phi,\chi_\alpha,s)$ by a nonzero multiple of $q^{ks}$, 
for any $k$ in $\Z$, hence the integrals $\Psi(W,\phi,\chi_\alpha,s)$ generate a fractional ideal of  $\C[q^s,q^{-s}]$. 
Now one hase the following integration formula, for good normalisations of measures, $Re(s)$ large enough, and $\e_n=0$ when $n$ is even, 
and $\e_n=1$ when is odd:
\begin{eqnarray} \label{iwaz} \Psi(W,\phi,\chi_\alpha,s) \end{eqnarray}
$$ =\int_{K_n^\sigma}\int_{N_n^\sigma\backslash P_n^\sigma}W(pk)\chi_\alpha(pk)\mu_n^{1/2}(p)|p|^{s-1/2}dp \int_{F^*}c_\pi(t)\alpha^{\e_n}(t)\phi(tl_n(k))|t|^{ns+\e_n/2}d^*t dk.$$
Now use again the argument from \cite{JPS}. By Remark \ref{Kir}, we can choose $W$ in $W(\pi,\theta)$ such that $W_{|P_n}$ is any function 
$f$ in $ind_{N_n}^{P_n}(\theta)$. Let $r$ be large enough such that $W$ and $\chi_\alpha$ are $K_{n,r}^\sigma$-invariant, and $\phi$ be the characteristic function of 
$l_n(K_{n,r}^\sigma)$, the integral reduces to a positive multiple of 
$\int_{N_n^\sigma\backslash P_n^\sigma}W(p)\chi_\alpha\mu_n^{1/2}(p)|p|^{s-1/2}dp=\int_{N_n^\sigma\backslash P_n^\sigma}f(p)\chi_\alpha\mu_n^{1/2}(p)|p|^{s-1/2}dp$. If $f$ is the characteristic function 
of a sufficiently small compact open subgroup with Iwahori decomosition, this integral is a positive constant independant of $s$, so $I(\pi,\chi_\alpha)$ contains $1$, and is thus spanned by an Euler factor. The second statement is proved similarly.
\end{proof}

When $\chi_\alpha$ is equal to $\1_{H_n}$, we write $L^{lin}(\pi,s)$ rather than $L^{lin}(\pi,\1_{H_n}, s)$. We now introduce the notion of exceptional pole of the Euler factor $L^{lin}$.

\begin{df}\label{defpolex}
Let $\pi$ be a representation of $G_n$ of Whittaker type, and let $s_0$ be a pole of order $d$ of $L^{lin}(\pi,\chi_\alpha,s)$. For $W$ in $W(\pi,\theta)$ and $\phi$ in $\sm_c(F^{[(n+1)/2]})$, we write the Laurent expansion $$\Psi(W,\phi,\chi_\alpha,s)=\frac{B(W,\phi,s)}{(1-q^{s_0-s})^d}+ \ terms \ of \ higher \ degree.$$
We say that $s_0$ is an exceptional pole of $L^{lin}(\pi,s)$ if the bilinear form $B$ vanishes on the space $W(\pi,\theta)\times \sm_{c,0}(F^{[(n+1)/2]})$, i.e. if 
it is of the form $B:(W,\phi)\mapsto L(W)\phi(0)$, where $L$ is a nonzero linear form. By definition of the integrals $\Psi(W,\phi,\chi_\alpha,s)$, 
the linear form $L$ is a nonzero element of the space $Hom_{H_n}(\pi,\chi_\alpha^{-1}\chi_n^{-1/2}|.|^{-s_0})$.
\end{df}

Let $\pi$ be a representation of Whittaker type of $G_n$, and $P_{ex}(\pi,\chi_\alpha)$ be the set of exceptional poles of $L^{lin}(\pi,\chi_\alpha,s)$, and denote by $d(s_0)$ the order of an element $s_0$ of $P_{ex}(\pi,\chi)$. By definition, the factor $L^{lin}_{ex}(\pi,\chi_\alpha,s)$ is equal to $\prod_{s_0\in P_{ex}(\pi,\chi_\alpha)} 1/(1-q^{s_0-s})^{d(s_0)}$. We now show that these poles are those of the factor $L^{lin}(\pi,\chi_\alpha,s)/L_{(0)}^{lin}(\pi,\chi_\alpha,s)$. 

\begin{prop}\label{radexlin}
Let $\pi$ be a representation of Whittaker type of $G_n$, and let $\e_n=0$ when $n$ is even and $1$ when $n$ is odd. The ideal $I_{(0)}(\pi,\chi_\alpha)$ is spanned by the integrals 
$\Psi(W,\phi,\chi_\alpha,s)$, for $W\in W(\pi,\theta)$ and $\phi\in \sm_{c,0}(F^{[(n+1)/2]})$. In particular the factor $L_{(0)}^{lin}(\pi,\chi_\alpha,s)$ divides $L^{lin}(\pi,\chi_\alpha,s)$, we denote $L_{rad(ex)}^{lin}(\pi,\chi_\alpha,s)$ their quotient. The factor $L_{rad(ex)}^{lin}(\pi,\chi_\alpha,s)$ divides $L(c_\pi\alpha^{\e_n},ns+\e_n/2)$, hence its poles are simple, they are exactly the exceptional poles of $L^{lin}(\pi,\chi_\alpha,s)$. 
\end{prop}
\begin{proof}
According to the proof of Proposition \ref{Ldef}, for a well chosen $\phi$ which vanishes at zero, the integral $\Psi(W,\phi,\chi_\alpha,s)$ reduces to $\Psi_{(0)}(W,\chi_\alpha,s)$. This proves that $L_{(0)}^{lin}(\pi,\chi_\alpha,s)$ divides $L^{lin}(\pi,\chi_\alpha,s)$. Moreover, using smoothness of elements of $W(\pi,\theta)$, Equation (\ref{iwaz}) in the proof of this proposition shows that any integral $\Psi(W,\phi,\chi_\alpha,s)$ is a linear combination of products of integrals $\Psi_{(0)}(W',\chi_\alpha,s)\int_{F^*}\phi'(t)c_{\pi}(t)\alpha^{\e_n}(t)|t|^{ns+\e_n/2}d^*t$, for $\phi'$ in $\sm_c(F)$. This implies that $L^{lin}(\pi,\chi_\alpha,s)$ divides the factor $L_{(0)}^{lin}(\pi,\chi_\alpha,s)L(c_\pi\alpha^{\e_n},ns+\e_n/2)$,
i.e. that $L_{rad(ex)}(\pi,\chi_\alpha,s)$ divides the factor $L(c_\pi\alpha^{\e_n},ns+\e_n/2)$, in particular it has simple poles. Using again Equation (\ref{iwaz}), we see that if 
$\phi$ belongs to $\sm_{c,0}(F^{[(n+1)/2]})$, the integral $\Psi(W,\phi,\chi_\alpha,s)$ is a linear combination of integrals $P'(q^{\pm s})\Psi_{(0)}(W',\chi_\alpha,s)$, 
for some Laurent polynomials $P'$. Hence the ideal $I_{(0)}(\pi,\chi_\alpha)$ is also the ideal generated by the integrals $\Psi(W,\phi,\chi_\alpha,s)$ with $W\in W(\pi,\theta)$ and $\phi\in \sm_{c,0}(F^{[(n+1)/2]})$.
\end{proof}

Now we recall the following result, obtained from the proof of Proposition 3.1. of \cite{M2012}, by noticing that the construction of 
the injection, gives actually an isomorphism (though only the injectivity is needed for all applications in this work), and by computing the characters 
denoted $\d$ and $\d_1$ in this proof (which are in fact given, for $h\in P_{n-1}^\sigma$, by $\d(h)=|h|^2$ and $\d_1(h(g_1,g_2))=|g_1|$ when 
$n$ is even, and 
$\d_1(h(g_1,g_2))=|g_2|$ when $n$ is odd).

\begin{prop}\label{iso}
Let $\rho$ be an irreducible representation of $P_n$, and $\mu$ be a character  of $P_n^\sigma$. 
If $n$ is even, one has $Hom_{P_n^\sigma}(\Phi^+(\rho),\mu)\simeq Hom_{P_{n-1}^\sigma}(\rho,\mu\d_n^{1/2})$. 
If $n$ is odd, one has $Hom_{P_n^\sigma}(\Phi^+(\rho),\mu)\simeq Hom_{P_{n-1}^\sigma}(\rho,\mu\d_n^{-1/2})$.
\end{prop}

We will use the following lemma to establish the functional equation, and also in the following section.

\begin{LM}\label{hom1}
Let $\pi=\D_1\times \dots \times \D_t$ be a representation of $G_n$ of Whittaker type, and $\alpha$ a character of $F^*$. The spaces 
$$Hom_{H_n}(\pi\otimes \sm_c(F^{[(n+1)/2]}),\chi_\alpha^{-1}\chi_n^{-1/2}|.|^{-s})$$ and $$Hom_{P_n^\sigma}(\pi, |.|^{-s+1/2}\chi_\alpha^{-1}\mu_n^{-1/2})$$
are of dimension $1$, except possibly when there is 
$k\in\{1,\dots,n-1\}$, and 
and a $n-k$-exponent $c_{k,i_k}$, such that $c_{k,i_k}(\w)=q^{ks+\e_k/2}\alpha(\w)^{-\e_k}$ with 
$\e_k=0$ when $k$ is even, and $\e_k=1$ when $k$ is 
odd.
\end{LM}
\begin{proof}
The space $Hom_{H_n}(\pi,|.|^{-s}\chi_n^{-1/2}\chi_\alpha^{-1})$ is zero except 
maybe when $$c_\pi(\w)=q^{ns+\e_n/2}\alpha^{-\e_n}(\w),$$ hence the space $Hom_{H_n}(\pi \otimes \sm_c(F^{[(n+1)/2]}),\chi_\alpha^{-1}\chi_n^{-1/2}|.|^{-s})$ is equal to 
$$Hom_{H_{n}}(\pi \otimes \sm_{c,0}(F^{[(n+1)/2]}),\chi_\alpha^{-1}\chi_n^{-1/2}|.|^{-s})$$ except for those values. But 
$$Hom_{H_n}(\pi \otimes \sm_{c,0}(F^{[(n+1)/2]}),\chi_\alpha^{-1}\chi_n^{-1/2}|.|^{-s})\simeq 
Hom_{H_n}(\pi \otimes ind_{P_n^\sigma}^{H_n}(1),\chi_\alpha^{-1}\chi_n^{-1/2}|.|^{-s})$$
$$\simeq Hom_{H_n}(\pi, Ind_{P_{n}^\sigma}^{H_n}(\mu_n^{-1/2}|.|^{-s+1/2}\chi_\alpha^{-1}))\simeq 
Hom_{P_n^\sigma}(\pi, |.|^{-s+1/2}\chi_\alpha^{-1}\mu_n^{-1/2}),$$
the first isomorhism identifying $P_n^\sigma\backslash H_n$ with $F^{[(n+1)/2]}-\{0\}$, and the last by Frobenius reciprocity law. 
Now, for $k$ between $0$ and $n-1$, and $\sum_{i=1}^t a_i=k$, we have, according to Proposition \ref{iso}:
$$Hom_{P_n^\sigma}((\Phi^+)^{n-k-1}\Psi^+(\D_1^{(a_1)}\times \dots \times \D_t^{(a_t)}), |.|^{-s+1/2}\chi_\alpha^{-1}\mu_n^{-1/2})$$ 
$$\simeq Hom_{P_{k+1}^\sigma}(\Psi^+(\D_1^{(a_1)}\times \dots \times \D_t^{(a_t)}), |.|^{-s+1/2}\chi_\alpha^{-1}\mu_{k+1}^{-1/2})$$
$$\simeq Hom_{G_k^\sigma}(\D_1^{(a_1)}\times \dots \times \D_t^{(a_t)}, |.|^{-s}\chi_\alpha^{-1}\chi_{k}^{-1/2}).$$ 
This ensures us that, for $k\geq 1$, if the central character $c$ of $\D_1^{(a_1)}\times \dots \times \D_t^{(a_t)}$ does not satisfy 
the relation $c(\w)=q^{ks+\e_k/2}\alpha(\w)^{-\e_k}$ for $\e_k=0$ when $k$ is even, and $\e_k=1$ when $k$ is 
odd, this last space is reduced to zero. Finally, if $k=0$, this space is isomorphic to $\C$. Now the lemma 
is a consequence of Proposition \ref{derwhittaker}. 
\end{proof}

 The preceding proposition has the following consequence.

\begin{prop}\label{fctequation}
 Let $\pi$ be a representation of Whittaker type of $G_n$, $\alpha$ a character of $F^*$, let $W$ belong to $W(\pi,\theta)$, and $\phi$ belong to 
 $\sm_c(F^{[(n+1)/2]})$, then one has $$\Psi(\widetilde{W},\chi_\alpha^{-1}\d_n^{-1/2},\widehat{\phi}^{\theta},1/2-s)/L^{lin}(\tilde{\pi},\chi_\alpha^{-1}\d_n^{-1/2},1/2-s)$$ 
 $$= \epsilon^{lin}(\pi,\chi_\alpha,\theta,s) \Psi(W,\chi_\alpha,\phi,s)/L^{lin}(\pi,\chi_\alpha,s)$$
for $\epsilon^{lin}(\pi,\chi_\alpha,\theta,s)$ a unit of $\C[q^s,q^{-s}]$. 
\end{prop}

We will denote by $\gamma^{lin}(\pi,\chi_\alpha,\theta,s)$ the element of $\C(q^{\pm s})$ given by  $$\gamma^{lin}(\pi,\chi_\alpha,\theta,s)=\frac{\e^{lin}(\pi,\chi_\alpha,\theta,s)L^{lin}(\widetilde{\pi},\chi_\alpha^{-1}\d_n^{-1/2},1/2-s)}{L^{lin}(\pi,\chi_\alpha,s)}.$$ 
When $\pi$ is a unitary generic representation of $G_n$, for $n$ even, we will prove that exceptional poles of $L^{lin}(\pi,s)$ at zero characterise distinction. More generally, this property is true under the following condition, which is verified by distinguished unitary generic representations, 
and should be true without the unitary assumption. 

\begin{prop}\label{suffexceptionnel}
 Suppose that $\pi$ is a generic representation of $G_{n}$, for $n=2m$ even, which is $(H_{n},\chi_\alpha^{-1})$-distinguished, and suppose that 
$$Hom_{P_{n}^{\sigma}}(\pi^\vee,\chi_\alpha)= Hom_{G_{n}^{\sigma}}(\pi^\vee,\chi_\alpha)$$ 
(which is the case if $Hom_{P_{n}^{\sigma}}(\pi^\vee,\chi_\alpha)$ is of dimension $\leq 1$). 
Then the factor $L^{lin}(\pi,\chi_\alpha,s)$ has an exceptional pole at zero.
\end{prop}
\begin{proof} 
 From Equation (\ref{iwaz}), for $Re(s)<<0$, as $\pi$ has necessarily a trivial central character, one has:

$$\Psi(\widetilde{W},\widehat{\phi}^\theta,\d_n^{-1/2}\chi_\alpha^{-1},1/2-s)=$$
\begin{equation} \label{sum'} 
 \int_{K_n}
 \int_{N_n\backslash P_n} \!\!\!\!\!\!\!\!\!\! \widetilde{W}(pk)\chi_\alpha^{-1}(pk)|det(p)|^{-s} dp \int_{F^*}
 \widehat{\phi}^\theta(t l_n(k))|t|^{n(1/2-s)} d^*t dk.\end{equation}

Now we notice that $\Lambda_s:\widetilde{W}\mapsto \Psi_{(0)}(\widetilde{W},\d_n^{-1/2}\chi_\alpha^{-1},1/2-s)/L^{lin}(\pi^\vee,\d_n^{-1/2}\chi_\alpha^{-1},1/2-s)$ is an element of $Hom_{P_n}(W(\pi^\vee,\theta^{-1}),\chi_\alpha|.|^s)$. We can write:

$$\Psi(\widetilde{W},\d_n^{-1/2}\chi_\alpha^{-1},1/2-s)/L^{lin}(\pi^{\vee},\d_n^{-1/2}\chi_\alpha^{-1},1/2-s)$$
\begin{equation} \label{summ'}=\int_{K_n}\Lambda_s(\pi^{\vee}(k)\widetilde{W})\int_{F^*}
 \widehat{\phi}^\theta(t l_n(k))|t|^{n(1/2-s)} d^*t dk.\end{equation}

The second member of the equality is actually a finite sum: $$\sum_i \lambda
 _i \Lambda_s(\pi^{\vee}(k_i)\widetilde{W}))\int_{F^*}
 \widehat{\phi}^\theta(t l_n(k_i))|t|^{n(1/2-s)} d^*t,$$ 
where the $\lambda_i$'s are positive constants and the $k_i$'s are elements of $K_n$ independent of $s$.\\
Notice that for
 $Re(s)<1/2$, the integral $$\int_{F^*} \widehat{\phi}^\theta(t l_n(k_i))|t|^{n(1/2-s)} d^* t$$ is absolutely convergent, and defines a holomorphic function.
So we have an equality (Equality (\ref{summ'})) of analytic functions (actually of polynomials in
 $q^{-s}$), hence it is true for all $s$ with $Re(s)<1/2$. For $s=0$, writing $m$ for $n/2$, we get:
 $$\Psi(\widetilde{W},\widehat{\phi}^\theta,\d_n^{-1/2}\chi_\alpha^{-1},1/2)/L(\pi^\vee,\d_n^{-1/2}\chi_\alpha^{-1},1/2)=\!\!
 \int_{K_n}\!\!\!\!\!\Lambda_0
 (\pi^{\vee}(k)\widetilde{W})\int_{F^*} \!\!\!\widehat{\phi}^\theta(t l_n(k))|t|^m d^* tdk.$$

But $\Lambda_0$ is a $(P_n^\sigma,\chi_\alpha)$-invariant linear form on
 $W(\pi^{\vee},\theta^{-1})$, so it follows by hypothesis that it is actually $(H_n,\chi_\alpha)$-invariant.\\
Finally $$\Psi(\widetilde{W},\widehat{\phi}^\theta,\d_n^{-1/2}\chi_\alpha^{-1},1/2)/L(\pi^{\vee},\d_n^{-1/2}\chi_\alpha^{-1},1/2)=
 \Lambda_0(\widetilde{W})\int_{K_n}\int_{F^*}
 \widehat{\phi}^\theta(t l_n(k))|t|^m d^*t dk$$ which is equal for a good normalisation $dk$, to:
$$ \Lambda_0(\widetilde{W})\int_{P_n^\sigma\backslash
 H_n}\widehat{\phi}^\theta(l_n(h))d_{\mu}h$$ where $d_{\mu}$ is up to
 scalar the unique $|det(\ )|^{-1}$ invariant measure on
 $P_n^\sigma\backslash H_n$.
 But as we have $$\int_{P_n^\sigma\backslash
 H_n}\widehat{\phi}^\theta(\eta_n(h))d_{\mu}h=
 \int_{F^m}\widehat{\phi}^\theta(x)dx=\phi(0),$$
we deduce from the functional equation that
 $\Psi(W,\phi,\chi_\alpha,0)/L(\pi,\chi_\alpha,0)=0$ whenever $\phi(0)$ is equal to $0$.\\
As one can choose $W$, and $\phi$ vanishing at zero, such that
  $\Psi(W,\phi,\chi_\alpha,0)$ is the constant function equal to $1$ (see the proof of Proposition \ref{Ldef}),
 the factor $L(\pi,\chi_\alpha,s)$ has a pole at zero, which must be
 exceptional.
\end{proof}

\begin{cor}\label{unitex}
If $n$ is even, $\pi$ is unitary, and $\alpha$ is a character of $F^*$, then $\pi$ is $(H_n,\chi_\alpha^{-1})$-distinguished if and only if 
$L^{lin}(\pi,\chi_\alpha,s)$ has an exceptional pole at zero.
\end{cor}
\begin{proof} As $\pi$ is $(H_n,\chi_\alpha^{-1})$-distinguished if and only if $\pi$ is $(H_n,\chi_\alpha)$-distinguished according to 
Lemma \ref{dualdist}, we can suppose that $Re(\alpha)$ is $\geq 0$. According to the criterion 7.4 of \cite{B}, the $(n-k)$-exponents of $\pi$ are of real part $>-k/2$, except the central character 
of the highest derivative $\pi^{(n)}=\1$. Let $\rho$ be a representation of $G_k$ with central character of real part $>-k/2$, we have 
$$Hom_{P_{n}^\sigma}((\Phi^+)^{n-k-1}\Psi^+(\rho),\chi_\alpha^{-1})$$ $$\simeq 
Hom_{P_{k+1}^\sigma}(\Psi^+(\rho), \chi_\alpha^{-1}\chi_{k+1}^{1/2})\simeq Hom_{G_k^\sigma}(|.|^{1/2}\rho, \chi_\alpha^{-1}\mu_{k}^{1/2}).$$ 
This last space is reduced to zero as soon as $Re(\alpha)\geq 0$. Proposition \ref{BZfiltration} implies that in this case, 
the space $Hom_{P_{n}^\sigma}(\pi,\chi_\alpha^{-1})$ is of dimension $\leq 1$, and we can apply Proposition \ref{suffexceptionnel}.
\end{proof}

As we know that if a discrete series representation of $G_n$ is $\chi_\alpha$-distinguished, the it is unitary and $n$ is even, we botain the following.

\begin{cor}\label{expolediscrete}
Let $\D$ be a discrete series of $G_n$, and $\alpha$ be a character of $F^*$, then $\pi$ is $(H_n,\chi_\alpha^{-1})$-distinguished if and only if 
$L^{lin}(\pi,\chi_\alpha,s)$ has an exceptional pole at zero.
\end{cor}

Now we give a consequence of Theorem \ref{distgen} in terms of $L$-functions.

\begin{prop}\label{translation}
Let $\pi=\D_1\times \dots \D_t$ be a generic representation of $G_n$, with $t\geq 2$, and $\alpha$ a character of $F^*$ with $Re(\alpha)\in [-1/2,0]$ 
such that the factor $L^{lin}_{rad(ex)}(\pi,\chi_\alpha,s)$ has a pole at $s_0$. Then we are in one of the following situations:
        \begin{description}
        \item{1.} There are $(i,j)\in \{1,\dots,t\}$, with $i\neq j$, such that 
          the factors $L_{rad(ex)}^{lin}(\D_i,\chi_\alpha,s)$ and $L_{rad(ex)}^{lin}(\D_j,\chi_\alpha,s)$ have $s_0$ as a common pole. 
	\item{2.} There are $(i,j,k,l)\in \{1,\dots,t\}$, with $\{i,j\}\neq \{k,l\}$,such that $L_{rad(ex)}(\D_i,\D_j,2s)$ and 
   $L_{rad(ex)}(\D_k,\D_l,2s)$ have $s_0$ as a common pole. 
        \item{3.} There are $(i,j,k)\in \{1,\dots,t\}$, and $i\neq j$, such that the factors $L_{rad(ex)}(\D_i,\D_j,2s)$ and 
        $L_{rad(ex)}^{lin}(\D_k,\chi_\alpha,s)$ have $s_0$ as a common pole.
        \item{4.} There are $(i,j)\in \{1,\dots,t\}$, and $i\neq j$, then $L_{rad(ex)}^{lin}(\D_i,\chi_\alpha,s)$ and $L(\alpha\otimes\D_j,s+1/2)$ have $s_0$ as a common pole. 
        \item{5.} There are $(i,j,k)\in \{1,\dots,t\}$, and $i\neq j$, the two factors $L_{rad(ex)}(\D_i,\D_j,2s)$ and $L(\alpha\otimes\D_k,s+1/2)$ 
      have $s_0$ as a common pole.
        \item{6.} The integer $t$ equals $2$, and the factor $L_{rad(ex)}(\D_1,\D_2,2s)$ has a pole at $s_0$.
        \end{description}
\end{prop} 

Before ending this section with a factorisation of $L^{lin}(\pi,\chi_\alpha,s)$ in terms of the exceptional factors 
of the derivatives of $\pi$, we state 
as a lemma the following integration formula, which will be used in the next proposition.

\begin{LM}\label{integralformula}
Let $n\geq k$ be positive integers, let $B_{n,k}^{-,\sigma}$ be the subgroup of $B_{n-k}^-$, of elements $b$ such 
that the matrix $diag(I_k,b)$ belongs to $H_n$, let $\mathcal{M}_{n-k,k}^\sigma$ be the subgroup of $\mathcal{M}(n,k,F)$ of elements 
$m$ such that $\begin{pmatrix} I_k & \\m & I_{n-k} \end{pmatrix}$ belongs to $H_n$. Let $f$ be a positive
 measurable (with respect to the up to scaling unique right invariant measure) function on $N_n\backslash G_n$, then, 
for good normalisations of (right invariant) Haar measures, one has:
$$\int_{N_n^\sigma\backslash H_n} f(h)dh= \int_{B_{n,k}^{-,\sigma}}\int_{\mathcal{M}_{n-k,k}^\sigma}
\int_{N_k^\sigma\backslash H_k} f\begin{pmatrix} h & \\m & b \end{pmatrix}|h|^{(k-n)/2}\chi_k^{1/2}(h)\chi_n^{-1/2}(h)dm dh db$$
\end{LM}

We have the analogue of Proposition \ref{factLpaires} for the factor $L^{lin}$. 

\begin{prop}\label{factL}
Let $\pi=\D_1\times \dots \times \D_t$ be a generic representation, and $\alpha$ be a character of $F^*$. If $\pi$ has only completely reducible derivatives, with irreducible components of the form $\D_1^{(a_1)}\times \dots \times \D_t^{(a_t)}$. For each $k$ between $0$ and $n$, we write 
$\pi^{(k)}=\oplus_{i_k}\pi_{i_k}^{(n-k)}$, where, the $G_{k}$-modules $\pi_{i_k}^{(n-k)}$ are the irreducible components of 
$\pi^{(n-k)}$, then one has:
$$L^{lin}_{(0)}(\pi,\chi_\alpha,s)=\vee_{i_k,k\leq n-1} L^{lin}_{ex} (\pi_{i_k}^{(n-k)},\chi_\alpha,s),$$ hence 
$$L^{lin}(\pi,\chi_\alpha,s)=\vee_{i_k,k\leq n} L^{lin}_{ex} (\pi_{i_k}^{(n-k)},\chi_\alpha,s).$$
\end{prop}
\begin{proof}
Let's first deduce the second equality from the first one. One must show that 
$$\vee_{i_k,k\leq n} L^{lin}_{ex} (\pi_{i_k}^{(n-k)},\chi_\alpha,s)=L_{rad(ex)}^{lin}(\pi,\chi_\alpha,s)L_{(0)}^{lin}(\pi,\chi_\alpha,s),$$ assuming 
$$L_{(0)}(\pi,\chi_\alpha,s)=\vee_{i_k,k\leq n-1} L^{lin}_{ex} (\pi_{i_k}^{(n-k)},\chi_\alpha,s).$$
Hence we want to prove $L_{rad(ex)}^{lin}(\pi,s)L_{(0)}^{lin}(\pi,\chi_\alpha,s)=L_{ex}^{lin}(\pi,\chi_\alpha,s)\vee L_{(0)}^{lin}(\pi,\chi_\alpha,s)$, or 
equivalently $$L_{rad(ex)}^{lin}(\pi,\chi_\alpha,s)=\frac{L_{ex}^{lin}(\pi,\chi_\alpha,s)\vee L_{(0)}^{lin}(\pi,\chi_\alpha,s)}{L_{(0)}^{lin}(\pi,\chi_\alpha,s)}.$$ 
Both sides have simple poles according to Proposition \ref{radexlin}, and they obviously have the same poles according to the same proposition, but as they are Euler factors, they are equal.\\
Now we prove \begin{equation}\label{ppcm} L_{(0)}(\pi,\chi_\alpha,s)=\vee_{i_k,k\leq n-1} L^{lin}_{ex} (\pi_{i_k}^{(n-k)},\chi_\alpha,s).\end{equation} 
Let $s_0$ be a pole of order $d$ of the factor $L^{lin}_{ex} (\pi_{i_k}^{(n-k)},s)$, it certainly occurs as a pole of arder $d$ of an integral $\Psi(W',\phi',\chi_\alpha,s)$, for $W'$ in $W(\pi_{i_k}^{(n-k)},\theta)$ and 
$\phi'$ in $\sm_c(F^{[(k+1)/2]})$ which does not vanish at zero. In fact we can choose $\phi'$ such that $\phi'(0)=1$, and we can even replace $\phi'$ by the characteristic 
function $\phi''$ of a neighbourhood of zero small enough, because $s_0$ occurs as a pole of order $<d$ of $\Psi(W',\phi'-\phi'',\chi_\alpha,s)$. We notice that the support of $\phi''$ can be taken as small as needed. Then, according 
to Proposition \ref{mirabolicrestriction}, there is $W$ in $W(\pi,\theta)_k\subset W(\pi,\theta)$ such that 
$$\Psi(W',\phi'',\chi_\alpha,s)=\int_{N_k^\sigma \backslash H_k} W(diag(h,I_{n-k}))\phi''(l_k(h))\chi_\alpha(h)\chi_k(h)|h|^{s+(k-n)/2}dh.$$ 
Suppose that $s_0$ is a pole of order at least $d$ of the integral 
$$\Psi_{(n-k-1)}(W,\chi_\alpha,s)=\int_{N_k^\sigma \backslash H_k} W(diag(h,I_{n-k}))\chi_\alpha(h)\chi_k(h)|h|^{s+(k-n)/2}dh.$$ 
This integral 
is equal to  $\Psi_{(0)}(W_1,\chi_\alpha,s)$ for some $W_1$ in $W(\pi,\theta)$, as a consequence of Lemma 9.2. of \cite{JPS} and Lemma 
\ref{integralformula}, hence $s_0$ is a pole of order at least $d$ of $L_{(0)}^{lin}(\pi,\chi_\alpha,s)$. Otherwise it is a pole of order at least $d$ 
of $$\int_{N_k^\sigma \backslash H_k} W(diag(h,I_{n-k}))(1-\phi''(l_k(h)))\chi_\alpha(h)\chi_k(h)|h|^{s+(k-n)/2}dh.$$ But as the map $1-\phi''$ vanishes at zero, 
this implies, thanks to the Iwasawa decomposition $G_k=P_kZ_k K_k$, as $W(diag(pzk,I_{n-k}))$ vanishes for $|z|>>0$, 
that this integral is a $\C[q^{\pm s}]$-combination of integrals of the form 
 $$\int_{N_{k-1}^\sigma \backslash H_{k-1}} W'(diag(hk_0,I_{n-k}))\chi_\alpha(h)\chi_{k-1}(h)|h|^{s+(k-n-1)/2}dh$$ 
$$= \int_{N_{k-1}^\sigma \backslash H_{k-1}} W''(diag(h,I_{n+1-k}))\chi_\alpha(h)\chi_{k-1}(h)|h|^{s+(k-n-1)/2}dh.$$
 for some $k_0$ in $K_k$ (see Equation 
(\ref{iwaz})), and $W''=\rho(diag(k_0,I_{n-k}))W'$. Again, by Lemma 9.2. of \cite{JPS} and Lemma 
\ref{integralformula}, this integral is of the form $\Psi_{(0)}(W_2,\chi_\alpha,s)$ for some $W_2$ in $W(\pi,\theta)$. This proves that the right hand side divides the left hand side in equality (\ref{ppcm}).\\
 It remains to show that any pole of order $d$ of $L_{(0)}^{lin}(\pi,\chi_\alpha,s)$, occurs as a pole of order $d$ of some factor $L^{lin}_{ex} (\pi_{i_k}^{(n-k)},\chi_\alpha,s)$, for some $k$ and $i_k$. We show this by induction on $n$. Let $s_0$ be a pole of order $d$ of $L_{(0)}^{lin}(\pi,s)$, it is a pole of order $d$ in an integral $$\Psi_{(0)}(W,\chi,s)=\int_{N_{n-1}^\sigma\backslash H_{n-1}} W(diag(h,1))\chi_\alpha(h)\chi_{n-1}(h)|h|^{s-1/2}dh$$ for some 
 $W$ in $W(\pi,\theta)$. Let $k$ be the smallest integer such that $s_0$ is a pole of order at least $d$ of an integral of the form 
 $$\Psi_{(n-k-1)}(W,\chi_\alpha,s)=\int_{N_{k}^\sigma\backslash H_{k}} W(diag(h,I_{n-k}))\chi_\alpha(h)\chi_k(h)|h|^{s-(n-k)/2}dh$$ 
 for some $W$ in $W(\pi,\theta)$. Then for any $\phi$ in $\sm_{c}(F^{[(k+1)/2]})$ whith $\phi(0)=1$, and $W$ in $W(\pi,\theta)$, the order of $s_0$ in  
 $$\int_{N_{k}^\sigma\backslash H_{k}} W(diag(h,I_{n-k}))(1-\phi(l_k(h)))\chi_\alpha(h)\chi_k(h)|h|^{s-(n-k)/2}dh $$ 
 is $<d$, because we already saw that one could express such an integral as a $\C[q^{\pm s}]$-combination of integrals of the form 
 $\Psi_{(n-k-2)}(W',\chi_\alpha,s)$. Hence let $W$ be such that $s_0$ is a pole of order $\geq d$ of $\Psi_{(n-k-1)}(W,\chi_\alpha,s)$, it is thus a pole of order $\geq d$ of the integral $$\int_{N_{k}^\sigma\backslash H_{k}} W(diag(h,I_{n-k}))\phi(l_k(h))\chi_\alpha(h)\chi_k(h)|h|^{s-(n-k)/2}dh $$ for nay 
 $\phi$ with $\phi(0)=1$. 
Write $\pi^{(n-k)}=\pi_1^{(n-k)} \oplus \dots \oplus \pi_l^{(n-k)}$ a decomposition of $\pi^{(n-k)}$ into a sum of simple factors, then $W$ can be written as 
$W_1+\dots+W_t$, with each $W_i$ projecting on some $W'_i$ in $W(\pi_i^{(n-k)},\theta)$ via the surjection $\pi\rightarrow \pi^{(n-k)}$. 
There must be $i$ such that $s_0$ is a pole of order at least $d$ of 
$\Psi_{(n-k-1)}(W_i,\chi_\alpha,s)$, hence a pole of order at least $d$ of $$\int_{N_{k}^\sigma\backslash H_{k}} W_i(diag(h,I_{n-k}))\phi(l_k(h))\chi_\alpha(h)\chi_k(h)|h|^{s-(n-k)/2}dh $$ for any $\phi$ with $\phi(0)=1$. Taking $\phi$ the characteristic function of a small enough neighbourhood of zero, and applying Proposition \ref{mirabolicrestriction}, the integral 
$$\int_{N_{k}^\sigma\backslash H_{k}} W_i(diag(h,I_{n-k}))\phi(l_k(h))\chi_\alpha(h)\chi_k(h)|h|^{s-(n-k)/2}dh$$ is equal to $\Psi(W'_i,\phi,\chi_\alpha,s)$, hence 
$\Psi(W'_i,\phi,\chi_\alpha,s)$ has a pole of order at least $d$ at $s_0$, and $L(\pi_i^{(n-k)},\chi_\alpha,s)$ as well. If $s_0$ is exceptional, we are done, 
otherwise it is a pole of the factor $L_{(0)}(\pi_i^{(n-k)},\chi_\alpha,s)$, and we conclude by induction, as the irreducible factors of the derivatives of 
$\pi_i^{(n-k)}$ appear amongst those of the derivatives of $\pi$.
\end{proof}

\subsection{Rationality of the Rankin-Selberg integrals under deformation}\label{rationalintegral}

We fix a representation $\pi=\D_1\times \dots \times \D_t$ of Whittaker type of $G_n$, and $\alpha$ a character of $F^*$. If 
$u=(u_1,\dots,u_t)$ is an element of $\mathcal{D}^t$, we recall that $\pi_u$ is the representation 
$$\pi_u=|.|^{u_1}\D_1\times\dots\times|.|^{u_1}\D_t.$$
We recall from section $3$ of \cite{CP}, that to any $f$ in the space $V_\pi$ of $\pi$, and any $u$ in $\mathcal{D}^t$, one can 
associate an element $W_{f,u}=W_{f_u}$ (i.e. which depends on $f_u$, see Section \ref{Bernstein}) in $W(\pi_u,\theta)$, such that if $g$ 
belongs to $G_n$, then $u\mapsto W_{f,u}(g)$ belongs to 
$\C[q^{\pm u}]$. Moreover, the map $f_u\mapsto W_{f,u}$ is a $G_n$-equivariant linear map, which is an isomorphism between $V_{\pi_u}$ and 
$W(\pi_u,\theta)$ whenever and $\pi_u$ is of Langlands' type. 
We define $W_\pi^{(0)}$ the complex vector space generated by the functions $(u,g)\mapsto W_{f_u}(gg')$ for $g'\in G_n$ and $W\in W(\pi,\theta)$. 
It is shown in \cite{CP} that the action of the group $G_n$ on $W_\pi^{(0)}$ by right translation is a smooth representation, and we denote 
by $W_{\pi,(0)}$ the space of restrictions of functions of $W_{\pi}^{(0)}$ to $P_n$. We denote by $\mathcal{P}_0$ the vector subspace of 
$\mathbb{C}[q^{\pm{u}}]$ consisting of all Laurent polynomials of the form $u\mapsto W(u,I_n)$ for some $W\in W_{\pi,(0)}$. We 
recall Proposition 3.1 of \cite{CP}.

\begin{prop}\label{incl}
Let $\pi$ be a representation of Whittaker type of $G_n$, the complex vector space $W_{\pi,(0)}$ defined above contains 
the space $\sm_c(N_n\backslash P_n, \mathcal{P}_0, \theta)$.
\end{prop}
 
We are going to show that the integrals $\Psi(W_{f,u},\phi,\chi_\alpha,s)$ are rational functions in $q_F^{-u}$ and $q_F^{-s}$. We 
first state a corollary of the preceding 
proposition, which will allow us to get a unique solution of some linear system in the proof of the next theorem.

\begin{cor}\label{normal}
Let $(f_\beta)_{\beta\in B}$ be a basis of $V_\pi$ indexed by a countable set $B$. Let $P_0$ be a nonzero element of $\mathcal{P}_0$, then there is a map 
$\phi_0$ in $\mathcal{C}_c^\infty (F^{[(n+1)/2]})$, a finite subset $F_B$ of $B$, and there are $g_\beta$'s in $G_n$, such that $$\Psi(\sum_{\beta\in F_B} \pi_u(g_\beta)W_{f_\beta,u},\phi_0,\chi_\alpha,s)=P_0(q^{\pm u}).$$
\end{cor} 
\begin{proof}
Let $W$ belong to $W_{\pi}^{(0)}$, such that the restriction of $W$ to $P_n\times \mathcal{D}^t$ is given by $f(p)P_0(q^{\pm u})$, for $f$ in $\sm_c(N_n\backslash P_n, \theta)$ the characteristic function of a sufficiently small subgroup of $G_n$ with Iwahori decomposition; this is possible according to Proposition \ref{incl}. Now choose $\phi_0$ the characteristic function of $l_n(K_{n,r}^\sigma)$ for 
$r$ large enough, when we apply Equation (\ref{iwaz}) to $\Psi(W_u,\phi_0,\chi_\alpha,s)$ for $W_u:g\mapsto W(u,g)$, we deduce that 
$\Psi(W_u,\phi_0,\chi_\alpha,s)$ is equal to $P_0(q^{\pm u})$, up to multiplying $\phi_0$ by a positive real. Finally, by definition of $W_\pi^{(0)}$, there is a finite subset $F_B$ of $B$, and there are $g_\beta$'s in $G_n$, such that $W(u,g)=\sum_{\beta\in F_A} \pi_u(g_\beta)W_{f_\beta,u}(g)$ for all $u$ and $g$, and this concludes the proof. 
\end{proof}
 
 As often in this paper, we follow \cite{CP}, and apply Theorem \ref{Ber} to prove the rationality of the Rankin-Selberg integrals with respect to $u$.
 
\begin{thm}\label{rationality}
Let $f$ belong to the space of a representation $\pi$ of Whittaker type of $G_n$, $\phi$ belong to $C_c^{\infty}(F^{[(n+1)/2]})$, and $\alpha$ be a character of $F^*$. The integral
 $\Psi(W_{f,u},\phi,\chi_\alpha,s)$ is a rational function in $q^{-u}$ and $q^{-s}$.
\end{thm}
\begin{proof}
We use notations of Section \ref{Bernstein}. Let $V$ be the vector space $\mathcal{F}_{\pi}\otimes \sm_c(F^{[(n+1)/2]})$, let $f_{\beta,c}$ be a basis of $\mathcal{F}_\pi$, corresponding to the basis $f_\beta$ of $V_\pi$, and let $(\phi_\gamma)_{\gamma\in C}$ be a basis of $\sm_c(F^{[(n+1)/2]})$. We choose $F_B$, the $g_\beta$'s, $P_0$ and $\phi_0$ as in Lemma \ref{normal}. We denote by 
$d$ a general element $(u,s)$ in $\mathcal{D}^{t+1}$. 
For a fixed $d=(u,s)$ in 
$\mathcal{D}^{t+1}$, we define the systems 
$$\Xi'_d = \left\lbrace \begin{array}{lc} (\pi_u(h)\pi_u(g_i)f_{\beta,c}\otimes \rho(h)\phi_\gamma - \chi_\alpha^{-1}(h)\chi_n^{-1}(h)|h|^{-s} \pi_u(g_i) f_{\beta,c} \otimes \phi_\gamma,0), \\ \beta\in B, \gamma \in C, h\in H_n, g_i \in G_n   
\end{array} \right\rbrace$$ and 
$$E_d=\left\lbrace (\sum_{\beta\in F_B}  \pi_u(g_{\beta}) f_{\beta,c} \otimes \phi_0, P_0(q^{\pm{u}})) \right\rbrace $$
of elements of $ V\times \C$. We denote by $\Xi_d$ the union of those two systems.\\ 
 Now, thanks to Lemma \ref{CV}, there are linear affine forms $L_1,\dots, L_r$ on $\mathcal{D}^{t}$, such that when $(u,s)$ satisfies 
 $Re(s)>Re(L_i(u))$, the factor $\Psi(W_{f,u},\phi,\chi_\alpha,s)$ 
is defined by an absolutely convergent integral, for all $f$ in $V_\pi$, and all $\phi$ in $\sm_c(F^{[(n+1)/2]})$.\\
Moreover, thanks to Lemma \ref{hom1}, there are affine linear forms $L'_1,\dots,L'_{r'}$ on $\mathcal{D}^t$, such that outside the finite number of hyperplanes $H_1=\{q^s=q^{L_1(u)}\},\dots,H_{r'}=\{q^s=q^{L_{r'}(u)}\}$ of $\mathcal{D}^{t+1}$, the space of solutions of the system $\Xi'_d$ is of dimension $1$.\\
This implies, that on the intersection $\Omega$ of the half-spaces $Re(s)>Re(L_i(u))$, and of the complementary of the $H_j$'s in $\mathcal{D}^{t+1}$, the system 
$\Xi_d$ has a unique solution which we denote by $J_d:f_u\otimes \phi \mapsto \Psi (W_{f,u},\phi,\chi_\alpha,s)$. We can now apply Bernstein's Theorem (Theorem \ref{Ber}), which says that there is a unique solution $I\in V_{\C(\mathcal{D}^{t+1})}^*$ of $\Xi$, and that $I(d)=J_d$ on $\Omega$. Put in another way, for $(u,s)$ in $\Omega$, we have $$I(u,s)(W_{f,u},\phi,\chi_\alpha,s) =\Psi (W_{f,u},\phi,\chi_\alpha,s),$$ and the left side belongs to $\C(q^{\pm u},q^{\pm s})$. Now fix a $u$, then for $Re(s)$ greater than the real $Max_{i,j}\{Re(L_i(u)),Re(L'_j(u))\}$, we have $I(u,s)(W_{f,u},\phi,\chi_\alpha,s) =\Psi (W_{f,u},\phi,\chi_\alpha,s)$, and both are rational functions 
$q^{-s}$, hence they are equal this open set of $\mathcal{D}$, hence on $\mathcal{D}$. We conclude that $\Psi (W_{f,u},\phi,\chi_\alpha,s)=I(u,s)(W_{f,u},\phi,\chi_\alpha,s)$ on $\mathcal{D}^{t+1}$, and this concludes the proof.
\end{proof}

Thanks to the functional equation (Proposition \ref{fctequation}), we have the following corollary.

\begin{cor}\label{gammarational}
Let $\pi$ be a representation of $G_n$ of Whittaker type, and $\alpha$ be a character of $F^*$, then the factor 
$\gamma(\pi_u,\chi_\alpha,\theta,s)$ belongs to $\C(\mathcal{D}^{t+1})$.
\end{cor}

\section{The equality of the Rankin-Selberg factor and the Galois factor}\label{egalitegalois}

\subsection{The inductivity relation for representations in general position}

In this section, we show that if $\pi=\D_1\times\dots\times\D_t$ is a representation of $G_n$ of Whittaker type, $\alpha$ is a character of 
$F^*$ with $Re(\alpha)\in[-1/2,0]$, and for $u$ in $\mathcal{D}^t$ in general position 
with respect to $(\pi,\chi_\alpha)$ (see Definition \ref{generalpos}), one has the relation 
$$L^{lin}(\pi_u,\chi_\alpha,s)=\prod_{1\leq i<j \leq t}L(|.|^{u_i}\D_i, |.|^{u_j}\D_j,2s)\prod_{k=1}^t L^{lin}(|.|^{u_k}\D_k,\chi_\alpha,s).$$

We now fix $\pi=\D_1\times\dots\times\D_t$ a representation of Whittaker type of $G_n$, with $t\geq 2$, for the
 rest of this section. We recall that for each $i$, the discrete series is a segment $[\rho_i,\dots,|.|^{l_i-1}\rho_i]$, with 
$\rho_i$ a cuspidal representation of $G_{r_i}$ and $\sum_{i=1} l_ir_i= n$.

\begin{df}\label{generalpos}
Let $\pi=\D_1\times\dots\D_t$ be a representation of Whittaker type of $G_n$, we say the $u=(u_1,\dots,u_t)\in \mathcal{D}^t$ is in 
general position with respect to $(\pi,\chi_\alpha)$, if:
\begin{description}
        \item{1.} For every sequence of non negative integers $(a_1,\dots,a_t)$ such that the representation 
$$|.|^{u_1}\D_1^{(a_1)}\times\dots\times|.|^{u_1}\D_t^{(a_t)}$$ is nonzero, this representation is irreducible.
	\item{2.} If $(a_1,\dots,a_t)$ and $(b_1,\dots,b_t)$ are two different sequences of positive integers, such that 
     $\sum_{i=1}^t a_i=\sum_{i=1}^t b_i$, and such that the representations $|.|^{u_1}\D_1^{(a_1)}\times\dots\times|.|^{u_t}\D_t^{(a_t)}$ and 
  $|.|^{u_1}\D_1^{(a_1)}\times\dots\times|.|^{u_t}\D_t^{(a_t)}$
 are non zero, then these two representations have distinct central characters.
	%\item  If $(a_1,\dots,a_t)$ is as before, denote by $a$ the integer $\sum_{i=1}^t a_i$, and suppose $a<n$, then the space 
	%$Hom_{P_{n}^\sigma}((\Phi^+)^{a-1}\Psi^+(|.|^{u_1}\D_1^{(a_1)}\times\dots\times|.|^{u_t}\D_t^{(a_t)}),\chi_n)$ is reduced to zero.
	\item{3.} If $(i,j)\in \{1,\dots,t\}$, and $i\neq j$, then $L^{lin}(|.|^{u_i}\D_i,\chi_\alpha,s)$ and $L^{lin}(|.|^{u_j}\D_j,\chi_\alpha,s)$ have no common poles. 
	\item{4.} If $(i,j,k,l)\in \{1,\dots,t\}$, and $\{i,j\}\neq \{k,l\}$, then the factors $L(|.|^{u_i}\D_i,|.|^{u_j}\D_j,2s)$ and 
   $L(|.|^{u_k}\D_k,|.|^{u_l}\D_l,2s)$ have no common poles. 
        \item{5.} If $(i,j,k)\in \{1,\dots,t\}$, and $i\neq j$, the two Euler factors $L(|.|^{u_i}\D_i,|.|^{u_j}\D_j,2s)$ and $L^{lin}(|.|^{u_k}\D_k,\chi_\alpha,s)$ 
      have no common poles.
        \item{6.} If $(i,j)\in \{1,\dots,t\}$, and $i\neq j$, then $L^{lin}(|.|^{u_i}\D_i,\chi_\alpha,s)$ and $L(\alpha\otimes \D_j|.|^{u_j},s+1/2)$ have no common poles. 
        \item{7.} If $(i,j,k)\in \{1,\dots,t\}$, and $i\neq j$, then the Euler factors $L(|.|^{u_i}\D_i,|.|^{u_j}\D_j,2s)$ and $L(\alpha\otimes \D_k|.|^{u_k},s+1/2)$ 
      have no common poles.
        \item{8.} If $t=2$, and $\D_2$ is equal to $|.|^e\D_1^\vee$ for some complex number $e$, then 
 the space $$Hom_{P_n^\sigma}(|.|^{(u_1-u_2-e)/2}\D_1\times |.|^{(u_2-u_1+e)/2}\D_1^\vee,\chi_\alpha^{-1})$$ is of dimension at most $1$.
\end{description} 
\end{df}

We now check that outside a finite numbers of hyperplanes in $u$, the representation $\pi_u$ is in general position.

\begin{prop}
 Let $\pi$ be as above, the elements $u$ in $\mathcal{D}^t$ that are not in general position in with respect to $(\pi,\chi_\alpha)$ belong to a finite number of 
affine hyperplanes.
\end{prop}
\begin{proof}
If the first condition is not satisfied. If $|.|^{u_1}\D_1^{(a_1)}\times\dots\times|.|^{u_1}\D_t^{(a_t)}$ is nonzero, then $\sum_{i=1}^ta_i\leq n$, hence 
there are a finite number of non negative sequences $(a_1,\dots,a_t)$ such that 
$|.|^{u_1}\D_1^{(a_1)}\times\dots\times|.|^{u_1}\D_t^{(a_t)}$ is nonzero.
If $|.|^{u_1}\D_1^{(a_1)}\times\dots\times|.|^{u_1}\D_t^{(a_t)}$ is reducible, this means that $|.|^{u_i}\D_i^{(a_i)}$ is linked to 
$|.|^{u_j}\D_j^{(a_j)}$ for some $i<j$. Hence there are two cuspidal representations $\rho_i$ and $\rho_j$, and two integers 
$k\in \{-l_i,\dots, l_i\}$ and $l\in \{-l_j,\dots, l_j\}$ such that 
$|.|^{u_i+k}\rho_i=|.|^{u_j+l}\rho_j$, hence $u$ belongs to the proper hyperplane $q^{-r(u_i-u_j+k-l)}=c_{\rho_i}^{-1}c_{\rho_j}(\w)$, where $r=r_i=r_j$. 
Hence $u$ belongs to a finite numbers of hyperplanes if $1.$ is not satisfied.\\
If the second condition is not satisfied, the central characters of $|.|^{u_1}\D_1^{(a_1)}\times\dots\times|.|^{u_1}\D_t^{(a_t)}$ and 
  $|.|^{u_1}\D_1^{(b_1)}\times\dots\times|.|^{u_1}\D_t^{(b_t)}$ are equal for two different sequences of nonnegative integers 
$(a_i)$ and $(b_i)$ such that $\sum_i a_i=\sum_i b_i$ such that these representations are nonzero. Again the sequences $(a_i)$ and $(b_i)$ 
belong to the finite set of sequences of non negative integers satisfying $\sum_i a_i\leq n$. 
Let $i$ be an element such that $a_i\neq b_i$, for example $a_i<b_i$, the equality of the central characters at $\w$ gives the relation: 
$q^{(b-a_i)u_i}\l=b(u_1,\dots,u_{i-1},u_{i+1},\dots,u_t),$ for $\l=c_{\D_i^{(b_i)}}(c_{\D_i^{(a_i)}})^{-1}(\w)$, and 
$b(u_1,\dots,u_{i-1},u_{i+1},\dots,u_t)=\prod_{j\neq i}c_{|.|^{u_j}\D_j^{(a_j)}}c_{|.|^{u_j}\D_j^{(b_j)}}^{-1}(\w)$. 
Hence $u$ belongs to a finite numbers of hyperplanes if $2.$ is not satisfied.\\
If condition 3. is not satisfied, then $L^{lin}(|.|^{u_i}\D_i,\chi_\alpha,s)$ and $L^{lin}(|.|^{u_j}\D_j,\chi_\alpha,s)$ have a common pole for some $i\neq j$. This implies 
that $L^{lin}_{ex}(|.|^{u_i}\D_i^{(a_i)},\chi_\alpha,s)$ and $L^{lin}_{ex}(|.|^{u_j}\D_j^{(b_i)},\chi_\alpha,s)$ 
have a common pole $s_0$ for some non negative $a_i$ and $b_j$ according to Proposition \ref{factL}. Thus 
$|.|^{u_j}\D_j^{(b_i)}$ and $|.|^{ u_i}\D_i^{(a_i)} $ are $(H_n,\chi_\alpha^{-1}|.|^{-s_0})$-distinguished according to Proposition \ref{radexlin}. But if $m_i\leq n$ and $m_j\leq n$ 
are the positive integers such that  $|.|^{u_j}\D_j^{(b_j)}$ and $|.|^{ u_i}\D_i^{(a_i)} $ are respectively representations of $G_{m_i}$ and $G_{m_j}$ (which must be even as they correspond to distinguished discrete series), 
then this gives the equation $q^{m_im_j(u_j-u_i)}c_{\D_j^{(b_j)}}^{m_i}c_{\D_i^{(a_i)}}^{-m_j}(\w)=1$ (as $\chi_\alpha$ restricted to 
$Z_{m_i}$ and $Z_{m_j}$ is trivial). Hence $u$ belongs to a finite numbers of hyperplanes 
if $3.$ is not satisfied.\\ 
If condition $4.$ is not satisfied, the factors $L(|.|^{u_i}\D_i,|.|^{u_j}\D_j,2s)$ and 
$L(|.|^{u_k}\D_k,|.|^{u_l}\D_l,2s)$ have a common pole for some $\{i,j\}\neq \{k,l\}$. According to Proposition \ref{factLpaires}, this implies 
that $L_{ex}(|.|^{u_i}\D_i^{(a_i)},|.|^{u_j}\D_j^{(a_j)},2s)$ and $L_{ex}(|.|^{u_k}\D_k^{(a_k)},|.|^{u_l}\D_l^{(a_l)},2s)$ have a common pole 
$s_0$ for some positive integers $a_i, a_j, a_k$ and $a_l$, which in turn implies $\D_j^{(a_j)}=|.|^{u_i-u_j-2s_0}(\D_i^{(a_i)})^\vee$ and 
$\D_l^{(a_l)}=|.|^{u_k-u_l-2s_0}(\D_k^{(a_k)})^\vee$ thanks to Proposition \ref{radexpaires}. Let $m_{i,j}\leq n$ and $m_{l,k}\leq n$ 
be the positive integers such that $\D_i^{(a_i)}$ and $\D_j^{(a_j)}$ are representations of $G_{m_{i,j}}$, and 
$\D_k^{(a_k)}$ and $\D_l^{(a_l)}$ are representations of $G_{m_{k,l}}$, we obtain the equation 
$$q^{m_{i,j}m_{k,l}(u_i-u_j-u_k+u_l)}(c_{\D_k^{(a_k)}}c_{\D_l^{(a_l)}})^{m_{i,j}}(\w)(c_{\D_i^{(a_i)}}c_{\D_j^{(a_j)}})^{-m_{k,l}}(\w) =1.$$
Hence $u$ belongs to a finite numbers of hyperplanes 
if $4.$ is not satisfied.\\ 
In a similar manner to $3.$ and $4.$, one proves that $u$ belongs to a finite numbers of hyperplanes 
if either $5.$, $6.$ or $7.$is not satisfied.\\
The last case $8.$ needs to be treated more carefully. We will actually check that except for a finite numbers of $s\in \mathcal{D}$, the space 
$Hom_{P_n^\sigma}(|.|^s\D_1\times |.|^{-s}\D_1^\vee,\chi_\alpha^{-1})$ is of dimension at most $1$; this will imply the result. Let's omit the index $1$ in the following, to simplify notations. According to Proposition 
\ref{derwhittaker}, the restriction to $P_n$ of the representation $|.|^s\D\times |.|^{-s}\D^\vee$ has a filtration by the $P_n$-modules 
$$(\Phi^+)^{(a+b)r-1}\Psi^+(|.|^s[|.|^a\rho,\dots,|.|^{l-1}\rho]\times |.|^{-s}[|.|^{1-l+b}\rho^\vee,\dots,\rho^\vee])$$ for $a+b\leq 2l$. In general, 
one has an isomorphism between the space 
$$Hom_{P_n^\sigma}((\Phi^+)^{(a+b)r-1}\Psi^+(|.|^s[|.|^a\rho,\dots,|.|^{l-1}\rho]\times |.|^{-s}[|.|^{1-l+b}\rho^\vee,\dots,\rho^\vee]),\chi_\alpha^{-1})$$ and 
$$Hom_{G_{n-(a+b)r}^\sigma}(|.|^s[|.|^a\rho,\dots,|.|^{l-1}\rho]\times |.|^{-s}[|.|^{1-l+b}\rho^\vee,\dots,\rho^\vee],\chi_\alpha^{-1}\mu_{(a+b)r}|.|^{-1/2}),$$ 
 according to Proposition \ref{iso}, where $\mu_{(a+b)r}=\d_{n-(a+b)r}^{-1/2}$ when $(a+b)r$ is even, and $\mu_{(a+b)r}$ is trivial when $(a+b)r$ is odd.
We deduce that for $a+b=2l$, the vector space of invariant linear forms $Hom_{P_n^\sigma}((\Phi^+)^{(a+b)r-1}\Psi^+(\1),\chi_\alpha^{-1})$ is of dimension at most $1$. 
If $a=b<l$, the central character of $$|.|^s[|.|^a\rho,\dots,|.|^{l-1}\rho]\times |.|^{-s}[|.|^{1-l+a}\rho^\vee,\dots,\rho^\vee]$$ is 
equal to $t\mapsto |t|^{2r(l-a)}$, hence
 $$Hom_{G_{n-(2a)r}^\sigma}(|.|^s[|.|^a\rho,\dots,|.|^{l-1}\rho]\times |.|^{-s}[|.|^{1-l+a}\rho^\vee,\dots,\rho^\vee],\chi_\alpha^{-1}\mu_{2ar}|.|^{-1/2})$$ 
is reduced to zero because $$t\mapsto \chi_\alpha^{-1}\mu_{2ar}(tI_{n-2ar})|t|^{(2ar-n)/2}=|t|^{(2ar-n)/2}$$ is different 
from  $t\mapsto |t|^{2r(l-a)}$ (they have respectively negative and positive real parts). If $a\neq b$, the central character of 
$$|.|^s[|.|^a\rho,\dots,|.|^{l-1}\rho]\times |.|^{-s}[|.|^{1-l+b}\rho^\vee,\dots,\rho^\vee]$$ is equal to 
$t\mapsto |t|^{sr(b-a)}c_{\rho}^{b-a}|t|^{r(a+\dots+l-1+(1-l+b)+\dots+-1)}$, hence 
$$q^{-sr(b-a)}c_{\rho}^{b-a}(\w)q^{-r(a+\dots+l-1+(1-l+b)+\dots+-1)}=\chi_\alpha^{-1}\mu_{(a+b)r}(\w I_{n-(a+b)r})q^{(n-(a+b)r)/2},$$
which is possible only for a finite number of $s\in \mathcal{D}$. This concludes the proof.
\end{proof}

The interest of considering representations in general position, is that the $L$ 
function $L^{lin}$ is easy to compute for those representations. From now on, untill the end of this pargraph, 
the character $\alpha$ of $F^*$ will satisfy $Re(\alpha)\in [-1/2,0]$, so that we can use Proposition \ref{translation}.

\begin{prop}\label{radexgeneral}
Let $u$ be in general position with respect to $(\pi,\chi_\alpha)$, and $\pi_u$ the representation $\D_1|.|^{u_1}\times \dots \times \D_t|.|^{u_t}$. One has 
$L^{lin}_{rad(ex)}(\pi_u,\chi_\alpha,s)=1$, except if $t=2$, in which case
 $L^{lin}_{rad(ex)}(\pi_u,\chi_\alpha,s)=L_{rad(ex)}(\D_1|.|^{u_1},\D_2|.|^{u_2} ,2s)$.
\end{prop}
\begin{proof}
If $t\geq 3$, then $L_{rad(ex)}(\pi_u,\chi_\alpha,s)=1$ as an immediate consequence of Proposition \ref{translation}. If $t=2$, and 
$L^{lin}_{rad(ex)}(\pi_u,\chi_\alpha,s)$ has a pole at $s_0$, because $u$ is in general position, Proposition \ref{translation} implies that 
$L_{rad(ex)}(|.|^{u_1}\D_1,|.|^{u_2}\D_2,2s)$ has a pole at $s_0$. Conversely, if $L_{rad(ex)}(|.|^{u_1}\D_1,|.|^{u_2}\D_2,2s)$ has a pole at $s_0$, 
then $|.|^{u_2}\D_2\simeq |.|^{-2s_0-u_1}\D_1^\vee$, i.e. 
$\D_2\simeq |.|^{e}\D_1^\vee$ with $e=-2s_0-u_1-u_2$, and $|.|^{s_0}\pi_u$ is $(H_{n},\chi_\alpha^{-1})$-distinguished according 
to Proposition \ref{distinduites}. 
But according to condition $8.$ of Definition \ref{generalpos}, this implies that the space
 $$Hom_{P_{n}^{\sigma}}(|.|^{s_0}\pi_u,\chi_\alpha^{-1})$$ is of dimension $\leq 1$. The representation 
$|.|^{u_2}\D_2\times |.|^{u_2}\D_2^\vee$ being generic, Proposition \ref{suffexceptionnel} implies that 
$L^{lin}_{rad(ex)}(|.|^{s_0}\pi_u,\chi_\alpha,s)$ has a pole at $0$, i.e. that $L^{lin}_{rad(ex)}(\pi_u,\chi_\alpha,s)$ has a pole at $s_0$. Hence 
both Euler factors $L^{lin}_{rad(ex)}(\pi_u,\chi_\alpha,s)$ and $L_{rad(ex)}(\D_1|.|^{u_1},\D_2|.|^{u_2} ,2s)$ have the same poles, which are simple according to Propositions \ref{radexlin} and Proposition \ref{radexpaires}, hence the two factors are equal.
\end{proof}

Proposition \ref{factL} has the following consequences.

\begin{prop}\label{division}
Let $u$ be in general position with respect to $(\pi,\chi_\alpha)$ be in general position. Let $(a_i)_{i=1\dots t}$ and
$(b_i)_{i=1\dots t}$ be two sequences of non negative integers, such that $b_i\geq a_i$, and $\D_i^{(a_i)}$ and $\D_i^{(b_i)}$ are nonzero for all 
$i$ in $\{1,\dots,t\}$. Then the Euler factor $L^{lin}(|.|^{u_1}\D_1^{(b_1)}\times \dots \times |.|^{u_t}\D_t^{(b_t)},\chi_\alpha,s)$ divides 
$L^{lin}(|.|^{u_1}\D_1^{(a_1)}\times \dots \times |.|^{u_t}\D_t^{(a_t)},\chi_\alpha,s)$.
\end{prop}
\begin{proof}
Let $\D_i'$ be equal to $|.|^{u_i}\D_i$ for all $i$ between $1$ and $t$. Let $\pi_a$ be the representation $\D_1'^{(a_1)}\times \dots \times \D_t'^{(a_t)}$, 
and $\pi_b$ be the representation 
$\D_1'^{(b_1)}\times \dots \times \D_t'^{(b_t)}$. The irreducible factors of the derivatives of $\pi_a$ are the nonzero representations 
of the form $\D_1'^{(c_1)}\times \dots \times \D_t'^{(c_t)}$, with $c_i\geq a_i$, and the irreducible factors of the derivatives of 
$\pi_b$ are the nonzero representations of the form $\D_1'^{(d_1)}\times \dots \times \D_t'^{(d_t)}$, with $d_i\geq b_i$. 
Hence the proposition is a consequence of Proposition \ref{factL}.
\end{proof}

Proposition \ref{radexgeneral}, \ref{factL} and \ref{division} imply immediatly.

\begin{prop}\label{factmieux}
Let $\pi_u=|.|^{u_1}\D_1\times \dots \times |.|^{u_t}\D_t$ be in general position. For $i$ in $\{1,\dots,t\}$, and $a_i$ be the 
smallest positive integer such that $\D_i^{(a_i)}\neq 0$, then $L^{lin}(\pi_u,\chi_\alpha,s)$ is equal to 
$$\vee_{i=1}^t L^{lin}(|.|^{u_1}\D_1\times \dots \times |.|^{u_{i-1}}\D_{i-1} \times |.|^{u_i}\D_i^{(a_i)} 
\times |.|^{u_{i+1}}\D_{i+1} \times \dots  \times |.|^{u_{t}}\D_t,\chi_\alpha,s),$$ if $t\geq 3$, and 
$L^{lin}(\pi_u,\chi_\alpha,s)$ is equal to the product of $L_{rad(ex)}(\D_1|.|^{u_1},\D_2|.|^{u_2} ,2s)$ with 
$$L^{lin}(|.|^{u_1}\D_1^{(a_1)}\times |.|^{u_{2}}\D_2,\chi_\alpha,s)\vee L^{lin}(|.|^{u_1}\D_1\times |.|^{u_{2}}\D_2^{(a_2)},\chi_\alpha,s)$$ 
if $t=2$.
\end{prop}

We finally obtain the following theorem for representations in general position.

\begin{thm}\label{Lgeneral}
Let $u$ be in general position with respect to $(\pi,\chi_\alpha)$. Then 
$$L^{lin}(\pi_u,\chi_\alpha,s)=\prod_{1\leq i<j \leq t}L(|.|^{u_i}\D_i, |.|^{u_j}\D_j,2s)\prod_{k=1}^t L^{lin}(|.|^{u_k}\D_k,\chi_\alpha,s).$$
\end{thm}
\begin{proof}
We prove this by induction on $n$. As $t\geq 2$, we have $n\geq 2$. For $n=2$, 
we are necessarily in the situation where $\pi_u=|.|^{u_1}\chi_1 \times |.|^{u_2}\chi_2$ for some characters $\chi_1$ and $\chi_2$ of $F^*$.
In this case, $L^{lin}_{rad(ex)}(\pi_u,\chi_\alpha,s)= L_{rad(ex)}(|.|^{u_1}\chi_1,|.|^{u_2}\chi_2,2s)=L(|.|^{u_1}\chi_1,|.|^{u_2}\chi_2,2s)$ according to Proposition \ref{radexgeneral}, and because $L(|.|^{u_1}\chi_1,|.|^{u_2}\chi_2,s)=L(|.|^{u_1+u_2}\chi_1\chi_2,s)$ has at most one simple pole, which 
is exceptional. Moreover, we have the equalities 
$$L^{lin}(|.|^{u_1}\chi_1^{(1)}\times |.|^{u_{2}}\chi_2,\chi_\alpha,s)= L^{lin}(|.|^{u_{2}}\chi_2,\chi_\alpha,s)$$ and 
$$L^{lin}(|.|^{u_1}\chi_1\times |.|^{u_{2}}\chi_2^{(1)},\chi_\alpha,s)= L^{lin}(|.|^{u_{1}}\chi_1,\chi_\alpha,s).$$ The result then follows from Proposition \ref{factmieux}.\\
If $t\geq 3$, according to Proposition \ref{factmieux}, one has 
$$L^{lin}(\pi_u,\chi_\alpha,s)$$ 
$$=\vee_{i=1}^t L^{lin}(|.|^{u_1}\D_1\times \dots \times |.|^{u_{i-1}}\D_{i-1} \times |.|^{u_i}\D_i^{(a_i)} \times 
|.|^{u_{i+1}}\D_{i+1} \times \dots  \times |.|^{u_{t}}\D_t,\chi_\alpha,s).$$ 
By induction hypothesis, the factor 
$$L^{lin}(|.|^{u_1}\D_1\times \dots \times |.|^{u_{i-1}}\D_{i-1} \times |.|^{u_i}\D_i^{(a_i)} \times |.|^{u_{i+1}}\D_{i+1} \times \dots  \times |.|^{u_{t}}\D_t,\chi_\alpha,\chi_\alpha,s)$$ is equal to the product of
$$\prod_{k<l,k\neq i, l\neq i} L(|.|^{u_k}\D_k,|.|^{u_l}\D_l,2s)\prod_{p\neq i} L^{lin}(|.|^{u_p}\D_p,\chi_\alpha,s)$$ and 
$$L^{lin}(|.|^{u_i}\D_i^{(a_i)},\chi_\alpha,s)\prod_{k'<i} L(|.|^{u_{k'}}\D_{k'},|.|^{u_i}\D_i^{(a_i)},2s)
\prod_{i<l'} L(|.|^{u_i}\D_i^{(a_i)},|.|^{u_{l'}}\D_{l'},2s).$$
This last product divides $$L^{lin}(|.|^{u_i}\D_i,\chi_\alpha,s)\prod_{k'<i} L(|.|^{u_{k'}}\D_{k'},|.|^{u_i}\D_i,2s)
\prod_{i<l'} L(|.|^{u_i}\D_i,|.|^{u_{l'}}\D_{l'},2s)$$ according to Propositions \ref{division} and \ref{divisionpaires}. 
Now, as $u$ is in general with respect to $(\pi,\chi_\alpha)$, according to Definition \ref{generalpos}, the factors 
$L(|.|^{u_i}\D_i,|.|^{u_j}\D_j,2s)$ and  $L(|.|^{u_k}\D_k,|.|^{u_l}\D_l,2s)$ are coprime when $\{i,j\}\neq \{k,l\}$, the factors 
$L(|.|^{u_i}\D_i,|.|^{u_j}\D_j,2s)$ and $L^{lin}(|.|^{u_k}\D_k,\chi_\alpha,s)$ are coprime 
when $i\neq j$, and the factors $L^{lin}(|.|^{u_k}\D_k,\chi_\alpha,s)$ and $L^{lin}(|.|^{u_l}\D_l,\chi_\alpha,s)$ as well when $k\neq l$. Hence when $t\geq 3$, the lcm 
 $$\vee_{i=1}^t L^{lin}(|.|^{u_1}\D_1\times \dots \times |.|^{u_{i-1}}\D_{i-1} \times 
|.|^{u_i}\D_i^{(a_i)} \times |.|^{u_{i+1}}\D_{i+1} \times \dots  \times |.|^{u_{t}}\D_t,\chi_\alpha,s)$$ is equal to 
 $$\prod_{1\leq i<j \leq t}L(|.|^{u_i}\D_i , |.|^{u_j}\D_j,2s)\prod_{k=1}^t L^{lin}(|.|^{u_k}\D_k,\chi_\alpha,s),$$ 
which is what we want.\\
If $t=2$, by induction, we have 
$$L^{lin}(|.|^{u_1}\D_1^{(a_1)}\times |.|^{u_{2}}\D_2,\chi_\alpha,s)$$ 
$$=L^{lin}(|.|^{u_1}\D_1^{(a_1)},\chi_\alpha,s)L^{lin}(|.|^{u_{2}}\D_2,\chi_\alpha,s)L(|.|^{u_1}\D_1^{(a_1)},|.|^{u_{2}}\D_2,2s)$$
and 
$$L^{lin}(|.|^{u_1}\D_1\times |.|^{u_{2}}\D_2^{(a_2)},\chi_\alpha,s)$$ 
$$=L^{lin}(|.|^{u_1}\D_1,\chi_\alpha,s)L^{lin}(|.|^{u_{2}}\D_2^{(a_2)},\chi_\alpha,s)L(|.|^{u_1}\D_1,|.|^{u_{2}}\D_2^{(a_2)},2s).$$
Again, as $u$ is in general position, and because of Propositions \ref{division} and \ref{divisionpaires}, 
the lcm of these two factors is the product of $L^{lin}(|.|^{u_1}\D_1,\chi_\alpha,s)L^{lin}(|.|^{u_{2}}\D_2,\chi_\alpha,s)$ with 
$$L(|.|^{u_1}\D_1^{(a_1)},|.|^{u_{2}}\D_2,2s)\vee 
L(|.|^{u_1}\D_1,|.|^{u_{2}}\D_2^{(a_2)},2s),$$ 
which is equal to 
$$L^{lin}(|.|^{u_1}\D_1,\chi_\alpha,s)L^{lin}(|.|^{u_{2}}\D_2,\chi_\alpha,s)L_{(0)}(|.|^{u_1}\D_1,|.|^{u_{2}}\D_2,2s)$$ 
according to Proposition \ref{factLpaires}. Finally we obtain that $L^{lin}(\pi_u,\chi_\alpha,s)$ is equal to 
$$L_{rad(ex)}(\D_1|.|^{u_1}\!\!\!,\D_2|.|^{u_2}\!\!\!,2s)
L^{lin}(|.|^{u_1}\D_1,\chi_\alpha,s)L^{lin}(|.|^{u_{2}}\D_2,\chi_\alpha,s)L_{(0)}(|.|^{u_1}\D_1,|.|^{u_{2}}\D_2,2s)$$ 
$$=L^{lin}(|.|^{u_1}\D_1,\chi_\alpha,s)L^{lin}(|.|^{u_{2}}\D_2,\chi_\alpha,s)L(|.|^{u_1}\D_1,|.|^{u_{2}}\D_2,2s)$$ 
according to Proposition \ref{factmieux}, which is what we wanted.
\end{proof}

\subsection{The inductivity relation for irreducible representations}\label{inductivity-irreducible}

We recall from \cite{Sil} that any irreducible representation $\tau$ of $G_n$, is the unique irreducible quotient of a unique representation $\pi$ 
of Langlands' type. By definition, we set $L^{lin}(\tau,s)=L^{lin}(\pi,s)$. 
We extend the relation $L^{lin}(\pi,\chi_\alpha,s)=\prod_{1\leq i<j \leq t}L(\D_i,\D_j,2s)\prod_{k=1}^t L^{lin}(\D_k,\chi_\alpha,s)$ 
to representations of Langlands' type, for $-1/2\leq Re(\alpha) \leq 0$.
We fix $\pi=\D_1\times \dots \times \D_t$ a representation of $G_n$ of Whittaker type for the 
beginning of this section, and take a character $\alpha$ of real part in $[-1/2,0]$. The representation $\pi$ will become of Langlands' type
 when we need it.
We first prove the following Lemma:

\begin{LM}\label{lmpolynomial}
Let $f$ belong to $V_{\pi}$, and $\alpha$ a character of $F^*$ with $Re(\alpha)\in [-1/2,0]$. The ratio 
$$\frac{\Psi(W_{f,u},\phi,\chi_\alpha,s)}{\prod_{1\leq i<j \leq t}L(|.|^{u_i}\D_i, |.|^{u_j}\D_j,2s)
\prod_{k=1}^t L^{lin}(|.|^{u_k}\D_k,\chi_\alpha,s)}$$ belongs to 
$\C[q^{\pm u}, q^{\pm s}]$.
\end{LM}
\begin{proof}
According to Theorem \ref{rationality}, the ratio in question is an element of $\C(q^{\pm u}, q^{\pm s})$. For fixed $u$ in general position with respect to $(\pi,\chi_\alpha)$ (i.e. outside of some hyperplanes $H_1,\dots,H_r$ of $\mathcal{D}^t$), it belongs to 
$\C[q^{\pm s}]$, in particular, it has no poles. Let $P$ be the denominator of $$\frac{\Psi(W_{f,u},\phi,\chi,s)}{\prod_{1\leq i<j \leq t}L(|.|^{u_i}\D_i\times |.|^{u_j}\D_j,2s)\prod_{k=1}^t L^{lin}(|.|^{u_k}\D_k,s)},$$ then the set $Z(P)$ of zeroes of $P$ is contained in the union 
$\bigcup_i H_i\times \mathcal{D}$. If $P$ was not a unit of $\C[q^{\pm u},q^{\pm s}]$, $Z(P)$ would be a hypersurface of $\mathcal{D}^{t+1}$, and one of its irreducible components would be 
$H_i\times \mathcal{D}$ for some $i$. This cannot be the case, as for any fixed $u$, for $s$ large enough, 
the fraction $\Psi(W_{f,u},\phi,\chi_\alpha,s)$ is defined by an absolutely convergent integral, hence is holomorphic, and the inverse of $\prod_{1\leq i<j \leq t}L(|.|^{u_i}\D_i\times |.|^{u_j}\D_j,2s)\prod_{k=1}^t L^{lin}(|.|^{u_k}\D_k,s)$ is polynomial, hence without poles. This concludes the proof.
\end{proof}

This has two corollaries, the first being straightforward.

\begin{cor}\label{divides}
Let $\alpha$ be a character of $F^*$ with $Re(\alpha)\in [-1/2,0]$, then $L^{lin}(\pi,\chi_\alpha,s)$ divides the factor 
$$\prod_{1\leq i<j \leq t}L(\D_i,\D_j,2s)\prod_{k=1}^t L^{lin}(\D_k,\chi_\alpha,s).$$
\end{cor}

 We introduce the following notation: if $R$ is a integral domain, with field of fractions $F(R)$, and $r$ and $r'$ belong to 
$F(R)$, we write $r \sim r'$ if $r$ and $r'$ are equal up to a unit of $R$.

\begin{prop}\label{egalgamma} Let $\alpha$ be a character of $F^*$ with $Re(\alpha)\in [-1/2,0]$, we have 
$$\gamma^{lin}(\pi_u,\chi_\alpha,\theta,s)\sim \prod_{1\leq i<j \leq t}
\gamma(|.|^{u_i}\D_i,|.|^{u_j}\D_j,\theta,2s)\prod_{k=1}^t \gamma^{lin}(|.|^{u_k}\D_k,\chi_\alpha,\theta,s)$$ in $\C(\mathcal{D}^{t+1})$, 
hence 
$$\gamma^{lin}(\pi,\chi_\alpha,\theta,s)\sim \prod_{1\leq i<j \leq t}
\gamma(\D_i,\D_j,\theta,2s)\prod_{k=1}^t \gamma^{lin}(\D_k,\chi_\alpha,\theta,s)$$ in $\C(\mathcal{D})$. 
\end{prop}
\begin{proof}
Let $\e_0^{lin}(\pi_u,\chi_\alpha,\theta,s)$ be the element of $\C(\mathcal{D}^{t+1})$ defined by the relation 
$$\gamma^{lin}(\pi_u,\chi_\alpha,\theta,s)/\e_0^{lin}(\pi_u,\chi_\alpha,\theta,s)$$ 
$$=\frac{\prod_{1\leq i<j \leq t}
L(|.|^{-u_j}\D_j^\vee,|.|^{-u_i}\D_i^\vee,1-2s)\prod_{k=1}^t L^{lin}(|.|^{-u_k}\D_k^\vee,\d_n^{-1/2}\chi_\alpha^{-1},1/2-s)}{\prod_{1\leq i<j \leq t}L(|.|^{u_i}\D_i,|.|^{u_j}\D_j,2s)\prod_{k=1}^t L^{lin}(|.|^{u_k}\D_k,\chi_\alpha,s)}.$$ 
This can be rewritten, for any $f$ in $\mathcal{F}_\pi$, as the equality: 
$$\frac{\Psi(\widetilde{W_{f,u}},\hat{\phi}^\theta,\d_n^{-1/2}\chi_\alpha^{-1},1/2-s)}{\prod_{k=1}^t L^{lin}(|.|^{-u_k}\D_k^\vee,\d_n^{-1/2}\chi_\alpha^{-1},1/2-s)\prod_{1\leq i<j \leq t}L(|.|^{-u_j}\D_j^\vee,|.|^{-u_i}\D_i^\vee,1-2s)}$$
\begin{equation}\label{nine}=\frac{\e_0^{lin}(\pi_u,\chi_\alpha,\theta,s)\Psi(W_{f,u},\phi,\chi_\alpha,s)}{\prod_{k=1}^t L^{lin}(|.|^{u_k}\D_k,\chi_\alpha,s)\prod_{1\leq i<j \leq t}L(|.|^{u_i}\D_i,|.|^{u_j}\D_j,2s)}.\end{equation}
If $u$ is in general position with respect to $(\pi,\chi_\alpha)$, then the ratio $\e_0^{lin}(\pi_u,\chi_\alpha,\theta,s)$ is equal to $\e^{lin}(\pi_u,\chi_\alpha,\theta,s)$, in particular $\e_0^{lin}(\pi_u,\chi_\alpha,\theta,s)$ has no poles for $u$ in general position. First, we notice that this implies that $\e_0^{lin}(\pi_u,\chi_\alpha,\theta,s)$ actually belongs to $\C[\mathcal{D}^{t+1}]$. Indeed, let $P$ be its denominator, reasonning as in proof of Lemma \ref{lmpolynomial}, if $P$ is not a unit of $\C[\mathcal{D}^{t+1}]$, the set $Z(P)$ necessarily contains a subset of the form $H\times\mathcal{D}$, for 
$H$ an affine hyperplane of $\mathcal{D}^{t}$. But Equation (\ref{nine}) implies that $Z(P)$, and hence $H\times\mathcal{D}$, is a subset of 
the zeroes of the polynomial $$\frac{\Psi(W_{f,u},\phi,\chi_\alpha,s)}{\prod_{k=1}^t L^{lin}(|.|^{u_k}\D_k,\chi_\alpha,s)\prod_{1\leq i<j \leq t}L(|.|^{u_i}\D_i,|.|^{u_j}\D_j,2s)}$$ for any $f$ and $\phi$. In particular, according to Proposition \ref{incl} and its corollary, for any $P_0$ in 
$\mathcal{P}_0$, the hyperplane $H\times\mathcal{D}$ is a subset of the zeores of $$\frac{P_0(q^{\pm u})}{\prod_{k=1}^t L^{lin}(|.|^{u_k}\D_k,\chi_\alpha,s)\prod_{1\leq i<j \leq t}L(|.|^{u_i}\D_i,|.|^{u_j}\D_j,2s)}.$$
 But by definition of $\mathcal{P}_0$, and 
thanks to Remark \ref{Kir}, for any fixed $u$, there is $P_0$ in $\mathcal{P}_0$ such that $P_0(q^{\pm u})=1$. For fixed $u$, this would implie 
that $H\times \mathcal{D}$ is a subset of the zeroes of $$\frac{1}{\prod_{k=1}^t L^{lin}(|.|^{u_k}\D_k,\chi_\alpha,s)\prod_{1\leq i<j \leq t}L(|.|^{u_i}\D_i,|.|^{u_j}\D_j,2s)},$$ which is impossible as 
for any fixed $u$, this polynomial is nonzero as a polynomial of $q^{-s}$. This thus implies that $P$ is a unit, and 
that $\e_0^{lin}(\pi_u,\chi_\alpha,\theta,s)$ belongs to $\C[\mathcal{D}^{t+1}]$.\\
Now, if $u$ is in general position with respect to $(\pi,\chi_\alpha)$ and $(\tilde{\pi},\d_n^{-1/2}\chi_\alpha^{-1})$, we have:
$$\e_0^{lin}(\pi_u,\chi_\alpha,\theta,s)=\e^{lin}(\pi_u,\chi_\alpha,\theta,s)$$ and 
$$\e_0^{lin}(\widetilde{\pi_u},\d_n^{-1/2}\chi_\alpha^{-1},\theta^{-1},1/2-s)=\e^{lin}(\widetilde{\pi_u},\d_n^{-1/2}\chi_\alpha^{-1},\theta,1/2-s).$$
Applying twice the functional equation, we obtain that $$\e^{lin}(\pi_u,\chi_\alpha,\theta,s)\e^{lin}(\widetilde{\pi_u},\d_n^{-1/2}\chi_\alpha^{-1},\theta^{-1},1/2-s)$$ is equal to 
the constant function $1$ or $-1$ (independantly of $u$ and $s$). Hence $$\e_0^{lin}(\pi_u,\chi_\alpha,\theta,s)\e_0^{lin}(\widetilde{\pi_u},\d_n^{-1/2}\chi_\alpha^{-1},\theta^{-1},1/2-s)$$ equals $1$ or $-1$, for $u$ outside 
a finite number of hyperplanes, so it is always valid as an equality of polynomials. In particular, $\e_0^{lin}(\pi_u,\chi_\alpha,\theta,s)$ 
is a unit of $\C[\mathcal{D}^{t+1}]$, of the form $\lambda q^{\sum_i k_iu_i+as}$ for integers $k_1,\dots,k_t,a$, and $\lambda$ in $\C^*$. 
We conclude thanks to the relation 
$$\frac{\gamma^{lin}(\pi_u,\chi_\alpha,\theta,s)}{\prod_{1\leq i<j \leq t}
\gamma(|.|^{u_i}\D_i,|.|^{u_j}\D_j,\theta,2s)\prod_{k=1}^t \gamma^{lin}(|.|^{u_k}\D_k,\chi_\alpha,\theta,s)}$$ 
$$=\frac{\e_0^{lin}(\pi_u,\chi_\alpha,\theta,s)}{\prod_{1\leq i<j \leq t}
\e(|.|^{u_i}\D_i,|.|^{u_j}\D_j,\theta,2s)\prod_{k=1}^t \e^{lin}(|.|^{u_k}\D_k,\chi_\alpha,\theta,s)}.$$ 
\end{proof}

\begin{LM}\label{divisioninduite}
Let $\pi_1$ and $\pi_2$ be two representations of Whittaker type of $G_{n_1}$ and $G_{n_2}$ respectively, and $\alpha$ be a character of $F^*$, then the 
factor $L^{lin}(\pi_2,\chi_\alpha,s)$ divides the factor $L^{lin}(\pi_1\times \pi_2,\chi_\alpha,s)$.
\end{LM}
\begin{proof}
Let $\pi=\pi_1\times \pi_2$. We saw in the proof of Proposition \ref{factL}, that for $W$ in $W(\pi,\theta)$, the integral 
$$\Psi_{(n_1-1)}(W,\chi_\alpha,\theta)=\int_{N_{n_2}^\sigma\backslash H_{n_2}} 
W(diag(h,I_{n_1}))\chi_\alpha\chi_{n_2}(h)|h|^{s-n_1/2}dh$$ belongs to the ideal 
$I(\pi,\chi_\alpha)$. But according to Proposition 9.1. of \cite{JPS}, for every $(W_2,\phi)$ in $W(\pi_2,\theta)\times \sm_c(F^{n_2})$, there is 
$W$ in $W(\pi,\theta)$, such that for $h$ in $G_{n_2}$, we have the equality $W(diag(h,I_{n_1}))=W_2(h)\phi(\eta_{n_2}h)|h|^{n_1/2}$. For this $W$, the integral 
$\Psi_{(n_1-1)}(W,\chi_\alpha,s)$ becomes $\Psi(W_2,\phi',\chi_\alpha,s)$, for 
$$\phi'(t_1,t_2,\dots,t_{[(n_2+1)/2]})=\phi(t_1,0,t_2,0,\dots).$$ 
As every $\phi'$ in $\sm_c(F^{[(n_2+1)/2]})$ can be obtained this way, this proves that 
$I(\pi_2,\chi_\alpha)\subset I(\pi,\chi_\alpha)$, which is what we wanted.
\end{proof}

For representations of Langlands' type, the preceding results implies the relation we are seeking for. If $\pi=\D_1\times \dots \times \D_t$ is a representation of Whittaker type of $G_n$, we will denote by $L_1^{lin}(\pi,\chi_\alpha,s)$ the factor 
$$L_1^{lin}(\pi,\chi_\alpha,s)= \prod_{k=1}^t L^{lin}(\D_k,\chi_\alpha,s)\prod_{1\leq i<j \leq t}L(\D_i,\D_j,2s).$$

\begin{thm}\label{indLanglands }
If $\pi$ is a representation of $G_n$ of Langlands' type, and $\alpha$ is a character of $F^*$ with $-1/2\leq Re(\alpha)\leq 0$. We have $L^{lin}(\pi,\chi_\alpha,s)=L_1^{lin}(\pi,\chi_\alpha,s)$.
\end{thm}
\begin{proof}
We write $\pi=\D_1\times \dots \times \D_t$, and we denote by $r_i$ the the central character of $\D_i$, so that $r_i\geq r_{i+1}$. It is a consequence of 
the corollary of Theorem 2.3 of \cite{JPS} that $L(\D_i,\D_j,2s)$ has its poles contained in the region $\{Re(s)\leq (-r_i-r_j)/2\}$, and of Corollary 
\ref{CVdiscr}, that $L^{lin}(\D_k,\chi_\alpha,s)$ has its poles contained in the region $\{Re(s)\leq -r_k\}$ (because $-1/2\leq Re(\alpha)$). Similarly, 
$L(\D_j^\vee,\D_i^\vee,1-2s)$ has its poles contained in the region $\{Re(s)\geq (1-r_i-r_j)/2\}$, and $L^{lin}(\D_k^\vee,\chi_\alpha^{-1}\d_n^{-1/2},1/2-s)$ has its poles contained in the region $\{Re(s)\geq 1/2-r_k\}$ (because $-1/2\leq Re(\alpha^{-1}|.|^{-1/2})$).\\
We now prove by induction on $t$ the equality $L^{lin}(\pi,\chi_\alpha,s)=L_1^{lin}(\pi,\chi_\alpha,s)$. There is nothing to prove 
when $t=1$. We suppose that it is true for $t-1$, and prove it for $t$.\\
We write $\pi_2=\D_2\times \dots \times \D_t$ and $\pi_3=\D_1\times \dots \times \D_{t-1}$. By Corollary \ref{divides}, there are two 
elements $P$ and $\tilde{P}$ of $\C[X]$ with constant term $1$, 
such that $$L_1^{lin}(\pi,\chi_\alpha,s)=P(q^{-s})^{-1}L^{lin}(\pi,\chi_\alpha,s)$$ and 
$$L_1^{lin}(\tilde{\pi},\chi_\alpha^{-1}\d_n^{-1/2},1/2-s)=\tilde{P}(q^{-s})^{-1}L^{lin}(\tilde{\pi},\chi_\alpha^{-1}\d_n^{-1/2},1/2-s).$$ 
According to Proposition 
\ref{egalgamma}, we moreover know that $\tilde{P}\sim P$, hence it will be sufficient to know that the set of zeroes of 
$P$ and $\tilde{P}$ are disjoint to conclude.  
 By Lemma \ref{divisioninduite}, there are two elements $Q$ and $\tilde{Q}$ of $\C[X]$ with constant term $1$, such that 
 $$L^{lin}(\pi,\chi_\alpha,s)=Q(q^{-s})^{-1}L^{lin}(\pi_2,\chi_\alpha,s)$$ and 
 $$L^{lin}(\tilde{\pi},\chi_\alpha^{-1}\d_n^{-1/2},1/2-s)=\tilde{Q}(q^{-s})^{-1}L^{lin}(\tilde{\pi_3},\chi_\alpha^{-1}\d_n^{-1/2},1/2-s).$$ By induction hypothesis, we know that $$L^{lin}(\pi_2,\chi_\alpha,s)=L_1^{lin}(\pi_2,\chi_\alpha,s)$$ and $$L^{lin}(\tilde{\pi_3},\chi_\alpha^{-1}\d_n^{-1/2},s)=L_1^{lin}(\tilde{\pi_3},\chi_\alpha^{-1}\d_n^{-1/2},s).$$
We thus have the relation $P(q^{-s})Q(q^{-s})=L_1^{lin}(\pi_2,\chi_\alpha,s)/L_1^{lin}(\pi,\chi_\alpha,s)$, i.e. 
$$P(q^{-s})Q(q^{-s})=L^{lin}(\D_1,s)^{-1}\prod_{j\geq 2} L(\D_1,\D_j,2s)^{-1}.$$
This implies that 
$P(q^{-s})$ has its zeroes in the region $$\{Re(s)\leq (-u_1-u_t)/2\}.$$
Similarly, 
$$\tilde{P}(q^{-s})\tilde{Q}(q^{-s})=L^{lin}(\D_t^\vee,\chi_\alpha^{-1}\d_n^{-1/2},1/2-s)^{-1}\prod_{i\leq t-1} L(\D_{t}^\vee,\D_i^\vee,1-2s)^{-1}$$
implies that $\tilde{P}(q^{-s})$ has its zeroes in the region $$\{Re(s)\geq (1-u_1-u_t)/2\}.$$
This ends the proof.
\end{proof}

\subsection{The equality for discrete series}
Let $\phi=\phi_1\oplus \dots \oplus \phi_t$ be a semi-simple representation of the Weil-Deligne group $W'_F$, with $\phi_i$ of dimension 
$n_i$, and $\phi$ of dimension $n=n_1+\dots+n_t$. Let $\D_i$ be the discrete series representation of $G_{n_i}$ such that 
$\phi_i=\phi(\D_i)$, we can always assume, up to re-ordering the $\phi_i$'s (which doesn't change $\phi$), that 
$\D_1\times \dots \times \D_t$ is of Langlands' type. In this case, if $\tau$ is the unique irreducible quotient of 
$\pi=\D_1\times \dots \times \D_t$, then $\phi(\tau)=\phi$. Now, thanks to the equality 
$$\wedge^2(\phi_1\oplus \dots \oplus \phi_t)= \oplus_{k=1}^t \wedge^2(\phi_k) \oplus_{1\leq i<j \leq t} \phi_i\otimes \phi_j,$$
if we denote by $L^{lin}(\phi,s)$ the factor $L(\phi,s+1/2)L(\phi,\wedge^2,2s)$, we have the equality 
$$L^{lin}(\phi,s)=\prod_{k=1}^t L^{lin}(\phi_k,s) \prod_{1\leq i<j\leq t}L(\phi_i\otimes \phi_j,2s).$$
As it is a consequence of Langlands' correspondance for $G_n$ that 
$L(\phi_i\otimes \phi_j,s)=L(\D_i\times \D_j,s)$, to prove the equality $L^{lin}(\phi(\tau),s)=L^{lin}(\phi,s)$ for any irreducible representation 
$\tau$ of $G_n$, it is enough, according to Section \ref{inductivity-irreducible}, to show this equality for discrete series representations. 
We will be concerned with the proof of this result in this last section.\\ 

 We start with the following corollary of Proposition \ref{factL}.

\begin{cor}\label{factLdiscr}
 Let $\D$ be a discrete series of $G_n$, and $r$ the minimal positive integer, such that $\D^{(r)}\neq 0$, 
 then one has $L^{lin}_{(0)}(\D,s)=L^{lin}(\D^{(r)},s)$.
\end{cor}
\begin{proof} Indeed, we recall that all the nonzero derivatives of $\D$ are irreducible. More precisely, if we write 
$\D=St_k(\rho)= [|.|^{(1-k)/2}\rho,\dots,|.|^{(k-1)/2}\rho]$, they are of the form 
$\D^{(lr)}= [|.|^{(2l+1-k)/2}\rho,\dots,|.|^{(k-1)/2}\rho]=|.|^{l/2}St_{k-l}(\rho)$ for $l$ between $1$ and $k$, now the claim is a 
straight-forward consequence of Proposition \ref{factL} and a simple induction.\end{proof}

We recall the following standard resut about $L$ functions of special representations.

\begin{prop}\label{Ldiscr}
Let $k$ and $r$ be two positive integers, let $\rho$ be a cuspidal representation of $G_r$, and let $\D=St_k(\rho)$. Then 
$L(\phi(\D),s)=L(\phi(\rho|.|^{(k-1)/2}),s)$, in particular $L(\phi(\D),s)=L(\phi(\D^{(r)}),s)$.
\end{prop}

We will need Theorem 3.1 and results discussed in section 5 and 6 of \cite{M2012.2} (namely Theorem 4.3 of \cite{K}, Theorem 1.1 of \cite{KR}, 
Proposition 4.3 of \cite{S}), which we summarise here.

\begin{thm}\label{discr}
Let $\D$ be a unitary discrete series representation of $G_n$, it cannot be $H_n$-distinguished when $n$ is odd, and when $n$ is even, 
it is $H_n$-distinguished if and only if it has a Shalika model. In the case $n$ even, we write 
$\D$ as $St_k(\rho)$, where $\rho$ is a cuspidal representation of $G_r$, with $n=kr$. When $k$ is odd, $\D$ is distinguished if and only if 
$L(\phi(\rho),\wedge^2,s)$ has a pole at zero, and  when $k$ is even, $\D$ is distinguished if and only if 
$L(\phi(\rho),Sym^2,s)$ has a pole at zero.
\end{thm}

We will use the following lemma.

\begin{LM}\label{simple}
Let $\rho$ be a cuspidal representation of $G_r$, then $L(\phi(\rho),Sym^2,s)$ and $L(\phi(\rho),\wedge^2,s)$ both have simple poles. The function $L(\phi(\rho),\wedge^2,s)$ is equal to $1$ when $r$ is odd.
\end{LM}
\begin{proof}
In this case, $\phi(\rho)$ is an irreducible representation of the Weil group $W_F$ of dimension $r$. Let $Fr$ be a geometric Frobenius element of 
$Gal_F(\bar{F})$. By definition, the factor 
$L(\phi(\rho)\otimes \phi(\rho),s)$ equals $1/det([I_d- q^{-s}\phi(\rho)(Fr)\otimes \phi(\rho)(Fr)]_{|V^{I_F}})$, where $V^{I_F}$ is the subspace of 
$\phi(\rho)\otimes \phi(\rho)$ fixed by the inertia subgroup $I_F$ of $W_F$. Moreover, 
we have $$L(\phi(\rho)\otimes \phi(\rho),s)=L(\phi(\rho),\wedge^2,s)L(\phi(\rho),Sym^2,s),$$ hence we just need to show that 
$L(\phi(\rho)\otimes \phi(\rho),s)$ has simple poles. To say that $s_0$ is a pole of order of $d$ of this factor, amounts to say that $\1_{W_F}$ occurs with multiplicity $d$ in $|.|^{s_0}\phi(\rho)\otimes \phi(\rho)$, i.e. that the space 
$Hom_{W_F}(\rho^\vee,|.|^{s_0}\rho)$ is of dimension $d$. As $\rho$ is irreducible, this dimension is necessarily $\leq 1$. Finally, the complex number 
$s_0$ is a pole of $L(\phi(\rho),\wedge^2,s)$ if and only if $\wedge^2(|.|^{s_0/2}\phi(\rho))$ contains $\1_{W_F}$, so its contragredient as well, hence 
there is a (necessarily by irreducibility of $\phi(\rho)$) non degenerate symplectic form on the space of $|.|^{s_0/2}\phi(\rho)$, and $r$ must be even in this case. When $r$ is odd, $L(\phi(\rho),\wedge^2,s)$ is thus necessarily $1$.
\end{proof}

Now Theorem \ref{discr} has the following consequence.

\begin{prop}\label{radexdiscr}
Let $k$ and $r$ be two positive integers, let $\rho$ be a cuspidal representation of $G_r$, and let $\D=St_k(\rho)$ be 
a discrete series of $G_n$, for $n=kr$.\\
If $k$ is odd, we have $L^{lin}_{rad(ex)}(\D,s)=L(\phi(\rho),\wedge^2,2s)$.\\
If $k$ is even, we have $L^{lin}_{rad(ex)}(\D,s)=L(\phi(\rho),Sym^2,2s)$.
\end{prop}
\begin{proof}
We recall from Corollary \ref{unitex}, that when $n$ is even, the factor $L^{lin}_{rad(ex)}(\D,s)$ has a pole at $s=s_0$ if and only if $|.|^{s_0}\D$ is $H_n$-distinguished. When $k$ is odd, and $r$ is odd, applying Theorem \ref{discr}, we have $L^{lin}_{rad(ex)}(\D,s)=1$, but the factor 
$L(\phi(\rho),\wedge^2,s)$ equals $1$ as well according to Lemma \ref{simple}. When $k$ is odd, and $r$ is even, applying Theorem \ref{discr}, we see that $|.|^{s_0}\D$ is $H_n$-distinguished if and only if 
$L(|.|^{s_0}\phi(\rho),\wedge^2,s)$ has a pole at $s=0$. Now because of the relation 
$$L(|.|^{s_0}\phi(\rho),\wedge^2,s)=L(\phi(\rho),\wedge^2,s+2s_0),$$ we deduce that $L_{rad(ex)}(\D,s)$ has a pole at $s_0$
 if and only if $L(\phi(\rho),\wedge^2,2s)$ has a pole at $s_0$. But from Proposition \ref{radexlin} and Lemma \ref{simple}, the factors 
$L_{rad(ex)}(\D,s)$ and $L(\phi(\rho),\wedge^2,2s)$ have simple poles, hence are equal. The case $k$ even is similar.
\end{proof}

\begin{prop}\label{rec}
 Let $\D$ be a discrete series representation of $G_n$, and let $\rho$ be the cuspidal representation of $G_r$, such that 
 $\D=St_k(\rho)$, with $n=kr$.\\
If $k$ is odd, then one has $L(\phi(\D),\wedge^2,s)=L(\phi(\rho),\wedge^2,s) L(\phi(\D^{(r)}),\wedge^2,s)$.\\
If $k$ is even, then one has $L(\phi(\D),\wedge^2,s)=L(\phi(\rho),Sym^2,s) L(\phi(\D^{(r)}),\wedge^2,s)$.
\end{prop}
\begin{proof}
 We recall the equality 
$$L(\phi(\D),\wedge^2,s)=\prod_{i=0}^{[(k-1)/2]}L(\phi(\rho),\wedge^2,s+k-2i-1)\prod_{j=0}^{[k/2-1]}L(\phi(\rho),Sym^2,s+k-2j-2)$$ from 
Proposition 6.2 of \cite{M2012.2}. If $k=2m+1$, then we have
\begin{eqnarray}\label{step}L(\phi(\D),\wedge^2,s)\end{eqnarray}
$$=\prod_{i=0}^{m}L(\phi(\rho),\wedge^2,s+2(m-i))\prod_{j=0}^{m-1}L(\phi(\rho),Sym^2,s+2(m-j)-1).$$
Then, $\D^{(r)}=|.|^{1/2}St_{2m}(\rho)$, hence we have:
\begin{eqnarray}\label{stepp} L(\phi(\D^{(r)}),\wedge^2,s)\end{eqnarray}
$$=\prod_{i=0}^{m-1}L(\phi(|.|^{1/2}\rho),\wedge^2,s+2(m-i)-1)\prod_{j=0}^{m-1}L(\phi(|.|^{1/2}\rho),Sym^2,s+2(m-j)-2)$$
$$=\prod_{i=0}^{m-1}L(\phi(\rho),\wedge^2,s+2(m-i))\prod_{j=0}^{m-1}L(\phi(\rho),Sym^2,s+2(m-j)-1).$$
Comparing Equations (\ref{step}) and (\ref{stepp}), we obtain the wanted equality. The case $k$ even is similar.
\end{proof}

Now we can prove.

\begin{thm}\label{egaldiscr}
Let $\D$ be a discrete series representation of $G_n$, then 
$$L^{lin}(\D,s)= L(\phi(\D),s+1/2)L(\phi(\D),\wedge^2,2s).$$
\end{thm}
\begin{proof}
We are going to prove the theorem by induction on $k$. We start with $k=1$, hence $\D=\rho$. 
In this case, there are two different cases, if $r=1$, then $\rho$ is a character $\chi$ of $F^*$. The integrals $\Psi(W,\phi,s)$ are just integrals 
of the form $\int_{F^*}\phi(t)|t|^{s+1/2}d^*t$ with $\phi$ in $\sm_c(F)$, hence $L^{lin}(\chi,s)=L(\chi,s+1/2)=L(\chi,s+1/2)L(\chi,\wedge^2,2s)$.\\
When $r\geq 2$, the $L$-factor $L_{(0)}(\rho,s)$ is equal to $1$, because as the elements of $W(\rho,\theta)$ restricted to $P_n$, 
have compact support mod $N_n$, the integrals $\Psi_{(0)}(W,s)$ which define $L_{(0)}(\rho,s)$ have no poles. 
Hence we have \begin{eqnarray}\label{step'}L^{lin}(\rho,s)=  L^{lin}_{rad(ex)}(\rho,s)= L(\phi(\rho),s+1/2)L^{lin}_{rad(ex)}(\rho,s)\end{eqnarray}
 because $L(\phi(\rho),s+1/2)$ is then equal to $1$.
 But according to Proposition \ref{radexdiscr}, the factor $L^{lin}_{rad(ex)}(\rho,s)$ is equal to $L(\phi(\rho),\wedge^2,2s)$, 
 and this proves the wanted equality.\\
Now we do the induction step. Suppose that the theorem is true for $k-1$. We do the case $k$ odd, the case $k$ even being similar.\\
Thanks to Corollary \ref{factLdiscr}, we have $L_{(0)}^{lin}(\D,s)= L^{lin}(\D^{(r)},s)$ which is thus equal to 
$$L(\phi(\D^{(r)}),s+1/2)L(\phi(\D^{(r)}),\wedge^2,2s)$$ by induction hypothesis. Proposition \ref{Ldiscr} then gives the equality 
$$L_{(0)}^{lin}(\D,s)=L(\phi(\D),s+1/2)L(\phi(\D^{(r)}),\wedge^2,2s),$$ so that 
$$L^{lin}(\D,s)=L_{rad(ex)}(\D,s)L_{(0)}^{lin}(\D,s)=L(\phi(\rho),\wedge^2,2s) L(\phi(\D),s+1/2)L(\phi(\D^{(r)}),\wedge^2,2s)$$ 
thanks to Proposition \ref{radexdiscr}. Applying Proposition \ref{rec}, we can replace the product of factors 
$L(\phi(\rho),\wedge^2,2s) L(\phi(\D^{(r)}),\wedge^2,2s)$ 
by the factor $L(\phi(\D),\wedge^2,2s)$, to finally obtain $$L^{lin}(\D,s)=L(\phi(\D),s+1/2)L(\phi(\D),\wedge^2,2s).$$
\end{proof}

As explained in the beginning of this last section, this has the following corollary.

\begin{cor}\label{egalfinal}
Let $\tau$ be an irreducible representation of $G_n$, then 
$$L^{lin}(\tau,s)=L(\phi(\tau),s+1/2)L(\phi(\tau),2s).$$
\end{cor}

I would like to thank N. Cheurfa for fruitful conversations.

\end{document}